\let\vecc\vec

\RequirePackage{fix-cm}
\documentclass[nospthms]{svjour3}                     %
\smartqed  %
\usepackage{graphicx}
\usepackage{subcaption}
\usepackage{amsmath,amssymb}
\usepackage{amsthm}
\usepackage{graphicx,color}
\usepackage[hyphens]{url}
\usepackage{dsfont}
\usepackage{booktabs}
\usepackage[square,sort,numbers]{natbib}
\usepackage[sort,nocompress,noadjust]{cite}
\usepackage[normalem]{ulem}
\usepackage{mathtools}
\usepackage{multirow}
\usepackage{algorithm}
\usepackage[noend]{algpseudocode}
\usepackage{xspace}
\usepackage{tikz}

\usepackage{tikz-cd}

\usepackage[font={small}]{caption}

\numberwithin{equation}{section}
\newtheorem{theorem}{Theorem}[section]
\newtheorem{cor}[theorem]{Corollary}
\newtheorem{lemma}[theorem]{Lemma}
\newtheorem{claim}[theorem]{Claim}
\newtheorem{remark}[theorem]{Remark}
\newtheorem{prop}[theorem]{Proposition}
\newtheorem{obs}[theorem]{Observation}
\newtheorem{problem}[theorem]{Problem}
\newtheorem{example}[theorem]{Example}

\newtheorem{defin}[theorem]{Definition}

\newtheorem{assump}[theorem]{Assumption}

\newcommand{\cM}{\mathcal{M}}

\newcommand{\cS}{\mathcal{S}}

\newcommand{\R}{\mathbb{R}}

\newcommand{\N}{\mathbb{N}}

\newcommand{\Rn}{\R^n}

\newcommand{\Rp}{\R_{\geq 0}}

\newcommand{\Rpn}{\R_{\geq 0}^n}

\newcommand{\Rntk}{(\R^n)^{\otimes k}}
\newcommand{\Rpntk}{(\R_{\geq 0}^n)^{\otimes k}}

\newcommand{\Prob}{\mathbb{P}}	
\newcommand*{\E}{\mathbb{E}}

\newcommand{\bone}{\mathbf{1}}

\newcommand*{\poly}{\mathrm{poly}}
\newcommand*{\polylog}{\mathrm{polylog}}

\newcommand*{\eps}{\varepsilon}
\newcommand*{\epst}{\tilde{\eps}}

\newcommand{\plusminus}{\raisebox{.2ex}{$\scriptstyle\pm$}}
\DeclareMathOperator*{\argmin}{argmin}
\DeclareMathOperator*{\argmax}{argmax}
\renewcommand{\leq}{\leqslant}
\renewcommand{\geq}{\geqslant}
\DeclareMathOperator{\half}{\frac{1}{2}}

\providecommand{\abs}[1]{\left\lvert#1\right\rvert}

\newcommand{\OT}{\textsf{OT}}
\newcommand{\MOT}{\textsf{MOT}}
\newcommand{\RMOT}{\textsf{RMOT}}

\newcommand{\MinO}{\textsf{MIN}}
\newcommand{\AMinO}{\textsf{AMIN}}
\newcommand{\SMinO}{\textsf{SMIN}}

\newcommand{\MinOCS}{\textsf{MIN}_{C,S}}
\newcommand{\AMinOCS}{\textsf{AMIN}_{C,S}}
\newcommand{\SMinOCS}{\textsf{SMIN}_{C,S}}

\newcommand{\ArgMinO}{\textsf{ARGMIN}}
\newcommand{\ArgAMinO}{\textsf{ARGAMIN}}
\newcommand{\MARG}{\textsf{MARG}}

\newcommand{\MWUoracle}{\textsf{MWU\_BOTTLENECK}}

\newcommand{\Sink}{\texttt{SINKHORN}}
\newcommand{\MWU}{\texttt{MWU}}
\newcommand{\ELLIP}{\texttt{ELLIPSOID}}
\newcommand{\COLGEN}{\texttt{COLGEN}}

\DeclareMathOperator*{\smin}{smin}
\DeclareMathOperator*{\smineta}{{\textstyle{smin_\eta}}}
\DeclareMathOperator*{\smax}{smax}

\newcommand{\Cmax}{C_{\max}}
\newcommand{\Rmax}{R_{\max}}

\newcommand{\Coup}{\cM(\mu_1,\dots,\mu_k)}

\newcommand{\jvec}{\vec{j}}

\DeclareMathOperator{\Ber}{Ber}

\newcommand{\pd}[2]{\frac{\partial #1}{\partial #2}} 

\let\baraccent=\= %
\renewcommand{\=}[1]{\stackrel{#1}{=}} %

\providecommand{\RR}{\mathbb{R}}

\providecommand{\cM}{\mathcal{M}}

\providecommand{\eps}{\varepsilon}

\providecommand{\N}{\mathbb{N}}

\renewcommand{\P}{\mathsf{P}}
\providecommand{\NP}{\mathsf{NP}}

\mathchardef\mhyphen="2D %

\newcommand{\interior}[1]{%
	{\kern0pt#1}^{\mathrm{o}}%
}

\usepackage{fancybox}
\newenvironment{fminipage}{\begin{Sbox}\begin{minipage}}{\end{minipage}\end{Sbox}\fbox{\TheSbox}}

\newcommand{\ppind}{\rho^{\mathrm{ind}}}
\newcommand{\ppmin}{\rho^{\min}}
\newcommand{\ppmax}{\rho^{\max}}

\newcommand{\Rbar}{\bar{\mathbb{R}}}
\let\vec\vecc
\usepackage[margin=1in]{geometry}
	
\makeatletter
\def\blfootnote{\gdef\@thefnmark{}\@footnotetext}
\makeatother

\begin{document}

	\title{Polynomial-time algorithms for Multimarginal Optimal Transport problems with structure\thanks{Work partially supported by NSF GRFP 1122374, a TwoSigma PhD fellowship, and a Siebel Scholarship.}%
}

\institute{J.M. Altschuler \and E. Boix-Adser\`a
\at
              Laboratory for Information and Decision Systems (LIDS), Massachusetts Institute of Technology, Cambridge MA 02139.  \\
              \email{jasonalt@mit.edu}           \\
			   \email{eboix@mit.edu} %
}

	\author{Jason M. Altschuler \and Enric Boix-Adser\`a}
	\date{}
	\maketitle

\maketitle

\begin{abstract}
	Multimarginal Optimal Transport (MOT) has attracted significant interest due to applications in machine learning, statistics, and the sciences. However, in most applications, the success of MOT is severely limited by a lack of efficient algorithms. Indeed, MOT in general requires exponential time in the number of marginals $k$ and their support sizes $n$. This paper develops a general theory about what ``structure'' makes MOT solvable in $\poly(n,k)$ time. 
	
	We develop a unified algorithmic framework for solving MOT in $\poly(n,k)$ time by characterizing the structure that different algorithms require in terms of simple variants of the dual feasibility oracle. This framework has several benefits. First, it enables us to show that the Sinkhorn algorithm, which is currently the most popular MOT algorithm, requires strictly more structure than other algorithms do to solve MOT in $\poly(n,k)$ time. Second, our framework makes it much simpler to develop $\poly(n,k)$ time algorithms for a given MOT problem. In particular, it is necessary and sufficient to (approximately) solve the dual feasibility oracle---which is much more amenable to standard algorithmic techniques. 
	
	We illustrate this ease-of-use by developing $\poly(n,k)$-time algorithms for three general classes of MOT cost structures: (1) graphical structure; (2) set-optimization structure; and (3) low-rank plus sparse structure. For structure (1), we recover the known result that Sinkhorn has $\poly(n,k)$ runtime~\citep{h20gm,teh2002unified}; moreover, we provide the first $\poly(n,k)$ time algorithms for computing solutions that are exact and sparse. For structures (2)-(3), we give the first $\poly(n,k)$ time algorithms, even for approximate computation. Together, these three structures encompass many---if not most---current applications of MOT.

\end{abstract}

\setcounter{tocdepth}{2}
\tableofcontents

\section{Introduction}\label{sec:intro}

Multimarginal Optimal Transport ($\MOT$) is the problem of linear programming over joint probability distributions with fixed marginal distributions. In this way, $\MOT$ generalizes the classical Kantorovich formulation of Optimal Transport from $2$ marginal distributions to an arbitrary number $k \geq 2$ of them. 

\par More precisely, an $\MOT$ problem is specified by a cost tensor $C$ in the $k$-fold tensor product space $\Rntk = \R^n \otimes \cdots \otimes \R^n$, and $k$ marginal distributions $\mu_1, \dots, \mu_k$ in the simplex $\Delta_n = \{v \in \R_{\geq 0}^n : \sum_{i=1}^n v_i = 1 \}$.\footnote{For simplicity, all $\mu_i$ are assumed to have the same support size $n$. Everything in this paper extends in a straightforward way to non-uniform sizes $n_i$ where $n^k$ is replaced by $\prod_{i=1}^k n_i$, and $\poly(n,k)$ is replaced by $\poly(\max_i n_i, k)$.} The $\MOT$ problem is to compute
\begin{align}
	\min_{P \in \Coup} \langle P, C \rangle
	\label{eq:MOT-intro}
	\tag{MOT}
\end{align}
where $\Coup$ is the ``transportation polytope'' consisting of all entrywise non-negative tensors $P \in (\R^{n})^{\otimes k}$ 
satisfying the marginal constraints $\sum_{j_1,\dots,j_{i-1}, j_{i+1}, \dots, j_{k}} P_{j_1, \dots, j_{i-1}, j, j_{i+1}, \dots, j_k} = [\mu_i]_j$ for all $i \in \{1, \dots, k\}$ and $j \in \{1, \dots, n\}$.

\par This $\MOT$ problem has many applications throughout machine learning, computer science, and the natural sciences since it arises in tasks that require ``stitching'' together aggregate measurements. For instance, applications of $\MOT$ include
inference from collective dynamics~\citep{h20gm,h20partial}, 
information fusion for Bayesian learning~\citep{srivastava2018scalable},
averaging point clouds~\citep{AguCar11,CutDou14},
the $n$-coupling problem~\citep{ruschendorf2002n},
quantile aggregation~\citep{makarov1982estimates,ruschendorf1982random}, matching for teams~\citep{ChiMccNes10,CarEke10}, image processing~\citep{RabPeyDel11,solomon2015convolutional}, random combinatorial optimization~\citep{zemel1982polynomial,weiss1986stochastic,Nat21extremal,Agr12,MeiNad79,nadas1979probabilistic,Han86}, Distributionally Robust Optimization~\citep{Nat18dro,Nat09,MisNat14}, simulation of incompressible fluids~\citep{Bre08,BenCarNen19}, and Density Functional Theory~\citep{cotar2013density,buttazzo2012optimal,BenCarNen16}.

\par However, in most applications, the success of $\MOT$ is severely limited by the lack of efficient algorithms. Indeed, in general, $\MOT$ requires \emph{exponential time} in the number of marginals $k$ and their support sizes $n$. For instance, applying a linear program solver out-of-the-box takes $n^{\Theta(k)}$ time because $\MOT$ is a linear program with $n^k$ variables, $n^k$ non-negativity constraints, and $nk$ equality constraints. Specialized algorithms in the literature such as the Sinkhorn algorithm yield similar $n^{\Theta(k)}$ runtimes. Such runtimes currently limit the applicability of $\MOT$ to tiny-scale problems (e.g., $n=k=10$).

\paragraph{Polynomial-time algorithms for $\MOT$.} This paper develops \emph{polynomial-time} algorithms for $\MOT$, where here and henceforth ``polynomial'' means in the number of marginals $k$ and their support sizes $n$---and possibly also $\Cmax/\eps$ for $\eps$-additive approximation, where $\Cmax$ is a bound on the entries of $C$.

\par At first glance, this may seem impossible for at least two ``trivial'' reasons. One is that it takes exponential time to read the input cost $C$ since it has $n^k$ entries. We circumvent this issue by considering costs $C$ with $\poly(n,k)$-size implicit representations, which encompasses essentially all $\MOT$ applications.\footnote{E.g., in the $\MOT$ problems of Wasserstein barycenters, generalized Euler flows, and Density Functional Theory, $C$ has entries $C_{j_1, \dots, j_k} = \sum_{i,i'=1}^k g_{i,i'}(j_i,j_{i'})$ and thus can be implicitly input via the $k^2$ functions $g_{i,i'} : \{1, \dots, n\}^2 \to \R$.} A second obvious issue is that it takes exponential time to write the output variable $P$ since it has $n^k$ entries. We circumvent this issue by returning solutions $P$ with $\poly(n,k)$-size implicit representations, for instance sparse solutions.

But, of course, circumventing these issues of input/output size is not enough to actually solve $\MOT$ in polynomial time. See~\citep{AltBoi20hard} for examples of $\NP$-hard $\MOT$ problems with costs that have $\poly(n,k)$-size implicit representations. 

\par Remarkably, for several $\MOT$ problems, there are specially-tailored algorithms that run in polynomial time---notably, for $\MOT$ problems with graphically-structured costs of constant treewidth~\citep{h20tree,h20gm,teh2002unified}, variational mean-field games~\citep{BenCarDi18}, computing generalized Euler flows~\citep{BenCarCut15}, computing low-dimensional Wasserstein barycenters~\citep{CarObeOud15,BenCarCut15, AltBoi20bary}, and filtering and estimation tasks based on target tracking~\citep{h20tree,h20gm,h20incremental,h19hmc,h20partial}. However, the number of $\MOT$ problems that are known to be solvable in polynomial time is small, and it is unknown if these techniques can be extended to the many other $\MOT$ problems arising in applications. This motivates the central question driving this paper:
\begin{align*}
	\text{\emph{Are there general ``structural properties'' that make }} \MOT \text{\emph{ solvable in }} \poly(n,k) 	\text{\emph{ time?}}
\end{align*}

This paper is conceptually divided into two parts. In the first part of the paper, we develop a unified algorithmic framework for $\MOT$ that characterizes the structure required for different algorithms to solve $\MOT$ in $\poly(n,k)$ time, in terms of simple variants of the dual feasibility oracle. This enables us to prove that some algorithms can solve $\MOT$ problems in polynomial time whenever any algorithm can; whereas the popular $\Sink$ algorithm cannot. Moreover, this algorithmic framework makes it significantly easier to design a $\poly(n,k)$ time algorithm for a given $\MOT$ problem (when possible) because it now suffices to solve the dual feasibility oracle---and this is much more amenable to standard algorithmic techniques. In the second part of the paper, we demonstrate the ease-of-use of our algorithmic framework by applying it to three general classes of $\MOT$ cost structures. 

\par Below, we detail these two parts of the paper in \S\ref{ssec:intro:cont-alg} and \S\ref{ssec:intro:cont-app}, respectively.

\subsection{Contribution $1$: unified algorithmic framework for $\MOT$}\label{ssec:intro:cont-alg}

In order to understand what structural properties make $\MOT$ solvable in polynomial time, we first lay a more general groundwork. The purpose of this is to understand the following fundamental questions:
\begin{itemize}
	\item[Q1] What are reasonable candidate algorithms for solving structured $\MOT$ problems in polynomial time?
	\item[Q2] What structure must an $\MOT$ problem have for these algorithms to have polynomial runtimes?
	\item[Q3] Is the structure required by a given algorithm more restrictive than the structure required by a different algorithm (or \emph{any} algorithm)?
	\item[Q4] How to check if this structure occurs for a given $\MOT$ problem?
\end{itemize} 
We detail our answers to these four questions below in \S\ref{sssec:intro-q1}
to
\S\ref{sssec:intro-q4}, and then briefly discuss practical tradeoffs beyond polynomial-time solvability in \S\ref{sssec:intro-tradeoff}; see Table~\ref{tab:oracles} for a summary. We expect that this general groundwork will prove useful in future investigations of tractable $\MOT$ problems.

\begin{table}
	\centering
	\begin{tabular}{|c|c|c|c|c|c|c|}
		\hline
		\textbf{Algorithm} & \textbf{Oracle} &  \textbf{Runtime} & 
		\textbf{Always applicable?} & \textbf{Exact solution?} & \textbf{Sparse solution?}  & \textbf{Practical?}\\ \hline
		$\ELLIP$ & $\MinO$  & Theorem~\ref{thm:ellip} & Yes & Yes & Yes & No \\ \hline
		$\MWU$ & $\AMinO$   & Theorem~\ref{thm:mwumotp} & Yes & No & Yes & Yes\\ \hline
		$\Sink$ & $\SMinO$  & Theorem~\ref{thm:sink-mot:smin} & No & No & No & Yes \\ \hline                 
	\end{tabular}
	\caption{These $\MOT$ algorithms have polynomial runtime except for a bottleneck ``oracle''. Each oracle is a simple variant of the dual feasibility oracle for $\MOT$. The number of oracle computations is $\poly(n,k)$ for $\ELLIP$, and $\poly(n,k,\Cmax/\eps)$ for both $\MWU$ and $\Sink$. From a theoretical perspective, the most important aspect of an algorithm is whether it can solve $\MOT$ in polynomial time if and only if any algorithm can. We show that $\ELLIP$ and $\MWU$ satisfy this (Theorem~\ref{thm-intro:generality}), but $\Sink$ does not (Theorem~\ref{thm-intro:sink-generality}). From a practical perspective, $\Sink$ is the most scalable when applicable.\protect\footnotemark}
	\label{tab:oracles}
\end{table}\footnotetext{Code for implementing these algorithms and reproducing all numerical simulations in this paper is provided at \url{https://github.com/eboix/mot}.}

\subsubsection{Answer to Q1: candidate $\poly(n,k)$-time algorithms}\label{sssec:intro-q1}

We consider three algorithms for $\MOT$ whose exponential runtimes can be isolated into a single bottleneck---and thus can be implemented in polynomial time whenever that bottleneck can. These algorithms are the Ellipsoid algorithm \ELLIP~\citep{GLSbook}, the Multiplicative Weights Update algorithm \MWU~\citep{young2001sequential}, and the natural multidimensional analog of Sinkhorn's scaling algorithm \Sink~\citep{BenCarCut15,PeyCut17}. $\Sink$ is specially tailored to $\MOT$ and is currently the predominant algorithm for it. 
To foreshadow our answer to Q3, the reason that we restrict to these candidate algorithms is: we show that $\ELLIP$ and $\MWU$ can solve an $\MOT$ problem in polynomial time if and only if any algorithm can.

\subsubsection{Answer to Q2: structure necessary to run candidate algorithms}\label{sssec:intro-q2}

These three algorithms only access the cost tensor $C$ through polynomially many calls of their respective bottlenecks. Thus the structure required to implement these candidate algorithms in polynomial time is equivalent to the structure required to implement their respective bottlenecks in polynomial time.

In \S\ref{sec:algs}, we show that the bottlenecks of these three algorithms are polynomial-time equivalent to natural analogs of the feasibility oracle for the dual LP to $\MOT$. Namely, given weights $p_1, \dots, p_k \in \R^n$, compute
\begin{align}
	\min_{(j_1,\dots,j_k) \in \{1, \dots, n\}^k} C_{j_1,\dots,j_k}- \sum_{i=1}^k [p_i]_{j_i}
	\label{eq:intro-min}
\end{align}
either exactly for $\ELLIP$, approximately for $\MWU$, or with the ``min'' replaced by a ``softmin'' for $\Sink$. We call these three tasks the $\MinO$, $\AMinO$, and $\SMinO$ oracles, respectively. See Remark~\ref{rem:oracles-feas} for the interpretation of these oracles as variants of the dual feasibility oracle.

These three oracles take $n^k$ time to implement in general. However, for a wide range of structured cost tensors $C$ they can be implemented in $\poly(n,k)$ time, see \S\ref{ssec:intro:cont-app} below. For such structured costs $C$, our oracle abstraction immediately implies that the $\MOT$ problem with cost $C$ and any input marginals $\mu_1, \dots, \mu_k$ can be (approximately) solved in polynomial time by any of the three respective algorithms.

\par Our characterization of the algorithms' bottlenecks as variations of the dual feasibility oracle has two key benefits---which are the answers to Q3 and Q4, described below.

\subsubsection{Answer to Q3: characterizing what $\MOT$ problems each algorithm can solve}\label{sssec:intro-q3}

A key benefit of our characterization of the algorithms' bottlenecks as variations of the dual feasibility oracles is that it enables us to establish whether the structure required by a given $\MOT$ algorithm is more restrictive than the structure required by a different algorithm (or by \emph{any} algorithm).

\par In particular, this enables us to answer the natural question: why restrict to just the three algorithms described above? Can other algorithms solve $\MOT$ in $\poly(n,k)$ time in situations when these algorithms cannot? Critically, the answer is no: restricting ourselves to the $\ELLIP$ and $\MWU$ algorithms is at no loss of generality.

\begin{theorem}[Informal statement of part of Theorems~\ref{thm:ellip} and~\ref{thm:mwumotp}]
	\label{thm-intro:generality}
	For any family of costs $C \in \Rntk$:
	\begin{itemize}
		\item $\ELLIP$ computes an exact solution for $\MOT$ in $\poly(n,k)$ time if and only if any algorithm can.
		\item $\MWU$ computes an $\eps$-approximate solution for $\MOT$ in $\poly(n,k,\Cmax/\eps)$ time if and only if any algorithm can.
	\end{itemize}
\end{theorem}

The statement for $\ELLIP$ is implicit from classical results about LP~\citep{GLSbook} combined with arguments from~\citep{AltBoi20bary}, see the previous work section \S\ref{ssec:intro:prev}. The statement for $\MWU$ is new to this paper.

The oracle abstraction helps us show Theorem~\ref{thm-intro:generality} because it reduces this question of what structure is needed for the algorithms to solve $\MOT$ in polynomial time, to the question of what structure is needed to solve their respective bottlenecks in polynomial time. Thus Theorem~\ref{thm-intro:generality} is a consequence of the following result. (The ``if'' part of this result is a contribution of this paper; the ``only if'' part was shown in~\citep{AltBoi20hard}.)

\begin{theorem}[Informal statement of part of Theorems~\ref{thm:ellip} and~\ref{thm:mwumotp}]
	For any family of costs $C \in \Rntk$:
	\begin{itemize}
		\item $\MOT$ can be exactly solved in $\poly(n,k)$ time if and only if $\MinO$ can.
		\item $\MOT$ can be $\eps$-approximately solved in $\poly(n,k,\Cmax/\eps)$ time if and only if $\AMinO$ can.
	\end{itemize}
\end{theorem}

Interestingly, a further consequence of our unified algorithm-to-oracle abstraction is that it enables us to show that $\Sink$---which is currently the most popular algorithm for $\MOT$ by far---requires strictly more structure to solve an $\MOT$ problem than other algorithms require. This is in sharp contrast to the complete generality of the other two algorithms shown in Theorem~\ref{thm-intro:generality}.

\begin{theorem}[Informal statement of Theorem~\ref{thm:sink-separation}]\label{thm-intro:sink-generality}
	Under standard complexity-theoretic assumptions, there exists a family of $\MOT$ problems that can be solved exactly in $\poly(n,k)$ time using $\ELLIP$, however it is impossible to implement a single iteration of $\Sink$ (even approximately) in $\poly(n,k)$ time.
\end{theorem}

The reason that our unified algorithm-to-oracle abstraction helps us show Theorem~\ref{thm-intro:sink-generality} is that it puts $\Sink$ on equal footing with the other two classical algorithms in terms of their reliance on variants of the dual feasibility oracle. This reduces proving Theorem~\ref{thm-intro:sink-generality} to showing the following separation between the $\SMinO$ oracle and the other two oracles.

\begin{theorem}[Informal statement of Lemma~\ref{lem:smin-separation}]\label{thm-intro:sink-smin}
	Under standard complexity-theoretic assumptions,
	there exists a family of cost tensors $C \in \Rntk$ such that there are $\poly(n,k)$-time algorithms for $\MinO$ and $\AMinO$, however it is impossible to solve $\SMinO$ (even approximately) in $\poly(n,k)$ time.
\end{theorem}

\subsubsection{Answer to Q4: ease-of-use for checking if $\MOT$ is solvable in polynomial time}\label{sssec:intro-q4}

The second key benefit of this oracle abstraction is that it is helpful for showing that a given $\MOT$ problem (whose cost $C$ is input implicitly through some concise representation) is solvable in polynomial time as it without loss of generality reduces $\MOT$ to solving any of the three corresponding oracles in polynomial time. The upshot is that these oracles are more directly amenable to standard algorithmic techniques since they are phrased as more conventional combinatorial-optimization problems. In the second part of the paper, we illustrate this ease-of-use via applications to three general classes of structured $\MOT$ problems; for an overview see \S\ref{ssec:intro:cont-app}.

\subsubsection{Practical algorithmic tradeoffs beyond polynomial-time solvability}\label{sssec:intro-tradeoff}

From a theoretical perspective, the most important aspect of an algorithm is whether it can solve $\MOT$ in polynomial time if and only if any algorithm can. As we have discussed, this is true for $\ELLIP$ and $\MWU$ (Theorem~\ref{thm-intro:generality}) but not for $\Sink$ (Theorem~\ref{thm-intro:sink-generality}). Nevertheless, for a wide range of $\MOT$ cost structures, all three oracles can be implemented in polynomial time, which means that all three algorithms $\ELLIP$, $\MWU$, and $\Sink$ can be implemented in polynomial time. Which algorithm is best in practice depends on the relative importance of the following considerations for the particular application.
\begin{itemize}
	\item \textit{Error.} $\ELLIP$ computes exact solutions, whereas $\MWU$ and $\Sink$ only compute low-precision solutions due to $\poly(1/\eps)$ runtime dependence.
	
	\item \textit{Solution sparsity.} $\ELLIP$ and $\MWU$ output solutions with polynomially many non-zero entries (roughly $nk$), whereas $\Sink$ outputs fully dense solutions with $n^k$ non-zero entries (through a polynomial-size implicit representation, see \S\ref{ssec:algs:sink}). Solution sparsity enables interpretability, visualization, and efficient downstream computation---benefits which are helpful in diverse applications, for example ranging from computer graphics~\citep{solomon2015convolutional,blondel2018smooth, pitie2007automated} to facility location problems~\citep{anderes2016discrete} to machine learning~\citep{AltBoi20bary,flamary2016optimal} to ecological inference~\citep{muzellec2017tsallis} to fluid dynamics (see \S\ref{ssec:graphical:fluid}), and more. Furthermore, in \S\ref{ssec:lr:proj}, we show that sparse solutions for $\MOT$ (a.k.a. linear optimization over the transportation polytope) enable efficiently solving certain non-linear optimization problems over the transportation polytope.
	
	\item \textit{Practical runtime.} Although all three algorithms enjoy polynomial runtime guarantees, the polynomials are smaller for some algorithms than for others. In particular, $\Sink$ has remarkably good scalability in practice as long the error $\eps$ is not too small and its bottleneck oracle $\SMinO$ is practically implementable. By Theorems~\ref{thm-intro:generality} and~\ref{thm-intro:sink-generality}, $\MWU$ can solve strictly more $\MOT$ problems in polynomial time than $\Sink$; however, it is less scalable in practice when both $\MWU$ and $\Sink$ can be implemented. $\ELLIP$ is not practical and is used solely as a proof of concept that problems are tractable to solve exactly; in practice, we use Column Generation (see, e.g.,~\citep[\S6.1]{BerTsi97}) rather than $\ELLIP$ as it has better empirical performance, yet still has the same bottleneck oracle $\MinO$, see \S\ref{sssec:ellip:cg}. Column Generation is not as practically scalable as $\Sink$ in $n$ and $k$ but has the benefit of computing exact, sparse solutions. 
\end{itemize}
To summarize: which algorithm is best in practice depends on the application. For example, Column Generation
produces the qualitatively best solutions for the fluid dynamics application in \S\ref{ssec:graphical:fluid}, $\Sink$ is the most scalable for the risk estimation application in \S\ref{ssec:lr:risk}, and $\MWU$ is the most scalable for the network reliability application in \S\ref{ssec:binary:rel} (for that application there is no known implementation of $\Sink$ that is practically efficient).

\subsection{Contribution $2$: applications to general classes of structured $\MOT$ problems}\label{ssec:intro:cont-app}

In the second part of the paper, we illustrate the algorithmic framework developed in the first part of the paper by applying it to three general classes of MOT cost structures:
\begin{enumerate}
	\item Graphical structure (in \S\ref{sec:graphical}).
	\item Set-optimization structure (in \S\ref{sec:binary}).
	\item Low-rank plus sparse structure (in \S\ref{sec:lr}).
\end{enumerate} 
Specifically, if the cost $C$ is structured in any of these three ways, then $\MOT$ can be (approximately) solved in $\poly(n,k)$ time for any input marginals $\mu_1, \dots, \mu_k$. 

\par Previously, it was known how to solve $\MOT$ problems with structure (1) using $\Sink$~\citep{h20gm,teh2002unified}, but this only computes solutions that are dense (with $n^k$ non-zero entries) and low-precision (due to $\poly(1/\eps)$ runtime dependence). We therefore provide the first solutions that are sparse and exact for structure (1). For structures (2) and (3), we provide the first polynomial-time algorithms, even for approximate computation. These three structures are incomparable: it is in general not possible to model a problem falling under any of the three structures in a non-trivial way using any of the others, for details see Remarks~\ref{rem:binary-incomparable} and~\ref{rem:lr-incomparable}. This means that the new structures (2) and (3) enable capturing a wide range of new applications. 
\par Below, we detail these structures individually in 
\S\ref{sssec:intro:cont:graphical}, \S\ref{sssec:intro:cont:bin}, and \S\ref{sssec:intro:cont:lr}. See Table~\ref{tab:apps} for a summary.

\begin{table}
	\centering
	{\renewcommand{\arraystretch}{1.5}
	\begin{tabular}{|c|c|c|c|c|}
		\hline
	\multirow{ 2}{*}{\textbf{Structure}}
		 & \multirow{ 2}{*}{\textbf{Definition}} & \multirow{ 2}{*}{\textbf{Complexity measure}} &  \multicolumn{2}{c|}{\textbf{Polynomial-time algorithm?}}\\
		\cline{4-5}
		& & & \textbf{Approximate} & \textbf{Exact} \\
		\hline
		Graphical  (\S\ref{sec:graphical})              &
		$C_{\jvec} = \sum_{S \in \cS} f_S(\jvec_S)$
		& treewidth 
		& Known \citep{teh2002unified,h20gm}
		& Corollary~\ref{cor:graphical-algs}
		\\ \hline
		Set-optimization (\S\ref{sec:binary})             &
		$C_{\jvec} = \mathds{1}[\jvec \notin S]$
		& optimization oracle over S 
		& Corollary~\ref{cor:binary-algs}    
		& Corollary~\ref{cor:binary-algs}
		\\ \hline
		Low-rank + sparse  (\S\ref{sec:lr})        & 
		$C = R + S$
		& rank of $R$, sparsity of $S$
		& Corollary~\ref{cor:lr-algs}
		& Unknown    
		\\ \hline
	\end{tabular}}
	\caption{
		In the second part of the paper, 
		we illustrate the ease-of-use of our algorithmic framework 
		by applying it to three general classes of $\MOT$ cost structures. 
		These structures encompass many---if not most---current applications of \MOT.
	}
	\label{tab:apps}
\end{table}

\subsubsection{Graphical structure}\label{sssec:intro:cont:graphical}

In \S\ref{sec:graphical}, we apply our algorithmic framework to $\MOT$ problems with graphical structure, a broad class of $\MOT$ problems that have been previously studied~\citep{h20gm,h20tree,teh2002unified}. Briefly, an $\MOT$ problem has graphical structure if its cost tensor $C$ decomposes as 
\[
	C_{j_1, \dots, j_k}
	=
	\sum_{S \in \cS} f_S(\jvec_S),
\]
where $f_S(\jvec_S)$ are arbitrary ``local interactions'' that depend only on tuples $\jvec_S := \{j_i\}_{i \in S}$ of the $k$ variables. 

\par In order to derive efficient algorithms, it is necessary to restrict how local the interactions are because otherwise $\MOT$ is $\NP$-hard (even if all interaction sets $S \in \cS$ have size $2$)~\citep{AltBoi20hard}. We measure the locality of the interactions via the standard complexity measure of the ``treewidth'' of the associated graphical model. See \S\ref{ssec:graphical:structure} for formal definitions. While the runtimes of our algorithms (and all previous algorithms) depend exponentially on the treewidth, we emphasize that the treewidth is a very small constant (either $1$ or $2$) in all current applications of $\MOT$ falling under this framework; see the related work section.

\par We show that for $\MOT$ cost tensors that have graphical structure of constant treewidth, all three oracles can be implemented in $\poly(n,k)$ time. We accomplish this by leveraging the known connection between graphically structured $\MOT$ and graphical models shown in~\citep{h20gm}. In particular, the $\MinO$, $\AMinO$, and $\SMinO$ oracles are respectively equivalent to the mode, approximate mode, and log-partition function of an associated graphical model. Thus we can implement all oracles in $\poly(n,k)$ time by simply applying classical algorithms from the graphical models community~\citep{KolFri09,wainwright2008graphical}.

\begin{theorem}[Informal statement of Theorem~\ref{thm:graphical-oracles}]\label{thm-intro:graphical-oracles}
	Let $C \in \Rntk$ have graphical structure of constant treewidth. Then the $\MinO$, $\AMinO$, and $\SMinO$ oracles can be computed in $\poly(n,k)$ time.
\end{theorem}

It is an immediate corollary of Theorem~\ref{thm-intro:graphical-oracles} and our algorithms-to-oracles reduction described in \S\ref{ssec:intro:cont-alg} that one can implement $\ELLIP$, $\MWU$, and $\Sink$ in polynomial time. Below, we record the theoretical guarantee of $\ELLIP$ since it is the best of the three algorithms as it computes exact, sparse solutions.

\begin{theorem}[Informal statement of Corollary~\ref{cor:graphical-algs}]\label{thm-intro:graphical}
	Let $C \in \Rntk$ have graphical structure of constant treewidth. Then an exact, sparse solution for $\MOT$ can be computed in $\poly(n,k)$ time.
\end{theorem}

Previously, it was known how to solve such $\MOT$ problems~\citep{teh2002unified,h20gm} using $\Sink$, but this only computes a solution that is fully dense (with $n^k$ non-zero entries) and low-precision (due to $\poly(1/\eps)$ runtime dependence). Details in the related work section. Our result improves over this state-of-the-art algorithm by producing solutions that are \emph{exact} and \emph{sparse} in $\poly(n,k)$ time.

\par In \S\ref{ssec:graphical:fluid}, we demonstrate the benefit of Theorem~\ref{thm-intro:graphical} on the application of computing generalized Euler flows, which was historically the motivation of $\MOT$ and has received significant attention, e.g.,~\citep{BenCarCut15,BenCarNen19,Bre08,Bre89,Bre93,Bre99}. While there is a specially-tailored version of the $\Sink$ algorithm for this problem that runs in polynomial time~\citep{BenCarCut15,BenCarNen19}, it produces solutions that are  approximate and fully dense. Our algorithm produces exact, sparse solutions which lead to sharp visualizations rather than blurry ones (see Figure~\ref{fig:fluids}).

\subsubsection{Set-optimization structure}\label{sssec:intro:cont:bin}

In \S\ref{sec:binary}, we apply our algorithmic framework to $\MOT$ problems whose cost tensors $C$ take value $0$ or $1$ in each entry. That, is costs $C$ of the form
\[
	C_{j_1,\dots,j_k} = \mathds{1}[(j_1,\dots,j_k) \notin S],
\]
for some subset $S \subseteq [n]^k$. Such $\MOT$ problems arise naturally in applications where one seeks to minimize the probability that some event $S$ occurs, given marginal probabilities on each variable $j_i$, see Example~\ref{ex:binary}. 
 
\par In order to derive efficient algorithms, it is necessary to restrict the (otherwise arbitrary) set $S$. We parametrize the complexity of such $\MOT$ problems via the complexity of finding the minimum-weight object in $S$. This opens the door to combinatorial applications of $\MOT$ because finding the minimum-weight object in $S$ is well-known to be polynomial-time solvable for many ``combinatorially-structured'' sets $S$ of interest---e.g., the set $S$ of cuts in a graph, or the set $S$ of independent sets in a matroid. 
\par We show that for $\MOT$ cost tensors with this structure, all three oracles can be implemented efficiently.

\begin{theorem}[Informal statement of Theorem~\ref{thm:binary-oracles}]\label{thm-intro:binary-oracles}
	Let $C \in \Rntk$ have set-optimization structure. Then the $\MinO$, $\AMinO$, and $\SMinO$ oracles can be computed in $\poly(n,k)$ time. 
\end{theorem}

It is an immediate corollary of Theorem~\ref{thm-intro:binary-oracles} and our algorithms-to-oracles reduction described in \S\ref{ssec:intro:cont-alg} that one can implement $\ELLIP$, $\MWU$, and $\Sink$ in polynomial time. Below, we record the theoretical guarantee for $\ELLIP$ since it is the best of these three algorithms as it computes exact, sparse solutions.

\begin{theorem}[Informal statement of Corollary~\ref{cor:binary-algs}]\label{thm-intro:binary}
	Let $C \in \Rntk$ have set-optimization structure. Then an exact, sparse solution for $\MOT$ can be computed in $\poly(n,k)$ time.
\end{theorem}

This is the first polynomial-time algorithm for this class of $\MOT$ problems. We note that a more restrictive class of $\MOT$ problems was studied in~\citep{zemel1982polynomial} under the additional restriction that $S$ is upwards-closed.

In \S\ref{ssec:binary:rel}, we show how this general class of set-optimization structure captures, for example, the classical application of computing the extremal reliability of a network with stochastic edge failures. Network reliability is a fundamental topic in network science and engineering~\citep{gertsbakh2011network,ball1995network,ball1986computational} which is often studied in an average-case setting where each edge fails independently with some given probability \citep{moore1956reliable,karger2001randomized,valiant1979complexity,provan1983complexity}. The application investigated here is a robust notion of network reliability in which edge failures may be maximally correlated (e.g., by an adversary) or minimally correlated (e.g., by a network maintainer) subject to a marginal constraint on each edge's failure probability, a setting that dates back to the 1980s~\citep{zemel1982polynomial,weiss1986stochastic}. We show how to express both the minimally and maximally correlated network reliability problems as $\MOT$ problems with set-optimization structure, recovering as a special case of our general framework the known polynomial-time algorithms in~\citep{zemel1982polynomial,weiss1986stochastic} as well as more practical polynomial-time algorithms that scale to input sizes that are an order-of-magnitude larger.

\subsubsection{Low-rank and sparse structure}\label{sssec:intro:cont:lr}

In \S\ref{sec:lr}, we apply our algorithmic framework to $\MOT$ problems whose cost tensors $C$ decompose as
\[
	C = R + S,
\]
where $R$ is a constant-rank tensor, and $S$ is a polynomially-sparse tensor. We assume that $R$ is represented in factored form, and that $S$ is represented through its non-zero entries, which overall yields a $\poly(n,k)$-size representation of $C$.

\par We show that for $\MOT$ cost tensors with low-rank plus sparse structure, the $\AMinO$ and $\SMinO$ oracles can be implemented in polynomial time.\footnote{It is an interesting open question if the $\MinO$ oracle can similarly be implemented in $\poly(n,k)$ time. This would enable implementing $\ELLIP$ in $\poly(n,k)$ time by our algorithms-to-oracles reduction, and thus would enable computing exact solutions for this class of $\MOT$  problems (cf., Theorem~\ref{thm-intro:lr}).\label{fn:lr-exact}}
This may be of independent interest because, by taking all oracle inputs $p_i = 0$ in~\eqref{eq:intro-min}, this generalizes the previously open problem of approximately computing the smallest entry of a constant-rank tensor with $n^k$ entries in $\poly(n,k)$ time.

\begin{theorem}[Informal statement of Theorem~\ref{thm:lr-oracles}]\label{thm-intro:lr-oracles}
	Let $C \in \Rntk$ have low-rank plus sparse structure. Then the $\AMinO$ and $\SMinO$ oracles can be computed in $\poly(n,k,\Cmax/\eps)$ time. 
\end{theorem}

It is an immediate corollary of Theorem~\ref{thm-intro:lr-oracles} and our algorithms-to-oracles reduction described in \S\ref{ssec:intro:cont-alg} that one can implement $\MWU$ and $\Sink$ in polynomial time. Of these two algorithms, $\MWU$ computes sparse solutions, yielding the following theorem.

\begin{theorem}[Informal statement of Corollary~\ref{cor:lr-algs}]\label{thm-intro:lr}
	Let $C \in \Rntk$ have low-rank plus sparse structure. Then a sparse, $\eps$-approximate solution for $\MOT$ can be computed in $\poly(n,k,\Cmax/\eps)$ time.
\end{theorem}

This is the first polynomial-time result for this class of $\MOT$ problems. We note that the runtime of our $\MOT$ algorithm depends exponentially on the rank $r$ of $R$, hence why we take $r$ to be constant. Nevertheless, such a restriction on the rank is unavoidable since unless $\P = \NP$, there does not exist an algorithm with runtime that is jointly polynomial in $n$, $k$, and the rank $r$~\citep{AltBoi20hard}.

\par We demonstrate this polynomial-time algorithm concretely on two applications. First, in \S\ref{ssec:lr:risk} we consider the risk estimation problem of computing an investor's expected profit in the worst-case over all future prices that are consistent with given marginal distributions. We show that this is equivalent to an $\MOT$ problem with a low-rank tensor and thereby provide the first efficient algorithm for it.

\par Second, in \S\ref{ssec:lr:proj}, we consider the fundamental problem of projecting a joint distribution $Q$ onto the transportation polytope. We provide the first polynomial-time algorithm for solving this when $Q$ decomposes into a constant-rank and sparse component, which models mixtures of product distributions with polynomially many corruptions. This application illustrates the versatility of our algorithmic results beyond polynomial-time solvability of $\MOT$, since this projection problem is a \emph{quadratic} optimization over the transportation polytope rather than linear optimization (a.k.a. $\MOT$). In order to achieve this, we develop a simple quadratic-to-linear reduction tailored to this problem that crucially exploits the sparsity of the $\MOT$ solutions enabled by the $\MWU$ algorithm.

\subsection{Related work}\label{ssec:intro:prev}

\subsubsection{$\MOT$ algorithms}

$\MOT$ algorithms fall into two categories. One category consists of general-purpose algorithms that do not depend on the specific $\MOT$ cost. For example, this includes running an LP solver out-of-the-box, or running the Sinkhorn algorithm where in each iteration one sums over all $n^k$ entries of the cost tensor to implement the marginalization bottleneck~\citep{LinHoJor19,Fri20,tupitsa2020multimarginal}. These approaches are robust in the sense that they do not need to be changed based on the specific $\MOT$ problem. However, they are impractical beyond tiny input sizes (e.g., $n=k=10$) because their runtimes scale as $n^{\Omega(k)}$.

\par The second category consists of algorithms that are much more scalable but require extra structure of the $\MOT$ problem. Specifically, these are algorithms that somehow exploit the structure of the relevant cost tensor $C$ in order to (approximately) solve an $\MOT$ problem in $\poly(n,k)$ time~\citep{teh2002unified,h20gm,h20tree,BenCarDi18,h20incremental,h19hmc,h20partial,CarObeOud15,Nen16,BenCarCut15,BenCarNen19,AltBoi20bary, Nat18dro,Nat09,MisNat14,zemel1982polynomial,weiss1986stochastic,Nat21extremal,Agr12,MeiNad79,nadas1979probabilistic,Han86}. Such a $\poly(n,k)$ runtime is far more tractable---but it is not well understood for which $\MOT$ problems such a runtime is possible. The purpose of this paper is to clarify this question.

 \par To contextualize our answer to this question with the rapidly growing literature requires further splitting this second category of algorithms. 

\paragraph{Sinkhorn algorithm.} Currently, the predominant approach in the second category is to solve an entropically regularized version of $\MOT$ with the Sinkhorn algorithm, a.k.a.~Iterative Proportional Fitting or Iterative Bregman Projections or RAS algorithm or Iterative Scaling algorithm, see
e.g.,~\citep{teh2002unified,BenCarDi18,BenCarNen16,BenCarNen19,Nen16,h20tree,h20gm}. Recent work has shown that a polynomial number of iterations of this algorithm suffices~\citep{LinHoJor19,Fri20,tupitsa2020multimarginal}. However, the bottleneck is that each iteration requires $n^{k}$ operations in general because it requires marginalizing a tensor with $n^k$ entries. The critical question is therefore: what structure of an $\MOT$ problem enables implementing this marginalization bottleneck in polynomial time.
\par This paper makes two contributions to this question. First, we identify new broad classes of $\MOT$ problems for which this bottleneck can be implemented in polynomial time, and thus $\Sink$ can be implemented in polynomial time (see \S\ref{ssec:intro:cont-app}). Second, we propose other algorithms that require strictly less structure than $\Sink$ does in order to solve an $\MOT$ problem in polynomial time (Theorem~\ref{thm:sink-separation}).

\paragraph{Ellipsoid algorithm.} The Ellipsoid algorithm is among the most classical algorithms for implicit LP~\citep{GLSbook,GroLovSch81,khachiyan1980polynomial}, however it has taken a back seat to the $\Sink$ algorithm in the vast majority of the $\MOT$ literature.
\par In \S\ref{ssec:algs:ellip}, we make explicit the fact that the variant of $\ELLIP$ from~\citep{AltBoi20bary} can solve $\MOT$ exactly in $\poly(n,k)$ time if and only if any algorithm can (Theorem~\ref{thm:ellip}). This is implicit from combining several known results~\citep{AltBoi20hard,AltBoi20bary,GLSbook}. In the process of making this result explicit, we exploit the special structure of the $\MOT$ LP to significantly simplify the reduction from the dual violation oracle to the dual feasibility oracle. The previously known reduction is highly impractical as it requires an indirect ``back-and-forth'' use of the Ellipsoid algorithm~\citep[page 107]{GLSbook}. In contrast, our reduction is direct and simple; this is critical for implementing our practical alternative to $\ELLIP$, namely $\COLGEN$, with the dual feasibility oracle.

\paragraph{Multiplicative Weights Update algorithm.} This algorithm, first introduced by~\citep{young2001sequential},
 has been studied in the context of optimal transport when $k=2$~\citep{BlaJamKenSid18,Qua18}, in which case implicit LP is not necessary for a polynomial runtime. $\MWU$ lends itself to implicit LP~\citep{young2001sequential}, but is notably absent from the $\MOT$ literature. 
\par In \S\ref{ssec:algs:mwu}, we show that $\MWU$ can be applied to $\MOT$ in polynomial time if and only if the approximate dual feasibility oracle can be solved in polynomial time. To do this, we show that in the special case of $\MOT$, the well-known ``softmax-derivative'' bottleneck of $\MWU$ is polynomial-time equivalent to the approximate dual feasibility oracle. Since it is known that the approximate dual feasibility oracle is polynomial-time reducible to approximate $\MOT$~\citep{AltBoi20hard}, we therefore establish that $\MWU$ can solve $\MOT$ approximately in polynomial time if and only if any algorithm can (Theorem~\ref{thm:mwumotp}).

\subsubsection{Graphically structured $\MOT$ problems with constant treewidth}

We isolate here graphically structured costs with constant treewidth because this framework encompasses all $\MOT$ problems that were previously known to be tractable in polynomial time~\citep{teh2002unified,h20gm}, with the exceptions of the fixed-dimensional Wasserstein barycenter problem and $\MOT$ problems related to combinatorial optimization---both of which are described below in \S\ref{sssec:prev:struc-general}. This family of graphical structured costs with treewidth $1$ (a.k.a. ``tree-structured costs''~\citep{h20tree}) includes applications in economics such as variational mean-field games~\citep{BenCarDi18}, interpolating histograms on trees~\citep{akagi2020probabilistic}, matching for teams~\citep{CarObeOud15,Nen16}; as well as encompasses applications in filtering and estimation for collective dynamics such as target tracking~\citep{h20tree,h20gm,h20incremental,h19hmc,h20partial} and Wasserstein barycenters in the case of fixed support~\citep{h20partial,Nen16,BenCarCut15,CarObeOud15}. With treewidth $2$, this family of costs also includes dynamic multi-commodity flow problems~\citep{haasler2021scalable}, as well as the application of computing generalized Euler flows in fluid dynamics~\citep{BenCarCut15,BenCarNen19,Nen16}, which was historically the original motivation of $\MOT$~\citep{Bre89,Bre93,Bre99,Bre08}.

\paragraph{Previous polynomial-time algorithms for graphically structured $\MOT$ compute approximate, dense solutions.} Implementing $\Sink$ for graphically structured $\MOT$ problems by using belief propagation to efficiently implement the marginalization bottleneck was first proposed twenty years ago in~\citep{teh2002unified}. There have been recent advancements in understanding connections of this algorithm to the Schr\"odinger bridge problem in the case of trees~\citep{h20tree}, as well as developing more practically efficient single-loop variations~\citep{h20gm}. 
\par All of these works prove theoretical runtime guarantees only in the case of tree structure (i.e., treewidth $1$). However, this graphical model perspective for efficiently implementing $\Sink$ readily extends to any constant treewidth: simply implement the marginalization bottleneck using junction trees. This, combined with the iteration complexity of $\Sink$ which is known to be polynomial~\citep{LinHoJor19,Fri20,tupitsa2020multimarginal}, immediately yields an overall polynomial runtime. This is why we cite~\citep{teh2002unified,h20gm} throughout this paper regarding the fact that $\Sink$ can be implemented in polynomial time for graphical structure with any constant treewidth. 
\par While the use of $\Sink$ for graphically structured $\MOT$ is mathematically elegant and can be impressively scalable in practice, it has two drawbacks.
The first drawback of this algorithm is that it computes (implicit representations of) solutions that are fully dense with $n^k$ non-zero entries. 
Indeed, it is well-known that $\Sink$ finds the unique optimal solution to the entropically regularized $\MOT$ problem $\min_{P \in \Coup} \langle P,C \rangle - \eta^{-1}H(P)$, and that this solution is fully dense~\citep{PeyCut17}. For example, in the simple case of cost $C = 0$, uniform marginals $\mu_i$, and any strictly positive regularization parameter $\eta > 0$, this solution $P$ has value $1/n^k$ in each entry. 
\par The second drawback of this algorithm is that it only computes solutions that are low-precision due to $\poly(1/\eps)$ runtime dependence on the accuracy $\eps$. This is because the number of $\Sink$ iterations is known to scale polynomially in the entropic regularization parameter $\eta$ even in the matrix case $k=2$~\citep[\S1.2]{linial1998deterministic}, and it is known that $\eta = \Omega(\eps^{-1} k \log n)$ is necessary for the converged solution of $\Sink$ to be an $\eps$-approximate solution to the (unregularized) original $\MOT$ problem~\citep{LinHoJor19}.

\paragraph{Improved algorithms for graphically structured $\MOT$ problems.}
 The contribution of this paper to the study of graphically structured $\MOT$ problems is that we give the first $\poly(n,k)$ time algorithms that can compute solutions which are exact and sparse (Corollary~\ref{cor:graphical-algs}). Our framework also directly recovers all known results about $\Sink$ for graphically structured $\MOT$ problems---namely that it can be implemented in polynomial time for trees~\citep{h20tree,teh2002unified} and for constant treewidth~\citep{h20gm,teh2002unified}.

\subsubsection{Tractable $\MOT$ problems beyond graphically structured costs}\label{sssec:prev:struc-general}

The two new classes of $\MOT$ problems studied in this paper---namely, set-optimization structure and low-rank plus sparse structure---are incomparable to each other as well as to graphical structure. Details in Remarks~\ref{rem:binary-incomparable}~and~\ref{rem:lr-incomparable}. This lets us handle a wide range of new $\MOT$ problems that could not be handled before.

There are two other classes of $\MOT$ problems studied in the literature which do not fall under the three structures studied in this paper. We elaborate on both below.

\begin{remark}[Low-dimensional Wasserstein barycenter]\label{rem:intro-wb}
	This $\MOT$ problem has cost $C_{j_1,\dots,j_k} = \sum_{i,i'=1}^k \|x_{i,j_i} - x_{i',j_{i'}}\|^2$ where $x_{i,j} \in \R^d$ denotes the $j$-th atom in the distribution $\mu_i$. Clearly this cost is not a graphically structured cost of constant treewidth---indeed, representing it through the lens of graphical structure requires the complete graph of interactions, which means a maximal treewidth of $k-1$.\footnote{
		We remark that the related but different problem of \emph{fixed-support} Wasserstein barycenters has graphical structure with treewidth $1$
		\citep{h20partial,Nen16,BenCarCut15,CarObeOud15}. However, it should be emphasized that the fixed-support Wasserstein barycenter problem is different from the Wasserstein barycenter problem: it only approximates the latter to $\eps$ accuracy if the fixed support is restricted to an $O(\eps)$-net which requires $n = 1/\eps^{\Omega(d)}$ discretization size for the barycenter's support, and thus (i) even in constant dimension, does not lead to high-precision algorithms due to $\poly(1/\eps)$ runtime; and (ii) scales exponentially in the dimension $d$. See~\citep[\S1.3]{AltBoi20baryhard} for further details about the complexity of Wasserstein barycenters.
	} This problem also does not fall under the set-optimization or constant-rank structures. Nevertheless, this $\MOT$ problem can be solved in $\poly(n,k)$ time for any fixed dimension $d$ by exploiting the low-dimensional geometric structure of the points $\{x_{i,j}\}$ that implicitly define the cost~\citep{AltBoi20bary}.
\end{remark}

\begin{remark}[Random combinatorial optimization]
	$\MOT$ problems also appear in the random combinatorial optimization literature since the 1970s, see e.g.,~\citep{MeiNad79,zemel1982polynomial,weiss1986stochastic,nadas1979probabilistic,Han86}, although under a different name and in a different community. These papers consider $\MOT$ problems with costs of the form $C(x) = \min_{v \in V} \langle x, v\rangle$ for polytopes $V \subseteq \{0,1\}^k$ given through a list of their extreme points. Applications include PERT (Program Evaluation and Review Technique), extremal network reliability, and scheduling. Recently, applications to Distributionally Robust Optimization were investigated in~\citep{Nat18dro,Nat09,MisNat14} which considered general polytopes $V \subset \R^k$, as well as in~\citep{Nat21extremal} which considered $\MOT$ costs of the related form $C(x) = \mathds{1}[\min_{v \in V} \langle x, v \rangle \geq t]$, and in~\citep{Agr12} which considers other combinatorial costs $C$ such as sub/supermodular functions. These papers show that these random combinatorial optimization problems are in general intractable, and give sufficient conditions on when they can be solved in polynomial time. In general, these families of $\MOT$ problems are different from the three structures studied in this paper, although some $\MOT$ applications fall under multiple umbrellas (e.g., extremal network reliability). It is an interesting question to understand to what extent these structures can be reconciled (as well as the algorithms, which sometimes use extended formulations in these papers).
\end{remark}

\subsubsection{Intractable $\MOT$ problems}

These algorithmic results beg the question: what are the fundamental limitations of this line of work on polynomial-time algorithms for structured $\MOT$ problems? To this end, the recent paper~\citep{AltBoi20hard} provides a systematic investigation of $\NP$-hardness results for structured $\MOT$ problems, including converses to several results in this paper. In particular,~\citep[Propositions 4.1 and 4.2]{AltBoi20hard} justify the constant-rank regime studied in \S\ref{sec:lr} by showing that unless $\P=\NP$, there does not exist an algorithm with runtime that is jointly polynomially in the rank $r$ and the input parameters $n$ and $k$. Similarly,~\citep[Propositions 5.1 and 5.2]{AltBoi20hard} justify the constant-treewidth regime for graphically structured costs studied in \S\ref{sec:graphical} and all previous work by showing that unless $\P=\NP$, there does not exist an algorithm with polynomial runtime even in the seemingly simple class of $\MOT$ costs that decompose into pairwise interactions $C_{j_1,\dots,j_k} = \sum_{i \neq i' \in [k]} c_{i,i'}(j_i,j_i')$. The paper~\citep{AltBoi20hard} also shows $\NP$-hardness for several $\MOT$ problems with repulsive costs, including for example the $\MOT$ formulation of Density Functional Theory with Coulomb-Buckingham potential. It is an problem whether the Coulomb potential, studied in~\citep{BenCarNen16,cotar2013density,buttazzo2012optimal}, also leads to an $\NP$-hard $\MOT$ problem~\citep[Conjecture 6.4]{AltBoi20hard}.

\subsubsection{Variants of $\MOT$}
The literature has studied several other variants of the $\MOT$ problem, notably with entropic regularization and/or with constraints on a subset of the $k$ marginals, see, e.g.,~\citep{BenCarCut15,BenCarDi18,BenCarNen19,BenCarNen16,h20partial,haasler2021scalable,h20incremental,h20gm,h20tree,h19hmc,LinHoJor19}. Our techniques readily apply with little change. Briefly, to handle entropic regularization, simply use the $\SMinO$ oracle and $\Sink$ algorithm with fixed regularization parameter $1/\eta > 0$ (rather than $1/\eta$ of vanishing size $\Theta(\eps / \log n)$) as described in \S\ref{ssec:algs:sink}. And to handle partial marginal constraints, essentially the only change is that in the $\MinO$, $\AMinO$, and $\SMinO$ oracles, the potentials $p_i$ are zero for all indices $i \in [k]$ corresponding to unconstrained marginals $m_i(P)$. Full details are omitted for brevity since they are straightforward modifications of our main results.

\subsubsection{Optimization over joint distributions}
Optimization problems over exponential-size joint distributions appear in many domains. For instance, they arise in game theory when computing correlated equilibria~\citep{pap08}; however, in that case the optimization has different constraints which lead to different algorithms. Such problems also arise in variational inference~\citep{wainwright2003variational}; however, the optimization there typically constrains this distribution to ensure tractability (e.g., mean-field approximation restricts to product distributions). The different constraints in these optimization problems over joint distributions versus $\MOT$ lead to significant differences in computational complexity, and thus also necessitate different algorithmic techniques.

\subsection{Organization}

In \S\ref{sec:prelim} we recall preliminaries about $\MOT$ and establish notation. The first part of the paper then establishes our unified algorithmic framework for $\MOT$. Specifically, in \S\ref{sec:oracles} we define and compare three variants of the dual feasibility oracle; and in \S\ref{sec:algs} we characterize the structure that $\MOT$ algorithms require for polynomial-time implementation in terms of these three oracles. For an overview of these results, see \S\ref{ssec:intro:cont-alg}. The second part of the paper applies this algorithmic framework to three general classes of $\MOT$ cost structures: graphical structure (\S\ref{sec:graphical}), set-optimization structure (\S\ref{sec:binary}), and low-rank plus sparse structure (\S\ref{sec:lr}). For an overview of these results, see \S\ref{ssec:intro:cont-app}. These three application sections are independent of each other and can be read separately. We conclude in \S\ref{sec:discussion}.

\section{Preliminaries}\label{sec:prelim}

\paragraph{General notation.} The set $\{1, \dots, n\}$ is denoted by $[n]$. For shorthand, we write $\poly(t_1, \dots, t_m)$ to denote a function that grows at most polynomially fast in those parameters. Throughout, we assume for simplicity of exposition that all entries of the input $C$ and $\mu_1, \dots, \mu_k$ have bit complexity at most $\poly(n,k)$, and same with the components defining $C$ in structured settings. As such, throughout runtimes refer to the number of arithmetic operations. The set $\R \cup \{-\infty\}$ is denoted by $\Rbar$, and note that the value $-\infty$ can be represented efficiently by adding a single flag bit. We use the standard $O(\cdot)$ and $\Omega(\cdot)$ notation, and use $\tilde{O}(\cdot)$ and $\tilde{\Omega}(\cdot)$ to denote that polylogarithmic factors may be omitted.

\paragraph{Tensor notation.} 
The $k$-fold tensor product space $\R^n \otimes \cdots \otimes \R^n$ is denoted by $\Rntk$, and similarly for $\Rpntk$. Let $P \in \Rntk$. Its $i$-th marginal, $i \in [k]$, is denoted by $m_i(P) \in \R^n$ and has entries $[m_i(P)]_j := \sum_{j_1,\dots,j_{i-1},j_{i+1}, \dots, j_k} P_{j_1,\dots,j_{i-1},j,j_{i+1}, \dots, j_k}$.
For shorthand, we often denote an index $(j_1,\dots,j_k)$ by $\jvec$. The sum of $P$'s entries is denoted by $m(P) = \sum_{\jvec} P_{\jvec}$. The maximum absolute value of $P$'s entries is denoted by $\|P\|_{\max} := \max_{\jvec} |P_{\jvec}|$, or simply $P_{\max}$ for short. For $\jvec \in [n]^k$, we write $\delta_{\jvec}$ to denote the tensor with value $1$ at entry $\jvec$, and $0$ elswewhere.
The operations $\odot$ and $\otimes$ respectively denote the entrywise product and the Kronecker product. The notation $\otimes_{i=1}^k d_i$ is shorthand for $d_1 \otimes \cdots \otimes d_k$.
A non-standard notation we use throughout is that $f[P]$ denotes a function $f : \R \to \R$ (typically $\exp$, $\log$, or a polynomial) applied \emph{entrywise} to a tensor $P$.

\subsection{Multimarginal Optimal Transport}\label{ssec:prelim:mot}

The transportation polytope between measures $\mu_1, \dots, \mu_k \in \Delta_n$ is
\begin{align}
\Coup := 
\left\{
P \in \Rpntk \; : \; m_i(P) = \mu_i, \; \forall i \in [k]	\right\}.
\label{eq:def-Coup}
\end{align}
For a fixed cost $C \in \Rntk$, the $\MOT_C$ problem is to solve the following linear program, given input measures $\mu = (\mu_1, \dots, \mu_k) \in (\Delta_n)^k$:
\begin{align}
\min_{P \in \Coup} \langle P, C \rangle.
\label{MOT}
\tag{MOT}
\end{align}
In the $k=2$ matrix case,~\eqref{MOT} is the Kantorovich formulation of $\OT$~\citep{Vil03}. Its dual LP is
\begin{align}
& \max_{p_1,\ldots,p_k \in \Rn} \sum_{i=1}^k \langle p_i, \mu_i \rangle \tag{MOT-D}
\quad \text{subject to} \quad  
C_{j_1,\dots,j_k} - \sum_{i=1}^k [p_i]_{j_i} \geq 0, \; \forall (j_1,\dots,j_k) \in [n]^k.
\label{MOT-D} 
\end{align}

A basic, folklore fact about $\MOT$ is that it always has a sparse optimal solution (e.g.,~\citep[Lemma 3]{anderes2016discrete}). This follows from elementary facts about standard-form LP; we provide a short proof for completeness.

\begin{lemma}[Sparse solutions for $\MOT$]\label{lem:mot-sparse}
	For any cost $C \in \Rntk$ and any marginals $\mu_1, \dots, \mu_k \in \Delta_n$, there exists an optimal solution $P$ to $\MOT_C(\mu)$ that has at most $nk-k+1$ non-zero entries. 
\end{lemma}
\begin{proof}
	Since \eqref{MOT} is an LP over a compact domain, it has an optimal solution at a vertex~\citep[Theorem 2.7]{BerTsi97}. Since \eqref{MOT} is a standard-form LP, these vertices are in correspondence with basic solutions, thus their sparsity is bounded by the number of linearly dependent constraints defining $\Coup$~\citep[Theorem 2.4]{BerTsi97}. We bound this quantity by $nk-k+1$ via two observations. First, $\Coup$ is defined by $nk$ equality constraints $[m_i(P)]_j = [\mu_i]_j$ in~\eqref{eq:def-Coup}, one for each coordinate $j \in [n]$ of each marginal constraint $i \in [k]$. Second, at least $k-1$ of these constraints are linearly dependent because we can construct $k$ distinct linear combinations of them, namely $\sum_{j \in [n]} [m_i(P)]_j = \sum_{j \in [n]} [\mu_i]_j$ for each marginal $i \in [k]$, which all simplify to the same constraint $m(P) = 1$, and thus are redundant with each other.
\end{proof}

\begin{defin}[$\eps$-approximate $\MOT$ solution]\label{def:mot}
	$P$ is an $\eps$-approximate solution to $\MOT_C(\mu)$ if $P$ is feasible (i.e., $P \in \Coup$) and $\langle C, P \rangle$ is at most $\eps$ more than the optimal value. 
\end{defin}

\subsection{Regularization}\label{ssec:prelim:rmot}

We introduce two standard regularization operators. First is the Shannon entropy $H(P) := -\sum_{\jvec} P_{\jvec} \log P_{\jvec}$ of a tensor $P \in \Rpntk$ with entries summing to $m(P) = 1$. We adopt the standard notational convention that $0 \log 0 = 0$. Second is the softmin operator, which is defined for parameter $\eta > 0$ as
\begin{align}
\smineta_{i \in [m]} a_i := -\frac{1}{\eta} \log \left( \sum_{i=1}^m e^{-\eta a_i} \right).
\end{align}
This softmin operator naturally extends to $a_i \in \R \cup \{\infty\}$ by adopting the standard notational conventions that $e^{-\infty} = 0$ and $\log 0 = -\infty$.
\par We make use of the following folklore fact, which bounds the error between the $\min$ and $\smin$ operators based on the regularization and the number of points. For completeness, we provide a short proof.
\begin{lemma}[Softmin approximation bound]\label{lem:smin}
	For any $a_1, \dots, a_m \in \R \cup \{\infty\}$ and $\eta > 0$,
	\begin{align*}
		\min_{i \in [m]} a_i
		\geq 
		\smineta_{i \in [m]} a_i
		\geq 
		\min_{i \in [m]} a_i - \frac{\log m}{\eta}.
	\end{align*}
\end{lemma}
\begin{proof}
	Assume without loss of generality that all $a_i$ are finite, else $a_i$ can be dropped (if all $a_i  = \infty$ then the claim is trivial). For shorthand, denote $\min_{i \in [m]} a_i$ by $a_{\min}$. For the first inequality, use the non-negativity of the exponential function to bound
	\[
		\smineta_{i \in [m]} a_i 
		=
		-\frac{1}{\eta} \log \left( \sum_{i=1}^m e^{-\eta a_i} \right)
		\leq 
		-\frac{1}{\eta} \log \left( e^{-\eta a_{\min}} \right)
		= 
		a_{\min}.
	\]
	For the second inequality, use the fact that each $a_i \geq a_{\min}$ to bound
	\[
		\smineta_{i \in [m]} a_i 
		=
		-\frac{1}{\eta} \log \left( \sum_{i=1}^m e^{-\eta a_i} \right)
		\geq 
		-\frac{1}{\eta} \log \left( m e^{-\eta a_{\min}} \right)
		=
		a_{\min} -\frac{\log m}{\eta}.
	\]
\end{proof}

The entropically regularized MOT problem ($\RMOT$ for short) is the convex optimization problem
\begin{align}
	\min_{P \in \Coup} \langle P, C \rangle - \eta^{-1} H(P).
	\label{RMOT}
	\tag{RMOT}
\end{align}
This is the natural multidimensional analog of entropically regularized $\OT$, which has a rich literature in statistics~\citep{Leo13} and transportation theory~\citep{Wil69}, and has recently attracted significant interest in machine learning~\citep{Cut13,PeyCut17}. The convex dual of~\eqref{RMOT} is the convex optimization problem
\begin{align}
	\max_{p_1, \dots, p_k \in \R^n} \sum_{i=1}^k \langle p_i, \mu_i \rangle + 
		\smineta_{\jvec \in [n]^k}\left( C_{\jvec} - \sum_{i=1}^k [p_i]_{j_i} \right).
	\label{RMOT-D}
	\tag{RMOT-D}
\end{align}

In contrast to $\MOT$, there is no analog of Lemma~\ref{lem:mot-sparse} for $\RMOT$: the unique optimal solution to $\RMOT$ is dense. Further, this solution may not even be ``approximately'' sparse. For example, when $C = 0$, all $\mu_i$ are uniform, and $\eta > 0$ is any positive number, the solution is fully dense with all entries equal to $1/n^k$.

We define $P$ to be an $\eps$-approximate $\RMOT$ solution in the analogous way as in Definition~\ref{def:mot}. A basic, folklore fact about $\RMOT$ is that if the regularization $\eta$ is sufficiently large, then $\RMOT$ and $\MOT$ are equivalent in terms of approximate solutions.

\begin{lemma}[$\MOT$ and $\RMOT$ are close for large regularization $\eta$]\label{lem:reg-close}
	Let $P \in \Coup$, $\eps > 0$, and $ \eta \geq \eps^{-1} k \log n$. If $P$ is an $\eps$-approximate solution to \eqref{RMOT}, then $P$ is also a $(2\eps)$-approximate solution to \eqref{MOT}; and vice versa. 
\end{lemma}
\begin{proof}
	Since a discrete distribution supported on $n^k$ atoms has entropy at most $k \log n$~\citep{CoverThomas}, the objectives of \eqref{MOT} and \eqref{RMOT} differ pointwise by at most $\eta^{-1} k \log n \leq \eps$. Since \eqref{MOT} and \eqref{RMOT} also have the same feasible sets, their optimal values therefore differ by at most $\eps$. 
\end{proof}

\section{Oracles}\label{sec:oracles}

Here we define the three oracle variants described in the introduction and discuss their relations.  In the below definitions, let $C \in \Rntk$ be a cost tensor.

\begin{defin}[$\MinO$ oracle]\label{def:oracle-min}
	For weights $p = (p_1,\ldots,p_k) \in \R^{n \times k}$, $\MinO_C(p)$ returns
	\[
	\min_{\jvec \in [n]^k} C_{\jvec} - \sum_{i=1}^k [p_i]_{j_i}.
	\]
\end{defin}

\begin{defin}[$\AMinO$ oracle]\label{def:oracle-amin}
	For weights $p = (p_1,\ldots,p_k) \in \R^{n \times k}$ and accuracy $\eps > 0$, $\AMinO_C(p,\eps)$ returns $\MinO_C(p)$ up to additive error $\eps$.
\end{defin}

\begin{defin}[$\SMinO$ oracle]\label{def:oracle-sminn}
	For weights $p = (p_1,\ldots,p_k) \in \Rbar^{n \times k}$ and regularization parameter $\eta > 0$, $\SMinO_C(p, \eta)$ returns 
	\[
	\displaystyle\smineta_{\jvec \in [n]^k} C_{\jvec} - \sum_{i=1}^k [p_i]_{j_i}.
	\]
\end{defin}

An algorithm is said to ``solve'' or ``implement'' $\MinO_C$ if given input $p$, it outputs $\MinO_C(p)$. Similarly for $\AMinO_C$ and $\SMinO_C$. Note that the weights $p$ that are input to $\SMinO$ have values inside $\Rbar = \R \cup \{-\infty\}$; this simplifies the notation in the treatment of the $\Sink$ algorithm below and does not increase the bit-complexity by more than $1$ bit by adding a flag for the value $-\infty$. 

\begin{remark}[Interpretation as variants of the dual feasibility oracle]\label{rem:oracles-feas}
	These three oracles can be viewed as variants of the feasibility oracle for~\eqref{MOT-D}. For $\MinO_C(p)$, this relationship is exact: $p \in \R^{n \times k}$ is feasible for~\eqref{MOT-D} if and only if $\MinO_C(p)$ is non-negative. For $\AMinO_C$ and $\SMinO_C$, this relationship is approximate, with the approximation depending on how small $\eps$ is and how large $\eta$ is, respectively.
\end{remark}

Since these oracles form the respective bottlenecks of all algorithms from the $\MOT$ and implicit linear programming literatures (see the overview in the introduction \S\ref{ssec:intro:cont-alg}), an important question is: if one oracle can be implemented in $\poly(n,k)$ time, does this imply that the other can be too?

Two reductions are straightforward: the $\AMinO$ oracle can be implemented in $\poly(n,k)$ time whenever either the $\MinO$ oracle or the $\SMinO$ oracle can be implemented in $\poly(n,k)$ time. We record these simple observations in remarks for easy recall.

\begin{remark}[$\MinO$ implies $\AMinO$]\label{rem:oracles:min-amin}
	For any accuracy $\eps > 0$, the $\MinO_C(p)$ oracle provides a valid answer to the $\AMinO_C(p,\eps)$ oracle by definition.
\end{remark}

\begin{remark}[$\SMinO$ implies $\AMinO$]\label{rem:oracles:smin-amin}
	For any $p \in \R^{n \times k}$ and regularization $\eta \geq \eps^{-1} k\log n$, the $\SMinO_C(p,\eta)$ oracle provides a valid answer to the $\AMinO_C(p,\eps)$ oracle due to the approximation property of the $\smin$ operator (Lemma~\ref{lem:smin}). 
\end{remark}

In the remainder of this section, we show a separation between the $\SMinO$ oracle and both the $\MinO$ and $\AMinO$ oracles by exhibiting a family of cost tensors $C$ for which there exist polynomial-time algorithms for $\MinO$ and $\AMinO$, however there is no polynomial-time algorithm for $\SMinO$. The non-existence of a polynomial-time algorithm of course requires a complexity theoretic assumption; our result holds under $\#$BIS-hardness---which is a by-now standard complexity theory assumption introduced in~\citep{dyer2000relative}, and in words is the statement that there does not exist a polynomial-time algorithm for counting the number of independent sets in a bipartite graph.

\begin{lemma}[Restrictiveness of the $\SMinO$ oracle]\label{lem:smin-separation}
	There exists a family of costs $C \in \Rntk$ for which $\MinO_C$ and $\AMinO_C$ can be solved in $\poly(n,k)$ time, however $\SMinO_C$ is \#BIS-hard.
\end{lemma}
\begin{proof}
	In order to prove hardness for general $n$, it suffices to exhibit such a family of cost tensors when $n=2$. Since $n=2$, it is convenient to abuse notation slightly by indexing a cost tensor $C \in \Rntk$ by $\jvec \in \{-1,1\}^k$ rather than by $\jvec \in \{1,2\}^k$. The family we exhibit is
	$
	\{ C(A,b)  :  A \in \R_{\geq 0}^{k \times k}, \, b \in \R^k \}$,
	where the cost tensors $C(A,b)$ are parameterized by a non-negative square matrix $A$ and a vector $b$, and have entries of the form
	\[
		C_{\jvec}(A,b) := - \langle \jvec, A \jvec \rangle - \langle b, \jvec \rangle, \qquad \jvec \in \{\plusminus 1\}^k.
	\]
	\par \underline{Polynomial-time algorithm for $\MinO$ and $\AMinO$.} We show that given a matrix $A \in \R_{\geq 0}^{k \times k}$, vector $b \in \R^k$, and weights $p \in \R^{2 \times k}$, it is possible to compute $\MinO_C(p)$ on the cost tensor $C(A,b)$ in $\poly(k)$ time. Clearly this also implies a $\poly(k)$ time algorithm for $\AMinO_C(p,\eps)$ for any $\eps > 0$, see Remark~\ref{rem:oracles:min-amin}.
	
	\par To this end, we first re-write the $\MinO_C(p)$ problem on $C(A,b)$ in a more convenient form that enables us to ``ignore'' the weights $p$. Recall that $\MinO_C(p)$ is the problem of 
	\[
		\MinO_C(p)
		=
		\min_{\jvec \in \{\plusminus 1\}^k} 
		- \langle \jvec, A \jvec \rangle 
		- \langle b, \jvec \rangle
		- \sum_{i=1}^k [p_i]_{j_i}.
	\]
	Note that the linear part of the cost is equal to 
	\begin{align}
		\langle b, \jvec \rangle
		+ \sum_{i=1}^k [p_i]_{j_i} = 
		\langle \ell, \jvec \rangle + d, 
		\label{eq:lem-smin-sep:simp}
	\end{align}
	where $\ell \in \R^k$ is the vector with entries $\ell_i = b_i + \half((p_i)_1 - (p_i)_{-1})$, and $d$ is the scalar $d = \half \sum_{i=1}^k ([p_i]_1 + [p_i]_{-1})$. Thus, since $d$ is clearly computable in $O(k)$ time, the $\MinO_C$ problem is equivalent to solving
	\begin{align}
		\min_{\jvec \in \{\plusminus 1\}^k} 
		- \langle \jvec, A \jvec \rangle 
		- \langle \ell, \jvec \rangle,
		\label{eq:lem-smin-sep:submodular}
	\end{align}
	when given as input a non-negative matrix $A \in \R_{\geq 0}^{k \times k}$ and a vector $\ell \in \R^k$.
	\par To show that this task is solvable in $\poly(k)$ time, note that the objective in~\eqref{eq:lem-smin-sep:submodular} is a submodular function because it is a quadratic whose Hessian $-A$ has non-positive off-diagonal terms~\citep[Proposition 6.3]{Bach13}. Therefore~\eqref{eq:lem-smin-sep:submodular} is a submodular optimization problem, and thus can be solved in $\poly(k)$ time using classical algorithms from combinatorial optimization~\citep[Chapter 10.2]{GLSbook}.

	\par \underline{$\SMinO$ oracle is $\#$BIS-hard.} On the other hand, by using the definition of the $\SMinO$ oracle, the re-parameterization~\eqref{eq:lem-smin-sep:simp}, and then the definition of the softmin operator,  
	\[
		\SMinO_C(p,\eta) = 
		\smineta_{\jvec \in \{\plusminus 1\}^k} - \langle \jvec, A \jvec \rangle 
		- \langle b, \jvec \rangle
		- \sum_{i=1}^k [p_i]_{j_i}
		=
		\smineta_{\jvec \in \{\plusminus 1\}^k} - \langle \jvec, A \jvec \rangle 
		- \langle \ell, \jvec \rangle - d
		= - \frac{\log Z}{\eta} - d,
	\]
	where $Z =  \sum_{\jvec \in \{\plusminus 1\}^k} Q(\jvec) $ is the partition function
	of the ferromagnetic Ising model with inconsistent external fields given by
	\[
		Q(\jvec) =  
		\exp\left( \eta  \langle  \vec{j}, A \vec{j} \rangle + \eta \langle \ell, \jvec \rangle  \right).
	\]
	Because it is $\#$BIS hard to compute the partition function $Z$ of a ferromagnetic Ising model with inconsistent external fields~\citep{goldberg2007complexity}, it is $\#$BIS hard to compute the value $-\eta^{-1} \log Z - d$ of the oracle $\SMinO_C(p,\eta)$.
\end{proof}

\begin{remark}[The restrictiveness of $\SMinO$ extends to approximate computation]\label{rem:smin-separation}
	The separation between the oracles shown in Lemma~\ref{lem:smin-separation} further extends to approximate computation of the $\SMinO$ oracle under the assumption that $\#$BIS is hard to approximate, since under this assumption it is hard to approximate the partition function of a ferromagnetic Ising model with inconsistent external fields~\citep{goldberg2007complexity}.
\end{remark}

\section{Algorithms to oracles}\label{sec:algs}

In this section, we consider three algorithms for $\MOT$. Each is iterative and requires only polynomially many iterations. The key issue for each algorithm is the per-iteration runtime, which is in general exponential (roughly $n^k$). We isolate the respective bottlenecks of these three algorithms into the three variants of the dual feasibility oracle defined in \S\ref{sec:oracles}.
See \S\ref{ssec:intro:cont-alg} and Table~\ref{tab:oracles} for a high-level overview of this section's results.

\subsection{The Ellipsoid algorithm and the $\MinO$ oracle}\label{ssec:algs:ellip}

Among the most classical algorithms for implicit LP is the Ellipsoid algorithm~\citep{GLSbook,GroLovSch81,khachiyan1980polynomial}.
However it has taken a back seat to the $\Sink$ algorithm in the vast majority of the $\MOT$ literature. The very recent paper~\citep{AltBoi20bary}, which focuses on the specific $\MOT$ application of computing low-dimensional Wasserstein barycenters, develops a variant of the classical Ellipsoid algorithm specialized to $\MOT$; henceforth this is called $\ELLIP$, see \S\ref{sssec:ellip:alg} for a description of this algorithm. The objective of this section is to analyze $\ELLIP$ in the context of general $\MOT$ problems in order to prove the following.

\begin{theorem}\label{thm:ellip}
	For any family of cost tensors $C \in \Rntk$, the following are equivalent:
	\begin{itemize}
		\item[(i)] $\ELLIP$ takes $\poly(n,k)$ time to solve the $\MOT_C$ problem. (Moreover, it outputs a vertex solution represented as a sparse tensor with at most $nk-k+1$ non-zeros.)
		\item[(ii)] There exists a $\poly(n,k)$ time algorithm that solves the $\MOT_C$ problem.
		\item[(iii)] There exists a $\poly(n,k)$ time algorithm that solves the $\MinO_C$ problem. 
	\end{itemize}
\end{theorem}

\paragraph{Interpretation of results.} In words, the equivalence ``(i) $\Longleftrightarrow$ (ii)'' establishes that $\ELLIP$ can solve any $\MOT$ problem in polynomial time that any other algorithm can. Thus from a theoretical perspective, this paper's restriction to $\ELLIP$ is at no loss of generality for developing polynomial-time algorithms that exactly solve $\MOT$. In words, the equivalence ``(ii) $\Longleftrightarrow$ (iii)'' establishes that the $\MOT$ and $\MinO$ problems are polynomial-time equivalent. Thus we may investigate when $\MOT$ is tractable by instead investigating the more amenable question of when $\MinO$ is tractable (see \S\ref{sssec:intro-q4}) at no loss of generality.

As stated in the related work section, Theorem~\ref{thm:ellip} is implicit from combining several known results~\citep{AltBoi20hard,AltBoi20bary, GLSbook}. Our contribution here is to make this result explicit, since this allows us to unify algorithms from the implicit LP literature with the $\Sink$ algorithm. We also significantly simplify part of the implication ``(iii) $\Longrightarrow$ (i)'', which
is crucial for making an algorithm that relies on the $\MinO$ oracle practical---namely, the Column Generation algorithm discussed below. 

\paragraph{Organization of \S\ref{ssec:algs:ellip}.}  In \S\ref{sssec:ellip:alg}, we recall this $\ELLIP$ algorithm and how it depends on the violation oracle for~\eqref{MOT-D}. 
In \S\ref{sssec:ellip:oracle}, we give a significantly simpler proof that the violation and feasibility oracles are polynomial-time equivalent in the case of~\eqref{MOT-D}, and use this to prove Theorem~\ref{thm:ellip}.
In \S\ref{sssec:ellip:cg}, we describe a practical implementation that replaces the $\ELLIP$ outer loop with Column Generation.

\subsubsection{Algorithm}\label{sssec:ellip:alg}

A key component of the proof of Theorem~\ref{thm:ellip} is the $\ELLIP$ algorithm introduced in~\citep{AltBoi20bary} for $\MOT$, which we describe below. In order to present this, we first define a variant of the $\MinO$ oracle that returns a minimizing tuple rather than the minimizing value.

\begin{defin}[Violation oracle for~\eqref{MOT-D}]
	Given weights $p = (p_1,\dots,p_k) \in \R^{n \times k}$, $\ArgMinO_C$ returns the minimizing solution $\jvec$ and value of $\min_{\jvec \in [n]^k} C_{\jvec} - \sum_{i=1}^k [p_i]_{j_i}$.
\end{defin}

$\ArgMinO_C$ can be viewed as a violation oracle\footnote{Recall that a violation oracle for a polytope $K = \{x : \langle a_i, x \rangle \leq b_i, \forall i \in [N]\}$ is an algorithm that given a point $p$, either asserts $p$ is in $K$, or otherwise outputs the index $i$ of a violated constraint $\langle a_i,p \rangle > b_i$.} for the decision set to~\eqref{MOT-D}. This is because, given $p = (p_1,\dots,p_k) \in \R^{n \times k}$, the tuple $\jvec$ output by $\ArgMinO_C(p)$ either provides a violated constraint if $C_{\jvec} - \sum_{i=1}^k [p_i]_{j_i} < 0$, or otherwise certifies $p$ is feasible. In \citep{AltBoi20bary} it is proved that $\MOT$ can be solved with polynomially many calls to the $\ArgMinO_C$ oracle.

\begin{theorem}[$\ELLIP$ guarantee; Proposition 12 of~\citep{AltBoi20bary}]\label{thm:ellip-argmin}
	Algorithm~\ref{alg:ellip} finds an optimal vertex solution for~$\MOT_C(\mu)$ using $\poly(n,k)$ calls to the $\ArgMinO_C$ oracle and $\poly(n,k)$ additional time. The solution is returned as a sparse tensor with at most $nk-k+1$ non-zero entries.
\end{theorem}

\begin{algorithm}
	\caption{$\ELLIP$: specialization of the classical Ellipsoid algorithm to $\MOT$}
	\hspace*{\algorithmicindent} \textbf{Input:} 
	Cost $C \in \Rntk$,
	marginals $\mu_1,\dots,\mu_k \in \Delta_n$  \\
	\hspace*{\algorithmicindent} \textbf{Output:} Vertex solution to $\MOT_C(\mu)$
	\begin{algorithmic}[1]
		\Statex \textbackslash\textbackslash$\;$ Solve dual
		\State Solve~\eqref{MOT-D} using the Ellipsoid algorithm, with $\ArgMinO_C$ as the violation oracle. Let $S$ denote the set of tuples returned by all calls to $\ArgMinO_C$.
		\Statex
		\Statex \textbackslash\textbackslash$\;$ Solve primal
		\State Solve~\eqref{eq-ellip:mots} using the Ellipsoid algorithm. 
	\end{algorithmic}
	\label{alg:ellip}
\end{algorithm}

\paragraph{Sketch of algorithm.}
Full details and a proof are in \citep{AltBoi20bary}. We give a brief overview here for completeness. First, recall from the implicit LP literature that the classical Ellipsoid algorithm can be implemented in polynomial time for an LP with arbitrarily many constraints so long as it has polynomially many variables and the violation oracle for its decision set is solvable in polynomial time~\citep{GLSbook}. This does not directly apply to the LP~\eqref{MOT} because that LP has $n^k$ variables. However, it can apply to the dual LP~\eqref{MOT-D} because that LP only has $nk$ variables. 

\par This suggests a natural two-step algorithm for $\MOT$. First, compute an optimal dual solution by directly applying the Ellipsoid algorithm to~\eqref{MOT-D}. Second, use this dual solution to construct a sparse primal solution. Although this dual-to-primal conversion does not extend to arbitrary LP~\citep[Exercise 4.17]{BerTsi97}, the paper~\citep{AltBoi20bary} provides a solution by exploiting the standard-form structure of $\MOT$. The procedure is to solve
\begin{align}
	\min_{\substack{P \in \Coup \\ \text{s.t. } P_{\jvec} = 0, \; \forall \jvec \notin S}} \langle C, P \rangle 
	\label{eq-ellip:mots}
\end{align}
which is the $\MOT$ problem restricted to sparsity pattern $S$, where $S$ is the set of tuples $\jvec$ returned by the violation oracle during the execution of step one of Algorithm~\ref{alg:ellip}. This second step takes $\poly(n,k)$ time using a standard LP solver, because running the Ellipsoid algorithm in the first step only calls the violation oracle $\poly(n,k)$ times, and thus $S$ has $\poly(n,k)$ size, and therefore the LP~\eqref{eq-ellip:mots} has $\poly(n,k)$ variables and constraints. In \citep{AltBoi20bary} it is proved that this produces a primal vertex solution to the original $\MOT$ problem.

\subsubsection{Equivalence of bottleneck to $\MinO$}\label{sssec:ellip:oracle}

Although Theorem~\ref{thm:ellip-argmin} shows that $\ELLIP$ can solve $\MOT$ in $\poly(n,k)$ time using the $\ArgMinO$ oracle, this is not sufficient to prove the implication ``(iii)$\implies$(i)'' in Theorem~\ref{thm:ellip}. In order to prove that implication requires showing the polynomial-time equivalence between $\MinO$ and $\ArgMinO$.

\begin{lemma}[Equivalence of $\MinO$ and $\ArgMinO$]\label{lem:minargmin}
	Each of the oracles $\MinO_C$ and $\ArgMinO_C$ can be implemented using $\poly(n,k)$ calls of the other oracle and $\poly(n,k)$ additional time. 
\end{lemma}

This equivalence follows from classical results about the equivalence of violation and feasibility oracles~\citep{yudin1976informational}. However, the known proof of that general result requires an involved and indirect argument based on ``back-and-forth'' applications of the Ellipsoid algorithm~\citep[\S4.3]{GLSbook}. Here we exploit the special structure of $\MOT$ to give a direct and elementary proof. This is essential to practical implementations (see \S\ref{sssec:ellip:cg}).

\begin{proof}
	It is obvious how the $\MinO_C$ oracle can be implemented via a single call of the $\ArgMinO_C$ oracle; we now show the converse. Specifically, given $p_1,\dots,p_k \in \R^n$, we show how to compute a solution $\jvec = (j_1,\dots,j_k) \in [n]^k$ for $\ArgMinO_C([p_1, \dots, p_k])$ using $nk$ calls to the $\MinO_C$ oracle and polynomial additional time. We use the first $n$ calls to compute the first index $j_1$ of the solution, the next $n$ calls to compute the next index $j_2$, and so on.
	\par Formally, for $s \in [k]$, let us say that $(j_1^*, \dots,j_{s}^*) \in [n]^s$ is a ``partial solution'' of size $s$ if there exists a solution $j \in [n]^k$ for $\ArgMinO_C([p_1, \dots, p_k])$ that satisfies $j_i = j_i^*$ for all $i \in [s]$. Then it suffices to show that for every $s \in [k]$, it is possible to compute a partial solution $(j_1^*,\dots,j_{s}^*)$ of size $s$ from a partial solution $(j_1^*,\dots,j_{s-1}^*)$ of size $s-1$ using $n$ calls to the $\MinO_C$ oracle and polynomial additional time. 
	\par The simple but key observation enabling this is the following. Below, for $i \in [k]$ and $j \in [n]$, define $q_{i,j}$ to be the vector in $\R^n$ with value $[p_i]_j$ on entry $j$, and value $-M$ on all other entries. In words, the following observation states that if the constant $M$ is sufficiently large, then for any indices $j_i'$, replacing the vectors $p_i$ with the vectors $q_{i,j_i'}$ in a $\MinO$ oracle query effectively performs a $\MinO$ oracle query on the original input $p_1,\dots,p_k$ except that now the minimization is only over $\jvec \in [n]^k$ satisfying $j_i = j_i'$.
		
	\begin{obs}\label{obs:minargmin}
		Set $M := 2\Cmax + 2\sum_{i=1}^k \|p_i\|_{\max} + 1$. 
		Then for any $s \in [k]$ and any $(j_1', \dots, j_s') \in [n]^s$,
		\begin{align}
			\MinO_C([q_{1,j_1'}, \dots, q_{s,j_s'}, p_{s+1},\dots, p_k])
			=
			\min_{\substack{\vec{j} \in [n]^k \\ \text{s.t. } j_1 = j_1', \dots, j_s = j_s'}} C_{\jvec} -  \sum_{i=1}^k [p_i]_{j_i} .
			\nonumber
		\end{align}
	\end{obs}
	\begin{proof}
		By definition of the $\MinO$ oracle,
		\begin{align*}
			\MinO_C([q_{1,j_1'}, \dots, q_{s,j_s'}, p_{s+1},\dots, p_k])
			&=
			\min_{\jvec \in [n]^k} C_{\jvec} - \sum_{i=1}^s [q_{i,j_i'}]_{j_i} - \sum_{i=s+1}^k [p_{i}]_{j_i}
		\end{align*}
		It suffices to prove that every minimizing tuple $\jvec \in [n]^k$ for the right hand side satisfies $j_i = j_i'$ for all $i \in [s]$. Suppose not for sake of contradiction. Then there exists a minimizing tuple $\jvec \in [n]^k$ for which $j_{\ell} \neq j_{\ell}'$ for some $\ell \in [s]$. But then $[q_{\ell,j_\ell'}]_{j_\ell} = -M$, so the objective value of $\jvec$ is at least
		\[
			C_{\jvec} - \sum_{i=1}^s [q_{i,j_i'}]_{j_i} - \sum_{i=s+1}^k [p_{i}]_{j_i}
			\geq 
			M - \Cmax  - \sum_{i=1}^k \|p_i\|_{\max}
			= 
			\Cmax + \sum_{i=1}^k \|p_i\|_{\max} + 1.
		\]
		But this is strictly larger (by at least $1$) than the value of any tuple with prefix $(j_1',\dots,j_{s}')$, contradicting the optimality of $\jvec$.
	\end{proof}
	
	Thus, given a partial solution $(j_1^*,\dots,j_{s-1}^*)$ of length $s-1$, we construct a partial solution $(j_1^*,\dots,j_s^*)$ of length $s$ by setting $j_s^*$ to be a minimizer of 
	\begin{align}
		\min_{j_s' \in [n]} \MinO_C([q_{1,j_1^*}, \dots, q_{s-1,j_{s-1}^*}, q_{s,j_s'}, p_{s+1},\dots, p_k]).
		\label{eq:lem-pf:minargmin}
	\end{align}
	The runtime bound is clear; it remains to show correctness. To this end, note that 
	\begin{align*}
			\min_{\substack{\jvec \in [n]^k \\ \text{s.t. }j_1 = j_1^*, \dots, j_s = j_s^* }} C_{\jvec} - \sum_{i=1}^k [p_i]_{j_i}
		&=
			\MinO_C([q_{1,j_1^*}, \dots,q_{s,j_s^*}, p_{s+1},\dots, p_k])
		\\ &= 
			\min_{j_s' \in [n]}
			\MinO_C([q_{1,j_1^*}, \dots, q_{s-1,j_{s-1}^*}, q_{s,j_s'}, p_{s+1},\dots, p_k])
		\\ &= 
			\min_{j_s' \in [n]}
		\min_{\substack{\jvec \in [n]^k \\ \text{s.t. } j_1 = j_1^*, \dots, j_{s-1} = j_{s-1}^*, j_s = j_s'}} C_{\jvec} - \sum_{i=1}^k [p_i]_{j_i}
		\\ &=
			\min_{\substack{\jvec \in [n]^k \\ \text{s.t. }j_1 = j_1^*, \dots, j_{s-1} = j_{s-1}^*}} C_{\jvec} - \sum_{i=1}^k [p_i]_{j_i}
		\\ &=
			\MinO_{C}([p_1,\dots,p_k]),
	\end{align*}
	where above the first and third steps are by Observation~\ref{obs:minargmin}, the second step is by construction of $j_s^*$, the fourth step is by simplifying, and the final step is by definition of $(j_1^*,\dots,j_{s-1}^*)$ being a partial solution of size $s-1$. We conclude that $(j_1^*,\dots,j_s^*)$ is a partial solution of size $s$, as desired.
\end{proof}

We can now conclude the proof of the main result of \S\ref{ssec:algs:ellip}.

\begin{proof}[Proof of Theorem~\ref{thm:ellip}]
	The implication ``(i) $\implies$ (ii)'' is trivial, and the implication ``(ii) $\implies$ (iii)'' is shown in~\citep{AltBoi20hard}. It therefore suffices to show the implication ``(iii) $\implies$ (i)''. This follows from combining the fact that $\ELLIP$ solves $\MOT_C$ in polynomial time given an efficient implementation of $\ArgMinO_C$ (Theorem~\ref{thm:ellip-argmin}), with the fact that the $\MinO_C$ and $\ArgMinO_C$ oracles are polynomial-time equivalent (Lemma~\ref{lem:minargmin}).
\end{proof}

\subsubsection{Practical implementation via Column Generation}\label{sssec:ellip:cg}

Although $\ELLIP$ enjoys powerful theoretical runtime guarantees, it is slow in practice because the classical Ellipsoid algorithm is. Nevertheless, whenever $\ELLIP$ is applicable (i.e., whenever the $\MinO_C$ oracle can be efficiently implemented), we can use an alternative practical algorithm, namely the delayed Column Generation method $\COLGEN$, to compute exact, sparse solutions to $\MOT$.
\par For completeness, we briefly recall the idea behind $\COLGEN$; for further details see the standard textbook~\citep[\S6.1]{BerTsi97}. $\COLGEN$ runs the Simplex method, keeping only basic variables in the tableau. 
Each time that $\COLGEN$ needs to find a Simplex variable on which to pivot, it solves the ``pricing problem'' of finding a variable with negative reduced cost. This is the key subroutine in $\COLGEN$. In the present context of the $\MOT$ LP, this pricing problem is equivalent to a call to the $\ArgMinO$ violation oracle (see~\citep[Definition 3.2]{BerTsi97} for the definition of reduced costs). By the polynomial-time equivalence of the $\ArgMinO$ and $\MinO$ oracles shown in Lemma~\ref{lem:minargmin}, this bottleneck subroutine in $\COLGEN$ can be computed in polynomial time whenever the $\MinO$ oracle can. 
For easy recall, we summarize this discussion as follows.

\begin{theorem}[Standard guarantee for $\COLGEN$; Section 6.1 of~\citep{BerTsi97}]\label{thm:colgen}
	For any $T > 0$, one can implement $T$ iterations of $\COLGEN$ in $\poly(n,k,T)$ time and calls to the $\MinO_C$ oracle.
	When $\COLGEN$ terminates, it returns an optimal vertex solution, which is given as a sparse tensor with at most $nk-k+1$ non-zero entries.
\end{theorem}

Note that
$\COLGEN$ does not have a theoretical guarantee stating that it terminates after a polynomial number of iterations. But it often performs well in practice and terminates after a small number of iterations, leading to much better empirical performance than $\ELLIP$.

\subsection{The Multiplicative Weights Update algorithm and the $\AMinO$ oracle}\label{ssec:algs:mwu}

The second classical algorithm for solving implicitly-structured LPs that we study in the context of $\MOT$ is the Multiplicative Weights Update algorithm $\MWU$ \citep{young2001sequential}. The objective of this section is to prove the following guarantees for its specialization to $\MOT$.

\begin{theorem}\label{thm:mwumotp}
For any family of cost tensors $C \in \Rntk$, the following are equivalent:
\begin{itemize}
\item[(i)] For any $\eps > 0$, $\MWU$ takes $\poly(n,k,\Cmax/\eps)$ time to solve the $\MOT_C$ problem $\eps$-approximately. (Moreover, it outputs a sparse solution with at most $\poly(n,k,\Cmax/\eps)$ non-zero entries.)
\item[(ii)] There exists a $\poly(n,k,\Cmax/\eps)$-time algorithm that solves the $\MOT_C$ problem $\eps$-approximately for any $\eps > 0$.
\item[(iii)] There exists a $\poly(n,k,\Cmax/\eps)$-time algorithm that solves the $\AMinO_C$ problem $\eps$-approximately for any $\eps > 0$.
\end{itemize}
\end{theorem}

\paragraph{Interpretation of results.} Similarly to the analogous Theorem~\ref{thm:ellip} for $\ELLIP$, the equivalence ``(i) $\Longleftrightarrow$ (ii)'' establishes that $\MWU$ can approximately solve any $\MOT$ problem in polynomial time that any other algorithm can. Thus, from a theoretical perspective, restricting to $\MWU$ for approximately solving $\MOT$ problems is at no loss of generality. In words, the equivalence ``(ii) $\Longleftrightarrow$ (iii)'' establishes that approximating $\MOT$ and approximating $\MinO$ are polynomial-time equivalent. Thus we may investigate when $\MOT$ is tractable to approximate by instead investigating the more amenable question of when $\MinO$ is tractable (see \S\ref{sssec:intro-q4}) at no loss of generality.

Theorem~\ref{thm:mwumotp} is new to this work. In particular, equivalences between problems with polynomially small error do not fall under the purview of classical LP theory, which deals with exponentially small error \citep{GLSbook}. Our use of the $\MWU$ algorithm exploits a simple reduction of $\MOT$ to a mixed packing-covering LP that has appeared in the $k=2$ matrix case of Optimal Transport in~\citep{BlaJamKenSid18,Qua18}, where implicit LP is not necessary for polynomial runtime.

\paragraph{Organization of \S\ref{ssec:algs:mwu}.} In \S\ref{sssec:algs:mwu:pseudo} we present the specialization of Multiplicative Weights Update to $\MOT$, and recall how it runs in polynomial time and calls to a certain bottleneck oracle. In \S\ref{sssec:algs:mwu:equiv}, we show that this bottleneck oracle is equivalent to the $\AMinO$ oracle, and then use this to prove Theorem~\ref{thm:mwumotp}.

\subsubsection{Algorithm}\label{sssec:algs:mwu:pseudo}

Here we present the $\MWU$ algorithm, which combines the generic Multiplicative Weights Update algorithm of \citep{young2001sequential} specialized to $\MOT$, along with a final rounding step that ensures feasibility of the solution.

In order to present $\MWU$, it is convenient to assume that the cost $C$ has entries in the range $[1,2] \subset \RR$, which is at no loss of generality by simply translating and rescaling the cost (see \S\ref{sssec:algs:mwu:equiv}), and can be done implicitly given a bound on $\Cmax$. This is why in the rest of this subsection, every runtime dependence on $\eps$ is polynomial in $1/\eps$ for costs in the range $[1,2]$; after transformation back to $[-\Cmax,\Cmax]$, this is polynomial dependence in the standard scale-invariant quantity $\Cmax/\eps$. 

\par Since the cost $C$ is assumed to have non-negative entries, for any $\lambda \in [1,2]$, the polytope
\begin{align*}
	K(\lambda) = \{P \in \Coup : \langle C, P \rangle \leq \lambda\}
\end{align*}
of couplings with cost at most $\lambda$ is a mixed packing-covering polytope (i.e., all variables are non-negative and all constraints have non-negative coefficients). Note that $K(\lambda)$ is non-empty if and only if $\MOT_C(\mu)$ has value at most $\lambda$. Thus, modulo a binary search on $\lambda$, this reduces computing the value of $\MOT_C(\mu)$ to the task of detecting whether $K(\lambda)$ is empty. Since $K(\lambda)$ is a mixed packing-covering polytope, the Multiplicative Weights Update algorithm of \citep{young2001sequential} determines whether $K(\lambda)$ is empty, and runs in polynomial time apart from one bottleneck, which we now define.

In order to define the bottleneck, we first define a potential function. For this, we define the softmax analogously to the softmin as 
\[
	\smax(a_1,\ldots,a_t) = -\smin(-a_1,\ldots,-a_t) = \log\left(\sum_{i=1}^t e^{a_i}\right).
\]
Here we use regularization parameter $\eta = 1$ for simplicity, since this suffices for analyzing $\MWU$, and thus we have dropped this index $\eta$ for shorthand.

\begin{defin}[Potential function for $\MWU$]
Fix a cost $C \in \Rntk$, target marginals $\mu \in (\Delta_n)^k$, and target value $\lambda \in \RR$. Define the potential function $\Phi := \Phi_{C,\mu,\lambda} : \Rpntk \to \R$ by
$$\Phi(P)
= \smax \left(\frac{\langle C, P \rangle }{\lambda}, \frac{m_1(P) }{ \mu_1}, \ldots, \frac{m_k(P) }{ \mu_k}\right).
$$
The softmax in the above expression is interpreted as a softmax over the $nk+1$ values in the concatenation of vectors and scalars in its input. (This slight abuse of notation significantly reduces notational overhead.)
\end{defin}

Given this potential function, we now define the bottleneck operation for $\MWU$: find a direction $\vec{j} \in [n]^k$ in which $P$ can be increased such that the potential is increased as little as possible.
\begin{defin}[Bottleneck oracle for $\MWU$]
 Given iterate $P \in \Rpntk$, target marginals $\mu \in (\Delta_n)^k$, target value $\lambda \in \R$, and accuracy $\eps > 0$, $\MWUoracle_C(P,\mu,\lambda,\eps)$ either:
\begin{itemize}
\item Outputs ``null'', certifying that $\min_{\jvec \in [n]^k} \pd{}{h} \Phi(P + h \cdot \delta_{\jvec}) \mid_{h = 0} > 1$, or
\item Outputs $\jvec \in [n]^k$ such that $\pd{}{h} \Phi(P + h \cdot \delta_{\jvec}) \mid_{h = 0} \leq 1 + \eps$.
\end{itemize}
(If $\min_{\jvec \in [n]^k} \pd{}{h} \Phi(P + h \cdot \delta_{\jvec}) \mid_{h = 0}$ is within $(1,1+\eps]$, then either return behavior is a valid output.)
\end{defin}

Pseudocode for the $\MWU$ algorithm is given in Algorithm~\ref{alg:mwu}. We prove that $\MWU$ runs in polynomial time given access to this bottleneck oracle.

\begin{theorem}\label{thm:mwuhelper}
Let the entries of the cost $C$ lie in the range $[1,2]$. Given $\lambda \in \RR$ and accuracy parameter $\eps > 0$, $\MWU$ either certifies that $\MOT_C(\mu) \leq \lambda$, or returns a $\poly(n,k,1/\eps)$-sparse $P \in \Coup$ satisfying $\langle C,P \rangle \leq \lambda + 8\eps$.

Furthermore, the loop in Step 1 runs in $\tilde{O}(nk/\eps^2)$ iterations, and Step 2 runs in $\poly(n,k,1/\eps)$ time.
\end{theorem}

The $\MWU$ algorithm can be used to output a $O(\eps)$-approximate solution for $\MOT$ time via an outer loop that performs binary search over $\lambda$; this only incurs $O(\log(1 / \eps))$-multiplicative overhead in runtime.

\begin{algorithm}[h]
	\caption{\texttt{MWU}: specialization of Multiplicative Weights Update~\citep{young2001sequential} to $\MOT$}
	\label{alg:mwu}
	\begin{algorithmic}[1]
		\Require Accuracy $\eps > 0$, marginals $\mu_1, \dots, \mu_k \in \Delta_n$, target value $\lambda > 0$
		\Ensure Either certifies $\MOT_C(\mu) > \lambda$ by returning ``infeasible'', or returns a solution $P$ with $\langle C, P \rangle \leq \lambda + 8\eps$ %
		\Statex \textbackslash\textbackslash$\;$ Assume cost $C$ satisfies $C_{\vec{j}} \in [1,2]$ for all $\vec{j} \in [n]^k$ (without loss of generality by rescaling)
		\Statex \textbackslash\textbackslash$\;$ Step 1: Multiplicative Weights Update
		
		\State $P \gets 0 \in \Rpntk$, $\eta \gets 2 (\log(nk + 1)) / \eps$ \label{line:step1first}
		\While{$m(P) < \eta$} \Comment{While total mass is small}
		\State $\jvec \gets \MWUoracle_C(P,\mu,\lambda,\eps)$ \Comment{Bottleneck: find direction with small potential increase}
		\If{$\jvec = $``null''} \Return{``infeasible''} \Comment{Infeasible if no good direction}
		\Else\ $P \gets P + \delta_{\jvec} \cdot \eps \cdot \min(\lambda /C_{\jvec},\min_i [\mu_i]_{j_i})$ \Comment{Else, increase in good direction}
		\EndIf
		\EndWhile \label{line:step1penultimate}
		\State $P \gets P / (\eta(1 + \eps)^4)$  \Comment{Rescale} \label{line:step1last}
		\Statex
		\Statex \textbackslash\textbackslash$\;$ Step 2: round to transportation polytope
		\While{$m(P) < 1$} \Comment{Until all constraints are satisfied}
		\State $j_i \gets \argmax_{j \in [n]} ([\mu_i]_j - [m_i(P)]_j)$ for each $i \in [k]$ \Comment{For each marginal, find unsatisfied constraint}\label{line:sparseround:exists}
		\State $\alpha \gets \min_{i \in [k]} ([\mu_i]_{j_i} - [m_i(P)]_{j_i})$ 
		\label{line:sparseround:i}\Comment{Maximal mass to add}
		\State $P \gets P + \alpha \cdot \delta_{\jvec}$ \label{line:sparseround:sat} \Comment{Add mass to saturate at least one constraint}
		\EndWhile
	\end{algorithmic}
\end{algorithm}

\begin{proof}
We analyze Step 1 (Multiplicative Weights Update) and Step 2 (rounding) of $\MWU$ separately. 

\begin{lemma}[Correctness of Step 1] \label{lem:mwu:step1}
Step 1 of Algorithm~\ref{alg:mwu} runs in $\tilde{O}(nk /\eps^2)$ iterations. It either returns (i) ``infeasible'', certifying that $K(\lambda)$ is empty; or (ii) finds a $\poly(n,k,1/\eps)$-sparse tensor $P \in \Rpntk$ that is approximately in $K(\lambda)$, i.e., $P$ satisfies:
$$m(P) \geq 1 - 4\eps, \quad \langle C, P \rangle \leq \lambda, \quad \mbox{ and } \quad m_i(P) \leq \mu_i\mbox{ for all }i \in [k]$$
\end{lemma}
Step 1 is the Multiplicative Weights Update algorithm of \citep{young2001sequential} applied to the polytope $K(\lambda)$, so correctness follows from the analysis of \citep{young2001sequential}. 
We briefly recall the main idea behind this algorithm for the convenience of the reader, and provide a full proof in Appendix~\ref{app:pf:mwu} for completeness. The main idea behind the algorithm is that on each iteration, $\jvec \in [n]^k$ is chosen so that the increase in the potential $\Phi(P)$ is approximately bounded by the increase in the total mass $m(P)$. If this is impossible, then the bottleneck oracle returns null, which means $K(\lambda)$ is empty. So assume otherwise. Then once the total mass has increased to $m(P) = \eta + O(\eps)$, the potential $\Phi(P)$ must be bounded by $\eta (1 + O(\eps))$. By exploiting the inequality between the max and the softmax, this means that $\max(\langle C, P \rangle / \lambda, \max_{i \in [n], j \in [k]} [m_i(P)]_j / [\mu_i]_j) \leq \Phi(P) \leq \eta (1 + O(\eps))$ as well. Thus, rescaling $P$ by $1 / (\eta(1 + O(\eps)))$ in Line~\ref{line:step1last} satisfies $m(P) \geq 1 - O(\eps)$, $\langle C, P \rangle / \lambda \leq 1$, and $m_i(P)/\mu_i \leq 1$. 
See Appendix~\ref{app:pf:mwu} for full details and a proof of the runtime and sparsity claims.

\begin{lemma}[Correctness of Step 2]\label{lem:mwu:step2}
Step 2 of Algorithm~\ref{alg:mwu} runs in $\poly(n,k,1/\eps)$ time and returns $P \in \Coup$ satisfying $\langle C, P \rangle \leq \lambda + 8\eps$. Furthermore, $P$ only has $\poly(n,k,1/\eps)$ non-zero entries.
\end{lemma}
\begin{proof}[Proof of Lemma~\ref{lem:mwu:step2}]
	By Lemma~\ref{lem:mwu:step1}, $P$ satisfies $m_i(P) \leq \mu_i$ for all $i \in [k]$. Observe that this is an invariant that holds throughout the execution of Step $2$. This, along with the fact that $\sum_{j=1}^n [m_i(P)]_j = m(P)$ is equal for all $i$, implies that the indices $(j_1,\dots,j_k)$ found in Line~\ref{line:sparseround:exists} satisfy $[\mu_i]_{j_i} - [m_i(P)]_{j_i} > 0$ for each $i \in [k]$. Thus in particular $\alpha > 0$ in Line~\ref{line:sparseround:i}. It follows that Line~\ref{line:sparseround:sat} makes at least one more constraint satisfied (in particular the constraint ``$[\mu_i]_{j_i} = [m_i(P)]_{j_i}$'' where $i$ is the minimizer in Line~\ref{line:sparseround:i}). Since there are $nk$ constraints total to be satisfied, Step 2 terminates in at most $nk$ iterations. Each iteration increases the number of non-zero entries in $P$ by at most one, thus $P$ is $\poly(n,k,1/\eps)$ sparse throughout. That $P$ is sparse also implies that each iteration can be performed in $\poly(n,k,1/\eps)$ time, thus Step 2 takes $\poly(n,k,1/\eps)$ time overall. 
	
	\par Finally, we establish the quality guarantee on $\langle C, P \rangle$. By Lemma~\ref{lem:mwu:step1}, this is at most $\lambda$ before starting Step 2. During Step 2, the total mass added to $P$ is equal to $1 - m(P)$. This is upper bounded by $4\eps$ by Lemma~\ref{lem:mwu:step1}. Since $\Cmax \leq 2$, we conclude that the value of $\langle C, P \rangle$ is increased by at most $8\eps$ in Step 2. 
\end{proof}
Combining Lemmas \ref{lem:mwu:step1} and \ref{lem:mwu:step2} concludes the proof of Theorem~\ref{thm:mwuhelper}.
\end{proof}

\subsubsection{Equivalence of bottleneck to $\AMinO$}\label{sssec:algs:mwu:equiv}

In order to prove Theorem~\ref{thm:mwumotp}, we show that the $\MWU$ algorithm can be implemented in polynomial time and calls to the $\AMinO$ oracle.  First, we prove this fact for the $\ArgAMinO$ oracle, which differs from the $\AMinO$ oracle in that it also returns a tuple $\jvec \in [n]^k$ that is an approximate minimizer.

\begin{defin}[Approximate violation oracle for~\eqref{MOT-D}]
	Given weights $p = (p_1,\ldots,p_k) \in \R^{n \times k}$ and accuracy $\eps > 0$, $\ArgAMinO_C$ returns $\jvec \in [n]^k$ that minimizes $\min_{\jvec \in [n]^k} C_{\jvec} - \sum_{i=1}^k [p_i]_{j_i}$ up to additive error $\eps$, and its value up to additive error $\eps$.
\end{defin}

\begin{lemma}\label{lem:mwuargamin}
Let the entries of the cost $C$ lie in the range $[1,2]$. The $\MWU$ algorithm (Algorithm~\ref{alg:mwu}), can be implemented by $\poly(n,k,1/\eps)$ time and calls to the $\ArgAMinO_C$ oracle with accuracy parameter $\eps' = \Theta(\eps^2 / (nk))$.
\end{lemma}
\begin{proof}
We show that on each iteration of Step 1 of Algorithm~\ref{alg:mwu} we can emulate the call to the $\MWUoracle$ oracle with one call to the $\ArgAMinO$ oracle. Recall that $\MWUoracle_C(P,\mu,\lambda,\eps)$ seeks to find $\jvec \in [n]^k $ such that
\[
	V_{\jvec} := \pd{}{h} \Phi(P + h \delta_{\jvec}) \Big|_{h=0}
\]
is at most $1 + \eps$, or to certify that for all $\jvec$ it is greater than $1$. By explicit computation, 
\begin{align}
V_{\jvec} &= \pd{}{h} \log\left(\exp\left(\left(\langle C, P \rangle + h C_{\jvec}\right)/\lambda + \sum_{s=1}^k \sum_{t=1}^n \exp\left(\left([m_s(P)]_t + \delta_{t,j_s}\right)/[\mu_s]_t\right)\right) \right) \Bigg|_{h=0} \nonumber \\
&= \left(C_{\jvec} - \sum_{i=1}^k [p_i]_{j_i}\right) \frac{\exp(\langle C, P \rangle / \lambda) / \lambda}{\exp(\langle C, P \rangle / \lambda) + \sum_{s=1}^k \sum_{t=1}^n \exp([m_s(P)]_t / [\mu_s]_t)},  \label{eq:mwubottleneckquantity}
\end{align}
where the weights $p = (p_1,\ldots,p_k) \in \RR^{k \times n}$ in the last line are defined as $$[p_i]_j = -\frac{\lambda }{ \exp(\langle C, P \rangle / \lambda)} \cdot\frac{\exp([m_i(P)]_{j} / [\mu_i]_{j})}{ [\mu_i]_{j}}, \qquad \forall i \in [k], j \in [n].$$
Note that the second term in the product in \eqref{eq:mwubottleneckquantity} is positive and does not depend on $\jvec$. This suggests that in order to minimize \eqref{eq:mwubottleneckquantity}, it suffices to compute $\jvec \gets \ArgAMinO_C(p,\eps')$ for some accuracy parameter $\eps' > 0$.

The main technical difficulty with formalizing this intuitive approach is that the weights $p$ are not necessarily efficiently computable. Nevertheless, using $\poly(n,k)$ extra time on each iteration, we can compute the marginals $m_1(P),\ldots,m_k(P)$. Since the $\ArgAMinO$ oracle returns an $\eps'$-additive approximation of the cost, we can also compute a running estimate $\tilde{c}$ of the cost such that, on iteration $T$, $$\tilde{c} - T\eps' \leq \langle C, P \rangle \leq\tilde{c} + T\eps'.$$
 
Therefore, we define weights $\tilde{p} \in \RR^{n \times k}$, which approximate $p$ and which can be computed in $\poly(n,k)$ time on each iteration:
$$[\tilde{p}_i]_j = -\frac{\lambda}{\exp(\tilde{c} / \lambda)} \cdot \frac{\exp([m_i(P)]_j / [\mu_i]_j)}{[\mu_i]_j}, \qquad \forall i \in [k], j \in [n].$$
We also define the approximate value for any $\jvec \in [n]^k$:
\begin{align*}
\tilde{V}_{\jvec} :=  \left(C_{\jvec} - \sum_{i=1}^k [\tilde{p}_i]_{j_i}\right)\frac{\exp(\tilde{c}/\lambda)/\lambda}{\exp(\tilde{c}/\lambda) + \sum_{s=1}^k \sum_{t=1}^n \exp([m_s(P)]_t / [\mu_s]_t)}
\end{align*}
It holds that $\ArgAMinO_C(\tilde{p}, \eps')$ returns a $\jvec \in [n]^k$  that minimizes $C_{\jvec} - \sum_{i=1}^k [\tilde{p}_i]_j$ up to multiplicative error $1/(1-\eps')$, because the entries of the cost $C$ are lower-bounded by 1, and $[\tilde{p}_i]_j \leq 0$ for all $i \in [n], j \in [k]$. In particular, $\ArgAMinO_C(\tilde{p}, \eps')$ minimizes $\tilde{V}_{\jvec}$ up to multiplicative error $1/(1 - \eps')$. We prove the following claim relating $V_{\jvec}$ and $\tilde{V}_{\jvec}$:
\begin{claim}
For any $\jvec \in [n]^k$, on iteration $T$, we have $V_{\jvec} / \tilde{V}_{\jvec} \in [\exp(-2T\eps'/\lambda), \exp(2T\eps'/\lambda)]$.
\end{claim}

By the above claim, therefore $\ArgAMinO_C(\tilde{p},\eps')$ minimizes $V_{\jvec}$ up to multiplicative error $\exp(4T\eps'/\lambda) / (1 - \eps') \leq (1 + \eps/3)$ if we choose $\eps' = \Omega(\lambda \eps / T)$. Thus one can implement $\MWUoracle_C(p,\mu,\lambda,\eps)$ by returning the value of $\ArgAMinO_C(\tilde{p},\eps')$ if its value is estimated to be at most $1 + \eps/3$, and returning ``null`` otherwise. The bound on the accuracy  $\eps' = \tilde{\Omega}(\eps^2 / (nk))$ follows since $\lambda \in [1,2]$ follows since $\lambda \in [1,2]$ and $T = \tilde{O}(nk/\eps^2)$ by Theorem~\ref{thm:mwuhelper}.

\begin{proof}[Proof of Claim]
We compare the expressions for $V_{\jvec}$ and $\tilde{V}_{\jvec}$. Each of these is a product of two terms. Since $C_{\jvec} \geq 0$, and $[\tilde{p}_i]_{j_i}, [p_i]_{j_i} \leq 0$ for all $i$, the ratio of the first terms is
\begin{align*}\frac{C_{\jvec} - \sum_{i=1}^k [\tilde{p}_i]_{j_i}}{C_{\jvec} - \sum_{i=1}^k [p_i]_{j_i}} \in [\min_{i} [\tilde{p}_i]_{j_i} / [p_i]_{j_i}, \max_{i} [\tilde{p}_i]_{j_i} / [p_i]_{j_i}]
\subset [\exp(-T\eps'/\lambda), \exp(T\eps'/\lambda)],\end{align*}
where we have used that, for all $i \in [k]$, $$[\tilde{p}_i]_{j_i} / [p_i]_{j_i} = \exp(\langle C, P \rangle / \lambda) / \exp(\tilde{c} / \lambda) \in [\exp(-T\eps'/\lambda), \exp(T\eps'/\lambda)].$$
Similarly the ratio of the second terms in the expression for $V_{\jvec}$ and $\tilde{V}_{\jvec}$ is also in the range $[\exp(-T\eps'/\lambda), \exp(T\eps'/\lambda)]$. This concludes the proof of the claim.
\end{proof}
\end{proof}

Finally, we show that the $\ArgAMinO$ oracle can be reduced to the $\AMinO$ oracle, which completes the proof that $\MWU$ can be run with $\AMinO$.

\begin{lemma}[Equivalence of $\AMinO$ and $\ArgAMinO$]\label{lem:argaminequalsamino}
	Each of the oracles $\AMinO_C$ and $\ArgAMinO_C$ with accuracy parameter $\eps > 0$ can be implemented using $\poly(n,k)$ calls of the other oracle with accuracy parameter $\Theta(\eps / k)$ and $\poly(n,k)$ additional time.
\end{lemma}
It is worth remarking that the equivalence that we show between $\AMinO$ and $\ArgAMinO$ is \textit{not} known to hold for the feasibility and separation oracles of general LPs, since the known result for general LPs requires exponentially small error in $nk$ \citep[\S4.3]{GLSbook}. However, in the case of $\MOT$ the equivalence follows from a direct and practical reduction, similar to the proof of the equivalence of the exact oracles (Lemma~\ref{lem:minargmin}). The main difference is that some care is needed to bound the propagation of the errors of the approximate oracles. For completeness, we provide the full proof of Lemma~\ref{lem:argaminequalsamino} in Appendix~\ref{app:pfs}. 

We conclude by proving Theorem~\ref{thm:mwumotp}.
\begin{proof}[Proof of Theorem~\ref{thm:mwumotp}]
The implication ``(i) $\implies$ (ii)'' is trivial, and the implication ``(ii) $\implies$ (iii)'' is shown in~\citep{AltBoi20hard}. It therefore suffices to show the implication ``(iii) $\implies$ (i)''. For costs $C$ with entries in the range $[1,2]$, this follows from combining the fact that $\MWU$ can be implemented to solve $\MOT_C$ in $\poly(n,k,1/\eps)$ time given an efficient implementation of $\ArgAMinO_C$ with polynomially-sized accuracy parameter $\eps' = \poly(1/n,1/k,\eps)$ (Lemma~\ref{lem:mwuargamin}), along with the fact that the $\AMinO_C$ and $\ArgAMinO_C$ oracles are polynomially-time equivalent with polynomial-sized accuracy parameter (Lemma~\ref{lem:argaminequalsamino}).

The assumption that $C$ has entries within the range $[1,2]$ can be removed with no loss by translating and rescaling the original cost $C' \gets (C + 3\Cmax)/(2\Cmax)$ and running Algorithm~\ref{alg:mwu} on $C'$ with approximation parameter $\eps' \gets \eps / (2\Cmax)$. Each $\tau'$-approximate query to the $\AMinO_{C'}$ oracle can be simulated by a $\tau$-approximate query to the $\AMinO_C$ oracle, where $\tau = 2\Cmax\tau'$.
\end{proof}

\begin{remark}[Practical optimizations]
Our numerical implementation of $\MWU$ has two modifications that provide practical speedups. One is maintaining a cached list of the tuples $\jvec \in [n]^k$ previously returned by calls to $\MWUoracle$. Whenever $\MWUoracle$ is called, we first check whether any tuple $\jvec$ in the cache satisfies the desiderata $\pd{}{h} \Phi(P + h \cdot \delta_{\jvec}) \mid_{h = 0} \leq 1 + \eps$, in which case we use this $\jvec$ to answer the oracle query. Otherwise, we answer the oracle query using $\AMinO$ as explained above. In practice, this cache allows us to avoid many calls to the potentially expensive $\AMinO$ bottleneck. Our second optimization is that, at each iteration of $\MWU$, we check whether the current iterate $P$ can be rescaled in order to satisfy the guarantees in Lemma~\ref{lem:mwu:step1} required from Step 1. If so, we stop Step 1 early and use this rescaled $P$.
\end{remark}
\subsection{The Sinkhorn algorithm and the $\SMinO$ oracle}\label{ssec:algs:sink}

The Sinkhorn algorithm ($\Sink$) is specially tailored to $\MOT$, and does not apply to general exponential-size LP. Currently it is by far the most popular algorithm in the $\MOT$ literature (see \S\ref{ssec:intro:prev}). 
However, in general each iteration of $\Sink$ takes exponential time $n^{\Theta(k)}$, and it is not well-understood when it can be implemented in polynomial-time. 
The objective of this section is to show that this bottleneck is polynomial-time equivalent to the $\SMinO$ oracle, and in doing so put $\Sink$ on equal footing with classical implicit LP algorithms in terms of their reliance on variants of the dual feasibility oracle for $\MOT$. Concretely, this lets us establish the following two results.
\par First, $\Sink$ can solve $\MOT$ in polynomial time whenever $\SMinO$ can be solved in polynomial time. 

\begin{theorem}\label{thm:sink-mot:smin}
	For any family of cost tensors $C \in \Rntk$ and accuracy $\eps > 0$, $\Sink$ solves $\MOT_C$ to $\eps$ accuracy in $\poly(n,k,\Cmax/\eps)$ time and $\poly(n,k,\Cmax/\eps)$ calls to the $\SMinO_C$ oracle with regularization $\eta = (2 k \log n)/\eps$.  (The solution is output through a polynomial-size implicit representation, see \S\ref{sssec:sink:alg}.)
\end{theorem}

Second, we show that $\Sink$ requires strictly more structure than other algorithms do to solve an $\MOT$ problem. This is why the results about $\ELLIP$ (Theorem~\ref{thm:ellip}) and $\MWU$ (Theorem~\ref{thm:mwumotp}) state that those algorithms solve an $\MOT$ problem whenever possible, whereas Theorem~\ref{thm:sink-mot:smin} cannot be analogously extended to such an ``if and only if'' characterization.

\begin{theorem}\label{thm:sink-separation}
	There is a family of cost tensors $C \in \Rntk$ for which
	$\ELLIP$ solves $\MOT_C$ exactly in $\poly(n,k)$ time, however it is $\#$BIS-hard to implement a single iteration of $\Sink$ in $\poly(n,k)$ time.
\end{theorem}

\paragraph{Organization of \S\ref{ssec:algs:sink}.} In \S\ref{sssec:sink:alg}, we recall this $\Sink$ algorithm and how it depends on a certain marginalization oracle. In \S\ref{sssec:sink:oracle}, we show that this marginalization oracle is polynomial-time equivalent to the $\SMinO$ oracle, and use this to prove Theorems~\ref{thm:sink-mot:smin} and~\ref{thm:sink-separation}.

\subsubsection{Algorithm}\label{sssec:sink:alg}

Here we recall $\Sink$ and its known guarantees. To do this, we first define the following oracle. While this oracle does not have an interpretation as a dual feasibility oracle, we show below that it is polynomial-time equivalent to $\SMinO$, which is a specific type of approximate dual feasibility oracle (Remark~\ref{rem:oracles:smin-amin}).

\begin{defin}[\MARG]
	Given scalings $d = (d_1, \dots, d_k) \in \Rp^{n \times k}$, regularization $\eta > 0$, and an index $i \in [k]$, the marginalization oracle $\MARG_C(d,\eta,i)$ returns the vector $m_i((\otimes_{i'=1}^k d_{i'}) \odot \exp[-\eta C]) \in \R_{\geq 0}^{n}$.
\end{defin}

It is known that $\Sink$ can solve $\MOT$ with only polynomially many calls to this oracle~\citep{LinHoJor19}. The approximate solution that $\Sink$ computes is a fully dense tensor with $n^k$ non-zero entries, but it is output implicitly in $O(nk)$ space through ``scaling vectors'' and ``rounding vectors'', described below.

\begin{theorem}[$\Sink$ guarantee,~\citep{LinHoJor19}]\label{thm:sink-mot:marg}
	Algorithm~\ref{alg:sink} computes an $\eps$-approximate solution to $\MOT_C(\mu)$ using $\poly(n,k,\Cmax/\eps)$ calls to the $\MARG_C$ oracle with parameter $\eta = (2 k \log n)/\eps$, and $\poly(n,k,\Cmax/\eps)$ additional time. The solution is of the form
	\begin{align}
		P = 
		\left(\otimes_{i=1}^k d_i\right) \odot \exp[-\eta C] + \left(\otimes_{i=1}^k v_i\right),
		\label{eq:sink-rmot}
	\end{align}
	and is output implicitly via the scaling vectors $d_1, \dots, d_k \in \Rpn$ and rounding vectors $v_1, \dots, v_k \in \Rpn$.
\end{theorem}

\begin{algorithm}
	\caption{$\Sink$: multidimensional analog of classical Sinkhorn scaling}
	\hspace*{\algorithmicindent} \textbf{Input:} 
	Cost $C \in \Rntk$,
	marginals $\mu_1,\dots,\mu_k \in \Delta_n$  \\
	\hspace*{\algorithmicindent} \textbf{Output:} Implicit representation of tensor~\eqref{eq:sink-rmot} that is an $\eps$-approximate solution to $\MOT_C(\mu)$
	\begin{algorithmic}[1]
		\Statex \textbackslash\textbackslash$\;$ Step 1: scale
		\State $d_1, \dots, d_k \gets \bone$ and $\eta \gets (2k \log n)/\eps$ \Comment{Initialize (no scaling)}
		\For{$\poly(n,k,\Cmax/\eps)$ iterations} 		
		\label{line:sink:stop}
		\State Choose $i \in [k]$
		\Comment{Round-robin, greedily, or randomly}
		\label{line:sink:choice}
		\State $\tilde{\mu}_i \gets \MARG_C(d,\eta,i)$
		\Comment{Bottleneck: compute $i$-th marginal}	
		\label{line:sink:marg}
		\State $d_i \gets d_i \odot( \mu_i / \tilde{\mu}_i)$
		\Comment{Rescale $i$-th marginal (division is entrywise)}
		\label{line:sink:rescale}
		\EndFor
		\Statex
		\Statex \textbackslash\textbackslash$\;$ Step 2: round to transportation polytope
		\For{$i=1,\dots,k$} \Comment{Rescale each marginal to be below marginal constraints}
		\State $\tilde{\mu}_i \gets \MARG_C(d,\eta,i)$ \Comment{Bottleneck: compute $i$-th marginal}\label{line:round:marg}
		\State $d_i \gets d_i \odot \min[\bone, \mu_i/\tilde{\mu}_i]$ \Comment{Rescale $i$-th marginal (operations are entrywise)} \label{line:round:rescale}
		\EndFor
		
		\State $v_i \gets \mu_i - \MARG_C(d,\eta,i)$ for each $i \in [k]$ \Comment{Add back mass}
		\State $v_1 \gets v_1 / \|v\|_1^{k-1}$ \Comment{Rescale so that~\eqref{eq:sink-rmot} satisfies marginal constraints}
		\State Return $d_1,\dots,d_k$ and $v_1,\dots,v_k$ \Comment{Implicit representation of solution~\eqref{eq:sink-rmot}}
	\end{algorithmic}
	\label{alg:sink}
\end{algorithm}

\paragraph{Sketch of algorithm.} Full details and a proof are in~\citep{LinHoJor19}. We give a brief overview here for completeness. The main idea of $\Sink$ is to solve $\RMOT$, the entropically regularized variant of $\MOT$ described in \S\ref{ssec:prelim:rmot}. On one hand, this provides an $\eps$-approximate solution to $\MOT$ by taking the regularization parameter $\eta = \Theta(\eps^{-1} k \log n)$ sufficiently high  (Lemma~\ref{lem:reg-close}). On the other hand, solving $\RMOT$ rather than $\MOT$ enables exploiting the first-order optimality conditions of $\RMOT$, which imply that the unique solution to $\RMOT$ is the unique tensor in $\Coup$ of the form 
\begin{align}
P^* = 
(\otimes_{i=1}^k d_i^*) \odot K,
\label{eq:sink-opt}
\end{align}
where $K$ denotes the entrywise exponentiated tensor $\exp[-\eta C]$, and $d_1^*, \dots, d_k^* \in \Rp^n$ are non-negative vectors. The $\Sink$ algorithm approximately computes this solution in two steps.

\par The first and main step of Algorithm~\ref{alg:sink} is the natural multimarginal analog of the Sinkhorn scaling algorithm~\citep{Sin67}. 
It computes an approximate solution $P = (\otimes_{i=1}^k d_i) \odot K$ by finding $d_1, \dots, d_k$ such that $P$ is nearly feasible in the sense that $m_i(P) \approx \mu_i$ for each $i \in [k]$. Briefly, it does this via
alternating optimization: initialize $d_i$ to the all-ones vector $\bone \in \R^n$, and then iteratively update one $d_i$ so that the $i$-th marginal $m_i(P)$ of the current scaled iterate $P= (\otimes_{i=1}^k d_i)  \odot K$ is $\mu_i$. 
Although correcting one marginal can detrimentally affect the others, this algorithm nevertheless converges---in fact, in a polynomial number of iterations~\citep{LinHoJor19}.

\par The second step of Algorithm~\ref{alg:sink} is the natural multimarginal analog of the rounding algorithm~\citep[Algorithm 2]{AltWeeRig17}. It rounds the solution $P = (\otimes_{i=1}^k d_i) \odot K$ found in step one to the transportation polytope $\Coup$. Briefly, it performs this by scaling each marginal $m_i(P)$ to be entrywise less than the desired $\mu_i$, and then adding mass back to $P$ so that all marginals constraints are exactly satisfied. The former adjustment is done by adjusting the diagonal scalings $d_1, \dots, d_k$, and the latter adjustment is done by adding a rank-$1$ term $\otimes_{i=1}^k v_i$. 

\par Critically, note that Algorithm~\ref{alg:sink} takes polynomial time except for possibly the calls to the $\MARG_C$ oracle. In the absence of structure in the cost tensor $C$, evaluating this $\MARG_C$ oracle takes exponential time because it requires computing marginals of a tensor with $n^k$ entries.

\par We conclude this discussion with several remarks about $\Sink$.

\begin{remark}[Choice of update index in $\Sink$]\label{rem:sink-update}
	In line~\ref{line:sink:choice} there are several ways to choose update indices, all of which lead to the polynomial iteration complexity we desire. Iteration-complexity bounds are shown for a greedy choice in~\citep{LinHoJor19,Fri20}. Similar bounds can be shown for random and round-robin choices by adapting the techniques of~\citep{AltPar20bal}. These latter two choices do not incur the overhead of $k$ $\MARG$ computations per iteration required by the greedy choice, which is helpful in practice. Empirically, we observe that round-robin works quite well, and we use this in our experiments.
\end{remark}

\begin{remark}[Alternative implementations of $\Sink$]\label{rem:sink}
	For simplicity, Algorithm~\ref{alg:sink} provides pseudocode for the ``vanilla'' version of $\Sink$ as it performs well in practice and it achieves the polynomial iteration complexity we desire.
	There are several variants in the literature, including accelerated versions and first rounding small entries of the marginals---these variants have iteration-complexity bounds with better polynomial dependence on $\eps$ and $k$, albeit sometimes at the expense of larger polynomial factors in $n$~\citep{LinHoJor19,tupitsa2020multimarginal}.
\end{remark}

\begin{remark}[Output of $\Sink$ and efficient downstream tasks]\label{rem:sink-output} 
	While the output $P$ of $\Sink$ is fully dense with $n^k$ non-zero entries, its specific form~\eqref{eq:sink-rmot} enables performing downstream tasks in polynomial time. This is conditional on a polynomial-time $\MARG_C$ oracle, which is at no loss of generality since that is required for running $\Sink$ in polynomial time in the first place. The basic idea is that $P$ is a mixture of two simple distributions (modulo normalization). The first is $\left(\otimes_{i=1}^k d_i\right) \odot \exp[-\eta C]$, which is marginalizable using $\MARG_C$. The second is $\otimes_{i=1}^k v_i$, which is easily marginalizable since it is a product distribution (as the $v_i$ are non-negative). Together, this enables efficient marginalization of $P$. By recursively marginalizing on conditional distributions, this enables efficiently sampling from $P$. This in turn enables efficient estimation of bounded statistics of $P$ (e.g., the cost $\langle C, P\rangle$) by Hoeffding's inequality. 
\end{remark}

\subsubsection{Equivalence of bottleneck to $\SMinO$}\label{sssec:sink:oracle}

Although Theorem~\ref{thm:sink-mot:marg} shows that $\Sink$ solves $\MOT$ in polynomial time using the $\MARG$ oracle, this is neither sufficient to prove the implication ``(ii)$\implies$(i)'' in Theorem~\ref{thm:sink-mot:smin}, nor to prove Theorem~\ref{thm:sink-separation}. In order to prove these results, we show that $\SMinO$ and $\MARG$ are polynomial-time equivalent.

\begin{lemma}[Equivalence of $\MARG$ and $\SMinO$]\label{lem:marg-smin}
	For any regularization $\eta > 0$, each of the oracles $\MARG_C$ and $\SMinO_C$ can be implemented using $\poly(n)$ calls of the other oracle and $\poly(n,k)$ additional time.
\end{lemma}
\begin{proof}
	\underline{Reduction from $\SMinO$ to $\MARG$.} First, we show how to compute $\SMinO_C(p,\eta)$ using one call to the marginalization oracle and $O(n)$ additional time. Consider the entrywise exponentiated matrix $d = \exp[\eta p] \in \Rp^{n \times k}$, and let $\mu_1 = m_1((\otimes_{i=1}^k d_i) \odot \exp[-\eta C])$ be the answer to $\MARG_C(d,\eta,1)$. Observe that
	\begin{align*}
		-\eta^{-1} \log \left( \sum_{j_1=1}^n [\mu_1]_{j_1} \right)
		&=
		-\eta^{-1} \log \left( \sum_{j_1=1}^n \sum_{j_2,\dots,j_k \in [n]} \prod_{i=1}^k [d_i]_{j_i} e^{-\eta C_{\jvec}} \right)
		\\ &=
		-\eta^{-1} \log \left( \sum_{\jvec \in [n]^k} e^{-\eta (C_{\jvec} - \sum_{i=1}^k [p_i]_{j_i})} \right)
		\\ &=
		\smineta_{\jvec \in [n]^k} \left(C_{\jvec} - \sum_{i=1}^k [p_i]_{j_i}\right),
	\end{align*}
	where above the first step is by definition of $\mu_1$, the second step is by definition of $d$ and combining the sums, and the third step is by definition of $\smin$. We conclude that $-\eta^{-1} \log \sum_{j_1=1}^n [\mu_1]_{j_1}$ is a valid answer to $\SMinO_C(p,\eta)$. Since this is clearly computable from $\mu_1$ in $O(n)$ time, this establishes the claimed reduction.

	\par \underline{Reduction from $\MARG$ to $\SMinO$.} Next, we show that for any marginalization index $i \in [k]$ and entry $\ell \in [n]$, it is possible to compute the $\ell$-th entry of the vector $\MARG_C(d,\eta,i)$ using one call to the $\SMinO_C$ oracle and $\poly(n,k)$ additional time. Define $v \in \R^n$ to be the vector with $\ell$-th entry equal to $[d_i]_{\ell}$, and all other entries $0$. Define the matrix $p = \eta^{-1} \log[d_1,\ldots,d_{i-1},v,d_{i+1},\ldots,d_k]  \in \Rbar^{n \times k}$, where recall that $\log 0 = -\infty$ (see \S\ref{sec:prelim}). Let $s\in \R$ denote the answer to $\SMinO_C(p, \eta)$. Observe that
	\begin{align*}
		e^{-\eta s}
		=
		\sum_{\jvec \in [n]^k} e^{-\eta(C_{\jvec} - \sum_{i=1}^k [p_i]_{j_i})} 
		= 
		\sum_{\jvec \in [n]^k \; : \; \jvec_{i} = \ell} \prod_{i=1}^k [d_i]_{j_i} e^{-\eta C_{\jvec}}
		= \left[ m_i\left( (\otimes_{i=1}^k d_i) \odot \exp[-\eta C] \right) \right]_{\ell},
	\end{align*}
	where above the first step is by definition of $s$, the second step is by definition of $p$ and $v$, and the third step is by definition of the marginalization notation $m_i(\cdot)$. We conclude that $\exp(-\eta s)$ is a valid answer for the $\ell$-th entry of the vector $\MARG_C(d,\eta,i)$. This establishes the claimed reduction since we may repeat this procedure $n$ times to compute all $n$ entries of the the vector $\MARG_C(d,\eta,i)$.
	
\end{proof}

We can now conclude the proofs of the main results of \S\ref{ssec:algs:sink}.

\begin{proof}[Proof of Theorem~\ref{thm:sink-mot:smin}]
	This follows from the fact that $\Sink$ approximates $\MOT_C$ in polynomial time given a efficient implementation of $\MARG_C$ (Theorem~\ref{thm:sink-mot:marg}), combined with the fact that the $\MARG_C$ and $\SMinO_C$ oracles are polynomial-time equivalent (Lemma~\ref{lem:marg-smin}).
\end{proof}

\begin{proof}[Proof of Theorem~\ref{thm:sink-separation}]
	Consider the family of cost tensors in Lemma~\ref{lem:smin-separation} 
	for which the $\MinO_C$ oracle admits a polynomial-time algorithm, but for which the $\SMinO_C$ oracle is \#BIS-hard. 
	Then on one hand, the $\ELLIP$ algorithm solves $\MOT_C$ in polynomial time by Theorem~\ref{thm:ellip}. And on the other hand, it is \#BIS-hard to implement a single iteration of $\Sink$ because that requires implementing the $\MARG_C$ oracle, which is polynomial-time equivalent to the $\SMinO_C$ oracle by Lemma~\ref{lem:marg-smin}.
\end{proof}

\section{Application: $\MOT$ problems with graphical structure}\label{sec:graphical}

In this section, we illustrate our algorithmic framework on $\MOT$ problems with graphical structure. Although a polynomial-time algorithm is already known for this particular structure~\citep{h20gm,teh2002unified}, that algorithm computes solutions that are approximate and dense; see the related work section for details. By combining our algorithmic framework developed above with classical facts about graphical models, we show that it is possible to compute solutions that are exact and sparse in polynomial time. 

\par The section is organized as follows. In \S\ref{ssec:graphical:structure}, we recall the definition of graphical structure.
 In \S\ref{ssec:graphical:alg}, we show that the $\MinO$, $\AMinO$, and $\SMinO$ oracles can be implemented in polynomial time for cost tensors with graphical structure; from this it immediately follows that all of the $\MOT$ algorithms discussed in part 1 of this paper can be implemented in polynomial time. Finally, in \S\ref{ssec:graphical:fluid}, we demonstrate our results on the popular application of computing generalized Euler flows, which was the original motivation of $\MOT$. Numerical simulations demonstrate how the exact, sparse solutions produced by our new algorithms provide qualitatively better solutions than previously possible in polynomial time.

\subsection{Setup}\label{ssec:graphical:structure}

We begin by recalling preliminaries about undirected graphical models, a.k.a., Markov Random Fields. We recall only the relevant background; for further details we refer the reader to the textbooks~\citep{KolFri09,wainwright2008graphical}. 

\par In words, graphical models provide a way of encoding the independence structure of a collection of random variables in terms of a graph. The formal definition is as follows. 
Below, all graphs are undirected, and the notation $2^V$ means the power set of $V$ (i.e., the set of all subsets of $V$).

\begin{defin}[Graphical model structure]\label{def:gm}
	Let $\cS \subset 2^{[k]}$. The graphical model structure corresponding to $\cS$ is the graph $G_{\cS} = (V,E)$ with vertices $V = [k]$ and edges $E = \{(i,j) : i,j \in S, \text{ for some } S \in \cS\}$. 
\end{defin}

\begin{defin}[Graphical model]\label{def:mrf}
	Let $\cS \subset 2^{[k]}$.
	A probability distribution $P$ over $\{X_i\}_{i \in [k]}$ is a graphical model with structure $\cS$ if there exist functions $\{\psi_S\}_{S \in \cS}$ and normalizing constant $Z$ such that
	\[
		P\Big(\{x_i\}_{i \in [k]}\Big) = \frac{1}{Z} 
		\prod_{S \in \cS} \psi_S\Big(\{x_i\}_{i \in S}\Big).
	\]
\end{defin}

A standard measure of complexity for graphical models is the treewidth of the underlying graphical model structure $G_{\cS}$ because this captures not just the storage complexity, but also the algorithmic complexity of performing fundamental tasks such as computing the mode, log-partition function, and marginal distributions~\citep{KolFri09,wainwright2008graphical}. There are a number of equivalent definitions of treewidth~\citep{bodlaender2007treewidth}. Each requires defining intermediate combinatorial concepts. We recall here the definition that is based on the concept of a junction tree because this is perhaps the most standard definition in the graphical models community.

\begin{defin}[Junction tree, treewidth]\label{def:treewidth}
	A junction tree $T = (V_T, E_T, \{B_u\}_{u \in V_T})$ for a graph $G = (V,E)$ consists of a tree $(V_T,E_T)$ and a set of bags $\{B_u \subseteq V\}_{u \in V_T}$ satisfying:
	\begin{itemize}
		\item For each variable $i \in V$, the set of nodes $U_i = \{u \in V_T : i \in B_u\}$ induces a subtree of $T$.
		\item For each edge $e \in E$, there is some bag $B_u$ containing both endpoints of $e$.
	\end{itemize}
	The width of the junction tree is one less than the size of the largest bag, i.e., is $\max_{u \in V_T} |B_u| - 1$.
	The treewidth of a graph is the width of its minimum-width junction tree.
\end{defin}

See Figures~\ref{fig:graphical-path},~\ref{fig:graphical-window}, and~\ref{fig:graphical:cycle} for illustrated examples. 
\par We now formally recall the definition of graphical structure for $\MOT$.

\begin{figure}
	\centering
	\begin{tabular}{cc}
		\begin{tabular}{c}\includegraphics[scale=0.85]{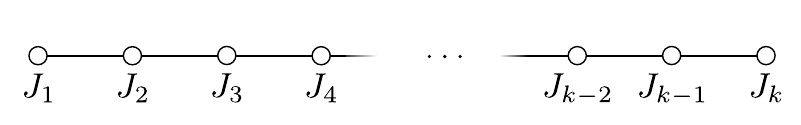}\end{tabular} &\begin{tabular}{c}\includegraphics[scale=0.85]{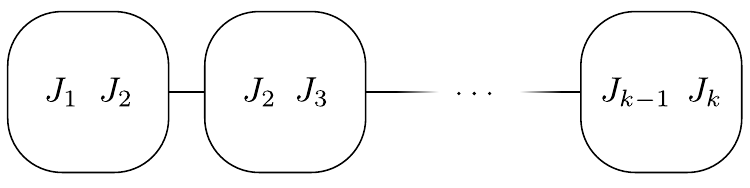} \end{tabular} \\
		\begin{tabular}{c}(a) Path graph.\end{tabular} & \begin{tabular}{c}(b) Junction tree.\end{tabular}
	\end{tabular}
	\caption{The path graph (left) has treewidth $1$ because the corresponding junction tree (right) has bags of size at most $2$.}
	\label{fig:graphical-path}
\end{figure}

\begin{figure}
	\centering
	\begin{tabular}{cc}
		\begin{tabular}{c}\includegraphics[scale=0.85]{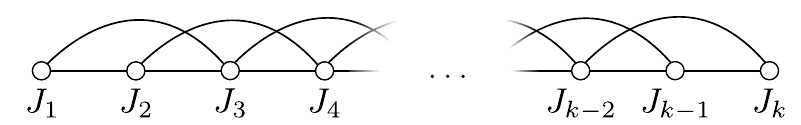}\end{tabular} &\begin{tabular}{c}\includegraphics[scale=0.85]{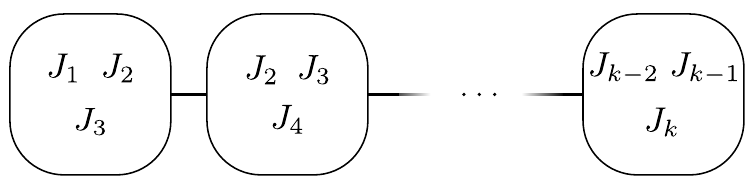} \end{tabular} \\
		\begin{tabular}{c}(a) Window graph with window size $2$.\end{tabular} & \begin{tabular}{c}(b) Junction tree.\end{tabular}
	\end{tabular}
	\caption{The graph that has an edges between all vertices of distance at most two when ordered sequentially (left) has treewidth $2$ because the corresponding junction tree (right) has bags of size at most $3$.}
	\label{fig:graphical-window}
\end{figure}

\begin{figure}
	\centering
	\begin{tabular}{cc}
		\begin{tabular}{c}\includegraphics[scale=0.85]{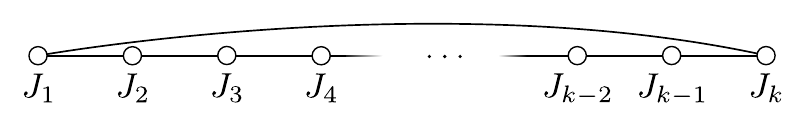}\end{tabular} &\begin{tabular}{c}\includegraphics[scale=0.85]{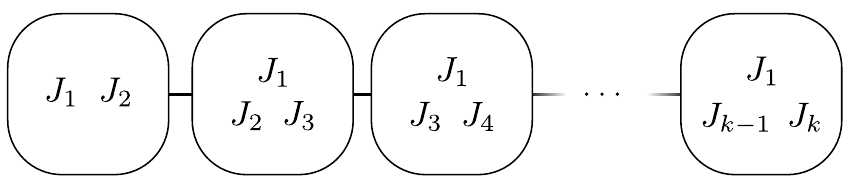} \end{tabular} \\
		\begin{tabular}{c}(a) Cycle graph.\end{tabular} & \begin{tabular}{c}(b) Junction tree.\end{tabular}
	\end{tabular}
	\caption{The cycle graph (left) has treewidth $2$ because the corresponding junction tree (right) has bags of size at most $3$.}
	\label{fig:graphical:cycle}
\end{figure}

\begin{defin}[Graphical structure for $\MOT$]\label{def:graphical}
	An $\MOT$ cost tensor $C \in \Rntk$ has graphical structure with treewidth $\omega$ if there is a graphical model structure $\cS \subset 2^{[k]}$ and functions $\{f_S\}_{S \in \cS}$ such that
	\begin{equation}\label{eq:decomposableC}
		C_{\jvec} = \sum_{S \in \cS} f_S\Big(\{j_i\}_{i \in S}\Big), 
		\qquad \forall \jvec := (j_1,\dots,j_k) \in [n]^k,
	\end{equation} 
	and such that the graph $G_{\cS}$ has treewidth $\omega$.
\end{defin}

We make three remarks about this structure. First, note that the functions $\{f_S\}_{S \in \cS}$ can be arbitrary so long as the corresponding graphical model structure has treewidth at most $\omega$. 

\par Second, if Definition~\ref{def:graphical} did not constrain the treewidth $\omega$, then every tensor $C$ would trivially have graphical structure with maximal treewidth $\omega = k -1$ (take $\cS$ to be the singleton containing $[k]$, $G_{\cS}$ to be the complete graph, and $f_{[k]}$ to be $C$). Just like all previous algorithms, our algorithms have runtimes that depend exponentially (only) on the treewidth of $G_{\cS}$. This is optimal in the sense that unless $\P = \NP$, there is no algorithm with jointly polynomial runtime in the input size and treewidth~\citep{AltBoi20hard}. We also point out that in all current applications of graphically structured $\MOT$, the treewidth is either $1$ or $2$, see \S\ref{ssec:intro:prev}.

\par Third, as in all previous work on graphically structured $\MOT$, we make the natural assumptions that the cost $C$ is input implicitly through the functions $\{f_S\}_{S \in \cS}$, and that each function $f_S$ can be evaluated in polynomial time, since otherwise graphical structure is useless for designing polynomial-time algorithms. In all applications in the literature, these two basic assumptions are always satisfied. Note also that if the treewidth of the graphical structure is constant, then there is a linear-time algorithm to compute the treewidth and a corresponding minimum-width junction tree~\citep{bodlaender1996linear}.

\subsection{Polynomial-time algorithms}\label{ssec:graphical:alg}

By our oracle reductions in \S\ref{sec:algs}, in order to design polynomial-time algorithms for $\MOT$ with graphical structure, it suffices to design polynomial-time algorithms for the $\MinO$, $\AMinO$, or $\SMinO$ oracles. This follows directly from classical algorithmic results in the graphical models literature~\citep{KolFri09}.

\begin{theorem}[Polynomial-time algorithms for the $\MinO$, $\AMinO$, and $\SMinO$ oracles for costs with graphical structure]\label{thm:graphical-oracles}
	Let $C \in \Rntk$ be a cost tensor that has graphical structure with constant treewidth $\omega$ (see Definition~\ref{def:graphical}). Then the $\MinO_C$, $\AMinO_C$, and $\SMinO_C$ oracles can be computed in $\poly(n,k)$ time.
\end{theorem}

\begin{algorithm}
	\caption{Polynomial-time algorithm for $\MinO$ for graphically structured costs (Theorem~\ref{thm:graphical-oracles}).}
	\hspace*{\algorithmicindent} \textbf{Input:} 
	Cost $C$ with graphical structure, matrix $p \in \R^{n \times k}$  \\
	\hspace*{\algorithmicindent} \textbf{Output:} Solution to $\MinO_C(p)$
	\begin{algorithmic}[1]
		\State $\jvec \gets$ mode of the graphical model $P$ in~\eqref{eq:pf-graphical-oracles:P} \Comment{Using the classical max-product algorithm~\citep[\S13.3]{KolFri09}}
		\State Return $C_{\jvec} - \sum_{i=1}^k [p_i]_{j_i}$ \Comment{Value of the $\MinO_C(p)$ oracle}
	\end{algorithmic}
	\label{alg:graphical-min}
\end{algorithm}

\begin{algorithm}
	\caption{Polynomial-time algorithm for $\SMinO$ for graphically structured costs (Theorem~\ref{thm:graphical-oracles}).}
	\hspace*{\algorithmicindent} \textbf{Input:} 
	Cost $C$ with graphical structure, matrix $p \in \Rbar^{n \times k}$, regularization $\eta > 0$  \\
	\hspace*{\algorithmicindent} \textbf{Output:} Solution to $\SMinO_C(p,\eta)$
	\begin{algorithmic}[1]
		\State $Z \gets $ partition function of the graphical model $P$ in~\eqref{eq:pf-graphical-oracles:P} \Comment{Using the classical sum-product algorithm~\citep[\S10.2]{KolFri09}}
		\State Return $-\eta^{-1} \log Z$ \Comment{Value of the $\SMinO_C(p,\eta)$ oracle}
	\end{algorithmic}
	\label{alg:graphical-smin}
\end{algorithm}

\begin{proof}
	Consider input $p$ for the oracles. Let $P$ denote the probability distribution on $[n]^k$ given by
	\begin{align}
		P(\jvec) = \frac{1}{Z} \exp\left(-\eta\left(C_{\jvec} - \sum_{i=1}^k [p_i]_{j_i}\right)\right), \qquad \forall \jvec \in [n]^k,
		\label{eq:pf-graphical-oracles:P}
	\end{align}
	where $Z = \sum_{\jvec \in [n]^{[k]}} \exp(-\eta(C_{\jvec} - \sum_{i=1}^k [p_i]_{j_i}))$ ensures $P$ is normalized. Observe that the $\MinO_C$ oracle amounts\footnote{In fact, for the purpose of computing $\MinO_C$, the distribution $P(\jvec)$ can be defined using any $\eta > 0$.} to computing the mode of the distribution $P$ because
	$\MinO_C(p) = C_{\jvec} - \sum_{i=1}^k [p_i]_{j_i}$, where $\jvec \in [n]^k$ is a maximizer of $P_{\jvec}$. Also, the $\SMinO_C$ oracle amounts to computing the partition function $Z$ because
	$\SMinO_C(p) = - \eta^{-1} \log Z$.
	Thus 
	it suffices to compute the mode and partition function of $P$ in polynomial time. (The $\AMinO_C$ oracle follows from the $\MinO_C$ oracle by Remark~\ref{rem:oracles:min-amin}).
	
	\par To this end, observe that by assumption on $C$, there is a graphical model structure $\cS \in 2^{[k]}$ and functions $\{f_S\}_{S \in \cS}$ such that the corresponding graph $G_{\cS}$ has treewidth $\omega$ and the distribution $P$ factors as
	\[
		P(\jvec)
		=
			\exp\left(-\eta\left(\sum_{S \in \cS} f_S\left( \{j_i\}_{i \in S} \right) - \sum_{i=1}^k [p_i]_{j_i}\right)\right).
	\]
	It follows that $P$ is a graphical model with respect to the same graphical model structure $\cS$ because the ``vertex potentials'' $\exp(\eta [p_i]_{j_i})$ do not affect the underlying graphical model structure. Thus $P$ is a graphical model with constant treewidth $\omega$, so we may compute the mode and partition function of $P$ in $\poly(n,k)$ time using, respectively, the classical max-product and sum-product algorithms~\citep[Chapters 13.3 and 10.2]{KolFri09}. For convenience, pseudocode summarizing this discussion is provided in Algorithms~\ref{alg:graphical-min} and~\ref{alg:graphical-smin}.
\end{proof}

An immediate consequence of Theorem~\ref{thm:graphical-oracles} combined with our oracle reductions is that all candidate $\MOT$ algorithms in \S\ref{sec:algs} can be efficiently implemented for $\MOT$ problems with graphical structure. From a theoretical perspective, $\ELLIP$ gives the best guarantee since it produces an exact, sparse solution.

\begin{cor}[Polynomial-time algorithms for $\MOT$ problems with graphical structure]\label{cor:graphical-algs}
	Let $C \in \Rntk$ be a cost tensor that has graphical structure with constant treewidth $\omega$ (see Definition~\ref{def:graphical}). Then:
	\begin{itemize}
		\item The $\ELLIP$ algorithm in \S\ref{ssec:algs:ellip} computes an exact solution to $\MOT_C$ in $\poly(n,k)$ time.
		\item The $\MWU$ algorithm in \S\ref{ssec:algs:mwu} computes an $\eps$-approximate solution to $\MOT_C$ in $\poly(n,k,\Cmax/\eps)$ time. 
		\item The $\Sink$ algorithm in \S\ref{ssec:algs:sink} computes an $\eps$-approximate solution to $\MOT_C$ in $\poly(n,k,\Cmax/\eps)$ time.
		\item The $\COLGEN$ algorithm in \S\ref{sssec:ellip:cg} can be run for $T$ iterations in $\poly(n,k,T)$ time. 
	\end{itemize}
	Moreover, $\ELLIP$, $\MWU$, and $\COLGEN$ output a polynomially sparse tensor, whereas $\Sink$ outputs a fully dense tensor through the implicit representation described in \S\ref{sssec:sink:alg}.
\end{cor}
\begin{proof}
	Combine the polynomial-time implementations of the oracles in Theorem~\ref{thm:graphical-oracles} with the polynomial-time algorithm-to-oracle reductions in Theorems~\ref{thm:ellip},~\ref{thm:mwumotp},~\ref{thm:sink-mot:smin}, and~\ref{thm:colgen}, respectively.
\end{proof}

\subsection{Application vignette: fluid dynamics}\label{ssec:graphical:fluid}

In this section, we numerically demonstrate our new results for graphically structured $\MOT$---namely the ability to compute exact, sparse solutions in polynomial time (Corollary~\ref{cor:graphical-algs}). We illustrate this on the problem of computing generalized Euler flows---an $\MOT$ application which has received significant interest and which was historically the motivation of $\MOT$, see e.g.,~\citep{BenCarCut15,BenCarNen16,Bre08,Bre89,Bre93,Bre99}. This $\MOT$ problem is already known to be tractable via a popular, specially-tailored modification of \Sink~\citep{BenCarCut15}---which can be interpreted as implementing $\Sink$ using graphical structure~\citep{h20gm,teh2002unified}. However, that algorithm is based on $\Sink$ and thus unavoidably produces solutions that are low-precision (due to $\poly(1/\eps)$ runtime dependence), fully dense (with $n^k$ non-zero entries), and have well-documented numerical precision issues. We offer the first polynomial-time algorithm for computing exact and/or sparse solutions.

\par We briefly recall the premise of this $\MOT$ problem; for further background see~\citep{Bre08,BenCarCut15}. An incompressible fluid (e.g., water) is modeled by $n$ particles which are uniformly distributed in space (due to incompressibility) at all times $t \in \{1, \dots,k+1\}$. We observe each particle's location at initial time $t=1$ and final time $t=k+1$. The task is to infer the particles' locations at all intermediate times $t \in \{2, \dots, k\}$, and this is modeled by an $\MOT$ problem as follows. 
\par Specifically, the locations of the fluid particles are discretized to points $\{x_{j}\}_{j \in [n]} \subset \R^d$, and $\sigma$ is a known permutation on this set that encodes the relation between each initial location $x_j$ at time $t=1$ and final location $\sigma(x_j)$ at time $t = k+1$. The total movement of a particle that takes the trajectory $x_{j_1},x_{j_2},\ldots,x_{j_k},\sigma(x_{j_1})$ is given by
\begin{equation}
	C_{j_1,\ldots,j_k} = \|\sigma(x_{j_1}) - x_{j_k}\|^2 + \sum_{t=1}^{k-1} \|x_{j_{t+1}} - x_{j_t}\|^2, \label{eq:Cfluiddynamicscost}
\end{equation}
By the principle of least action, the generalized Euler flow problem of inferring the most likely trajectories of the fluid particles is given by the solution to the $\MOT$ problem with this cost $C$ and uniform marginals $\mu_t = \bone_n/n \in \Delta_n$ which impose the constraint that the fluid is incompressible.

\begin{cor}[Exact, sparse solutions for generalized Euler flows]\label{cor:fluiddynamics}
	The $\MOT$ problem with cost~\eqref{eq:Cfluiddynamicscost} can be solved in $d \cdot \poly(n,k)$ time. The solution is returned as a sparse tensor with at most $nk-k+1$ non-zeros.
\end{cor}
\begin{proof}
	This cost tensor $C$ can be expressed in graphical form $C_{\jvec} = \sum_{S \in \cS} f_S(\{j_i\})$ where $\cS$ consists of the sets $\{1,2\}, \dots, \{k-1,k\}$ of adjacent time points as well as the set $\{1,k\}$. Moreover, each function $f_S : [n]^2 \to \R$ can be computed in $O(dn^2)$ time since this simply requires computing $\|x_j - x_{j'}\|^2$ for $n^2$ pairs of points $x_j,x_{j'} \in \R^d$. Once this graphical representation is computed, Corollary~\ref{cor:graphical-algs} implies a $\poly(n,k)$ time algorithm for this $\MOT$ problem because the graphical model structure $\cS$ is a cycle graph and thus has treewidth $2$ (cf., Figure~\ref{fig:graphical:cycle}).
\end{proof}

\begin{figure}
	\centering
	\begin{tabular}{c@{}c@{\hskip 0.5in}c@{}c}
	{\renewcommand{\arraystretch}{1.7}\begin{tabular}{@{}c@{}} t = 1 \\ t = 2 \\ t = 3 \\ t = 4 \\ t = 5 \\ t = 6 \\ t = 7 \end{tabular}}& \begin{tabular}{@{}c@{}} \includegraphics[scale=0.25]{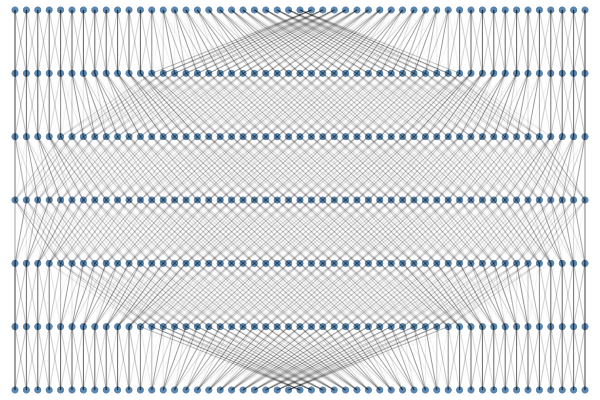} \end{tabular} & {\renewcommand{\arraystretch}{1.7}\begin{tabular}{@{}c@{}} t = 1 \\ t = 2 \\ t = 3 \\ t = 4 \\ t = 5 \\ t = 6 \\ t = 7 \end{tabular}} & \begin{tabular}{@{}c@{}}\includegraphics[scale=0.25]{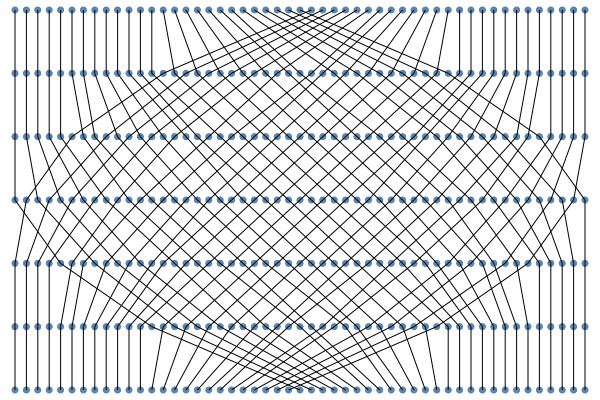}\end{tabular}
	\\
	& $\Sink$ & & $\COLGEN$ 
	\end{tabular}
	\caption{
		Transport maps computed by the fast implementation of $\Sink$~\citep{BenCarCut15} (left) and our $\COLGEN$ implementation (right) on a standard fluid dynamics benchmark problem in dimension $d = 1$ \citep{Bre08}. The pairwise transport maps between successive timesteps are plotted with opacity proportional to the mass. The $\Sink$ algorithm is run at the highest precision (i.e., smallest regularization) before serious numerical precision issues (NaNs). It returns a dense, approximate solution in 2.25 seconds.\protect\footnotemark\ $\COLGEN$ returns an exact, sparse solution in 9.52 seconds. Furthermore, in this particular problem instance, the $\COLGEN$ method returns a Monge solution, i.e., the sparsity is $n$ so that the particles never split in the computed trajectories.
	}
	\label{fig:fluids}
\end{figure}\footnotetext{All experiments in this paper are run on a standard-issue Apple MacBook Pro 2020 laptop with an M1 Chip.}

Figure~\ref{fig:fluids} illustrates how the exact, sparse solutions found by our new algorithm provide visually sharper estimates than the popular modification of $\Sink$ in~\citep{BenCarCut15}, which blurs the trajectories. The latter is the state-of-the-art algorithm in the literature and in particular is the only previously known non-heuristic algorithm that has polynomial-time guarantees. 
Note that this algorithm is identical to implementing $\Sink$ by exploiting the graphical structure to perform exact marginalization efficiently~\citep{teh2002unified,h20gm}.

\par The numerical simulation is on a standard benchmark problem used in the literature (see e.g., \citep[Figure 9]{BenCarCut15} and~\citep[Figure 2]{Bre08})
in which the particle at initial location $x \in [0,1]$ moves to final location $\sigma(x) = x + \half \pmod{1}$. This is run with $k = 6$ and marginals $\mu_1 = \dots = \mu_k$ uniformly supported on $n = 51$ positions in $[0,1]$. See Appendix~\ref{app:exp} for numerics on other standard benchmark instances. 
Note that this amounts to solving an $\MOT$ LP with $n^k = 51^6 \approx 1.8 \times 10^{10}$ variables, which is infeasible for standard LP solvers. Our algorithm is the first to compute exact solutions for problem instances of this scale. 

\par Two important remarks. First, since this $\MOT$ problem is a discretization of the underlying PDE, an exact solution is of course not necessary; however, there is an important---even qualitative---difference between low-precision solutions (computable with $\poly(1/\eps)$ runtime) and high-precision solutions (computable with $\polylog(1/\eps)$ runtime) for the discretized problem. Second, a desirable feature of $\Sink$ that should be emphasized is its practical scalability, which might make it advantageous for problems where very fine discretization is required. It is an interesting direction of practical relevance to develop algorithms that can compute high-precision solutions at a similarly large scale in practice (see the discussion in \S\ref{sec:discussion}).

\section{Application: $\MOT$ problems with set-optimization structure}\label{sec:binary}

In this section, we consider $\MOT$ problems whose cost tensors $C$ take values $0$ and $1$---or more generally any two values, by a straightforward reduction\footnote{If $C$ takes two values $a < b$, then define the tensor $\tilde{C}$ with $\{0,1\}$-entries by $\tilde{C}_{\jvec} = (C_{\jvec} - a)/(b-a)$. It is straightforward to see that the $\MOT$ problems with costs $C$ and $\tilde{C}$ have identical solutions.}.
Such $\MOT$ problems arise naturally in applications where one seeks to minimize or maximize the probability that some event occurs given marginal probabilities on each variable (see Example~\ref{ex:binary}). We establish that this general class of $\MOT$ problems can be solved in polynomial time under a condition on the sparsity pattern of $C$ that is often simple to check due its connection to classical combinatorial optimization problems.
\par The section is organized as follows. In \S\ref{ssec:binary:setup} we formally describe this setup
and discuss why it is incomparable to all other structures discussed in this paper. In \S\ref{ssec:binary:alg}, we show that for costs with this structure, the $\MinO$, $\AMinO$, and $\SMinO$ oracles can be implemented in polynomial time; from this it immediately follows that the $\ELLIP$, $\MWU$, $\Sink$, and $\COLGEN$ algorithms discussed in part 1 of this paper can be implemented in polynomial time. In \S\ref{ssec:binary:rel}, we illustrate our results via a case study on network reliability.

\subsection{Setup}\label{ssec:binary:setup}

\begin{example}[Motivation for binary-valued $\MOT$ costs: minimizing/maximizing probability of an event]\label{ex:binary}
	Let $S \subset [n]^k$. If $C_{\jvec} = \mathds{1}[\jvec \in S]$, then the $\MOT_C$ problem amounts to minimizing the probability that event $S$ occurs, given marginals on each variable. On the other hand, if $C_{\jvec} = \mathds{1}[\jvec \notin S]$, then the $\MOT_C$ problem amounts to maximizing the probability that event $S$ occurs since
	\[
		\MOT_C(\mu_1,\dots,\mu_k) 
		=
		\min_{P \in \Coup} \Prob_{\jvec \sim P}[\jvec \notin S]
		=
		1 - \max_{P \in \Coup} \Prob_{\jvec \sim P}[\jvec \in S].
	\] 
\end{example}

Even if the cost is binary-valued, there is no hope to solve $\MOT$ in polynomial time without further assumptions---essentially because in the worst case, any algorithm must query all $n^k$ entries if $C$ is a completely arbitrary $\{0,1\}$-valued tensor. 
\par We show that $\MOT$ is polynomial-time solvable under the general and often simple-to-check condition that the $\MinO$, $\AMinO$, and $\SMinO$ oracles introduced in \S\ref{sec:oracles} are polynomial-time solvable when restricted to the set $S$ of indices $\jvec \in [n]^k$ for which $C_{\jvec} = 0$. For simplicity, our definition of these set oracles removes the cost $C_{\jvec}$ as it is constant on $S$. Of course it is also possible to remove the negative sign in $-p$ by re-parameterizing the inputs as $w = -p$; however, we keep this notation in order to parallel the original oracles.

\begin{defin}[$\MinO$ set oracle]\label{def:oracle-min-S}
	Let $S \subset [n]^k$. For weights $p = (p_1,\ldots,p_k) \in \R^{n \times k}$, $\MinOCS(p)$ returns
	\[
	\min_{\jvec \in S} - \sum_{i=1}^k [p_i]_{j_i}.
	\]
\end{defin}

\begin{defin}[$\AMinO$ set oracle]\label{def:oracle-amin-S}
	Let $S \subset [n]^k$. For weights $p = (p_1,\ldots,p_k) \in \R^{n \times k}$ and accuracy $\eps > 0$, $\AMinOCS(p,\eps)$ returns $\MinOCS(p)$ up to additive error $\eps$.
\end{defin}

\begin{defin}[$\SMinO$ set oracle]\label{def:oracle-smin-S}
	Let $S \subset [n]^k$. For weights $p = (p_1,\ldots,p_k) \in \Rbar^{n \times k}$ and regularization parameter $\eta > 0$, $\SMinOCS(p, \eta)$ returns 
	\[
	\displaystyle\smineta_{\jvec \in S} - \sum_{i=1}^k [p_i]_{j_i}.
	\]
\end{defin}

The key motivation behind these set oracle definitions (aside from the syntactic similarity to the original oracles) is that they encode the problem of (approximately) finding the min-weight object in $S$. This opens the door to combinatorial applications of $\MOT$ because finding the min-weight object in $S$ is well-known to be polynomial-time solvable for many ``combinatorial-structured'' sets $S$ of interest---e.g., the set $S$ of cuts in a graph, or the set $S$ of independent sets in a matroid. See \S\ref{ssec:binary:rel} for fully-detailed applications.

\begin{defin}[Set-optimization structure for $\MOT$]\label{def:binary}
	An $\MOT$ cost tensor $C \in \Rntk$ has exact, approximate, or soft set-optimization structure of complexity $\beta$ if
	\[
		C_{\jvec} = \mathds{1}[\jvec \notin S]
	\]
	for a set $S \subset [n]^k$ for which there is an algorithm solving $\MinOCS$, $\AMinOCS$, or $\SMinOCS$, respectively, in $\beta$ time.
\end{defin}

We make two remarks about this structure.

\begin{remark}[Only require set oracle for $C^{-1}(0)$, not for $C^{-1}(1)$]
 	Note that Definition~\ref{def:binary} only requires the set oracles for the set $S$ of entries where $C$ is $0$, and does \emph{not} need the set oracles for the set $[n]^k \setminus S$ where $C$ is $1$. The fact that both set oracles are not needed makes set-optimization structure easier to check than the original oracles in \S\ref{sec:oracles}, because those effectively require optimization over both $S$ and $[n]^k \setminus S$.
\end{remark}

\begin{remark}[Set-optimization structure is incomparable to graphical and low-rank plus sparse structure]\label{rem:binary-incomparable}
	Costs $C$ that satisfy Definition~\ref{def:binary} in general do not have non-trivial graphical structure or low-rank plus sparse structure. Specifically, there are costs $C$ that satisfy Definition~\ref{def:binary}, yet require maximal $k-1$ treewidth to model via graphical structure, and super-constant rank or exponential sparsity to model via low-rank plus sparse structure. (A concrete example is the network reliability application in \S\ref{ssec:binary:rel}.) 
	Because of the $\NP$-hardness of $\MOT$ problems with $(k-1)$-treewidth graphical structure or super-constant rank~\citep{AltBoi20hard}, simply modeling such problems with graphical structure or low-rank plus rank structure is therefore useless for the purpose of designing polynomial-time $\MOT$ algorithms. 
\end{remark}

\subsection{Polynomial-time algorithms}\label{ssec:binary:alg}

By our oracle reductions in part 1 of this paper, in order to design polynomial-time algorithms for $\MOT$ with set-optimization structure, it suffices to design polynomial-time algorithms for the $\MinO$, $\AMinO$, or $\SMinO$ oracles. We show how to do this for all three oracles in a straightforward way by exploiting the set-optimization structure. 

\begin{theorem}[Polynomial-time algorithms for the $\MinO$, $\AMinO$, and $\SMinO$ oracles for costs with set-optimization structure]\label{thm:binary-oracles}
	If $C \in \Rntk$ is a cost tensor with exact, approximate, or soft set-optimization structure of complexity $\beta$ (see Definition~\ref{def:binary}), then the $\MinO_C$, $\AMinO_C$, and $\SMinO_C$ oracles, respectively, can be computed in $\beta +  \poly(n,k)$ time.
\end{theorem}

\begin{algorithm}
	\caption{Polynomial-time algorithm for $\MinO$ for costs with exact set-optimization structure (Theorem~\ref{thm:binary-oracles}).}
	\hspace*{\algorithmicindent} \textbf{Input:} 
	Access to $C$ via $\MinOCS$ oracle, matrix $p \in \R^{n \times k}$  \\
	\hspace*{\algorithmicindent} \textbf{Output:} Solution to $\MinO_C(p)$
	\begin{algorithmic}[1]
		\State $a \gets \MinOCS(p)$ \Comment{One oracle call}
		\State $x \gets -\sum_{i=1}^k \max_{j \in [n]} [p_i]_{j}$ \Comment{Takes $O(nk)$ time}
		\State Return $a$ if $a \leq x$, or $\min(a,1+x)$ otherwise \Comment{Takes $O(1)$ time}
	\end{algorithmic}
	\label{alg:binary-min}
\end{algorithm}

\begin{algorithm}
	\caption{Polynomial-time algorithm for $\SMinO$ for costs with soft set-optimization structure (Theorem~\ref{thm:binary-oracles}).}
	\hspace*{\algorithmicindent} \textbf{Input:} 
	Access to $C$ via $\SMinOCS$ oracle, matrix $p \in \Rbar^{n \times k}$, regularization $\eta$  \\
	\hspace*{\algorithmicindent} \textbf{Output:} Solution to $\SMinO_C(p,\eta)$
	\begin{algorithmic}[1]
		\State $a \gets \exp(-\eta \cdot \SMinOCS(p))$ \Comment{One oracle call}
		\State $x \gets \prod_{i=1}^k \sum_{j=1}^n \exp(\eta [p_i]_{j_i})$ \Comment{Takes $O(nk)$ time}
		\State Return $-\eta^{-1} \log(e^{-\eta}x + (1-e^{-\eta})a)$ \Comment{Takes $O(1)$ time}
	\end{algorithmic}
	\label{alg:binary-smin}
\end{algorithm}

\begin{proof}
	\underline{Polynomial-time algorithm for $\MinO$.} We first claim that Algorithm~\ref{alg:binary-min} implements the $\MinO_C(p)$ oracle.
	To this end, define
	\begin{align}
		a := \MinOCS(p) = \min_{\substack{\jvec \in [n]^k \\ \text{s.t. } C_{\jvec} = 0}} - \sum_{i=1}^k [p_i]_{j_i}
		\qquad \text{and} \qquad
		b := \min_{\substack{\jvec \in [n]^k \\ \text{s.t. } C_{\jvec} = 1}} - \sum_{i=1}^k [p_i]_{j_i}.
		\label{eq-bin-oracle:min:ab}
	\end{align}
	By re-arranging the sum and max, it follows that
	\begin{align}
	x 
	:=
	-\sum_{i=1}^k \max_{j \in [n]} [p_i]_{j}
	=
	- \max_{\jvec \in [n]^k} \sum_{i=1}^k [p_i]_{j_i}
	=
	\min_{\jvec \in [n]^k} \sum_{i=1}^k - [p_i]_{j_i}
	=
	\min(a,b).
		\label{eq-bin-oracle:min:x}
	\end{align}
	Therefore
	\begin{align}
		\MinO_C(p)
		=
		\min_{\jvec \in [n]^k} C_{\jvec} - \sum_{i=1}^k [p_i]_{j_i}
		=
		\min(a,1+b)
		&=
		\begin{cases}
			a & \text{ if } a \leq b \\
			\min(a,1+\min(a,b)) & \text{ if } a > b
		\end{cases}
		\nonumber
		\\ &=
		\begin{cases}
			a & \text{ if } a \leq x \\
			\min(a,1+x) & \text{ if } a > x
		\end{cases},
			\label{eq-bin-oracle:min:min}
	\end{align}
	where above the first step is by definition of $\MinO_C$; the second step is by partitioning the minimization over $\jvec \in [n]^k$ into $\jvec$ such that $C_{\jvec} = 0$ or $C_{\jvec} = 1$, and then plugging in the definitions of $a$ and $b$; the third step is by manipulating $\min(a,1+b)$ in both cases; and the last step is because $x = \min(a,b)$ as shown above. We conclude that Algorithm~\ref{alg:binary-min} correctly outputs $\MinO_C(p)$. Since the algorithm uses one call to the $\MinOCS$ oracle and $O(nk)$ additional time, the claim is proven.

	\par \underline{Polynomial-time algorithm for $\AMinO$.} Next, we claim that the same Algorithm~\ref{alg:binary-min}, now run with the approximate oracle $\AMinOCS(p,\eps)$ in the first step instead of the exact oracle $\MinOCS(p)$, computes a valid solution to $\AMinO_C(p,\eps)$. To prove this, let $a$, $b$, and $x$ be as defined in~\eqref{eq-bin-oracle:min:ab} and~\eqref{eq-bin-oracle:min:x} for the $\MinO$ analysis, and let $\tilde{a} = \AMinOCS(p,\eps)$. By the same logic as in~\eqref{eq-bin-oracle:min:min}, except now reversed, the output
	\begin{align*}
		\begin{cases}
			\tilde{a} & \text{ if } \tilde{a} \leq x \\
			\min(\tilde{a},1+x) & \text{ if } \tilde{a} > x
		\end{cases}
	\end{align*}
	is equal to $\min(\tilde{a},1+b)$. Now because $\tilde{a}$ is within additive $\eps$ error of $a$ (by definition of the $\AMinOCS$ oracle), it follows that the above output is within $\eps$ additive error of
	\[
		\min(a,1+b) 
		=
		\min_{\jvec \in [n]^k} C_{\jvec} - \sum_{i=1}^k [p_i]_{j_i}
		=	\MinO_C(p).
	\]
	Thus the output is a valid answer to $\AMinO_C(p,\eps)$, establishing correctness. The runtime claim is obvious.
	
	\par \underline{Polynomial-time algorithm for $\SMinO$.} Finally, we claim that Algorithm~\ref{alg:binary-smin} implements the $\SMinO_C(p,\eta)$ oracle. To this end, define
	\[
		a := e^{-\eta \cdot \SMinOCS(p,\eta)} = \sum_{\substack{\jvec \in [n]^k \\ \text{s.t. } C_{\jvec = 0}}} e^{\eta \sum_{i=1}^k [p_i]_{j_i}}
		\qquad \text{ and } \qquad
		b := 
		\sum_{\substack{\jvec \in [n]^k \\ \text{s.t. } C_{\jvec = 1}}} e^{\eta \sum_{i=1}^k [p_i]_{j_i}}.
	\]
	By re-arranging products and sums, it follows that
	\[
		x
		:=
		\prod_{i=1}^k \sum_{j=1}^n e^{\eta [p_i]_{j_i}}
		=
		\sum_{\jvec \in [n]^k} \prod_{i=1}^k e^{\eta [p_i]_{j_i}}
		= a + b.
	\]
	Therefore
	\[
		\SMinO_C(p,\eta)
		=
		-\frac{1}{\eta} \log \left( 
			\sum_{\jvec \in [n]^k} e^{-\eta (C_{\jvec} - \sum_{i=1}^k [p_i]_{j_i})}
		\right) 
		=
		-\frac{1}{\eta} \log \left( a + e^{-\eta}b \right)
		=
		-\frac{1}{\eta} \log \left( e^{-\eta} x + (1-e^{-\eta})a \right),
	\]
	where above the first step is by definition of $\SMinO_C$; the second step is by partitioning the sum over $\jvec \in [n]^k$ into $\jvec$ such that $C_{\jvec} = 0$ or $C_{\jvec} = 1$, and then plugging in the definitions of $a$ and $b$; and the third step is because $x = a+b$ as shown above. We conclude that Algorithm~\ref{alg:binary-smin} correctly outputs $\SMinO_C(p,\eta)$. Since the algorithm uses one call to the $\SMinOCS$ oracle and $O(nk)$ additional time, the claim is proven.
\end{proof}

An immediate consequence of Theorem~\ref{thm:binary-oracles} combined with our oracle reductions is that all of the candidate $\MOT$ algorithms described in \S\ref{sec:algs} can be efficiently implemented for $\MOT$ problems with set-optimization structure. From a theoretical perspective, the $\ELLIP$ algorithm gives the best guarantee since it produces an exact, sparse solution in polynomial time.

\begin{cor}[Polynomial-time algorithms for $\MOT$ problems with set-optimization structure]\label{cor:binary-algs}
	Let $C \in \Rntk$ be a cost tensor that has set-optimization structure with $\poly(n,k)$ complexity (see Definition~\ref{def:binary}).
	\begin{itemize}
		\item \underline{Exact set-optimization structure.} The $\ELLIP$ algorithm in \S\ref{ssec:algs:ellip} computes an exact solution to $\MOT_C$ in $\poly(n,k)$ time. Also, the $\COLGEN$ algorithm in \S\ref{sssec:ellip:cg} can be run for $T$ iterations in $\poly(n,k,T)$ time. 
		\item \underline{Approximate set-optimization structure.} The $\MWU$ algorithm in \S\ref{ssec:algs:mwu} computes an $\eps$-approximate solution to $\MOT_C$ in $\poly(n,k,\Cmax/\eps)$ time. 
		\item \underline{Soft set-optimization structure.} The $\Sink$ algorithm in \S\ref{ssec:algs:sink} computes an $\eps$-approximate solution to $\MOT_C$ in $\poly(n,k,\Cmax/\eps)$ time. 
	\end{itemize}
	Moreover, $\ELLIP$, $\MWU$, and $\COLGEN$ output a polynomially sparse tensor, whereas $\Sink$ outputs a fully dense tensor through the implicit representation described in \S\ref{sssec:sink:alg}.
\end{cor}
\begin{proof}
	Combine the polynomial-time implementations of the oracles in Theorem~\ref{thm:lr-oracles} with the polynomial-time algorithm-to-oracle reductions in Theorems~\ref{thm:ellip},~\ref{thm:colgen},~\ref{thm:mwumotp}, and~\ref{thm:sink-mot:smin}, respectively.
\end{proof}

\subsection{Application vignette: network reliability with correlations}\label{ssec:binary:rel}

In this section, we illustrate this class of $\MOT$ structures via an application to network reliability, a central topic in network science, engineering, and operations research, see e.g., the textbooks~\citep{gertsbakh2011network,ball1995network,ball1986computational}. The basic network reliability question is: given an undirected graph $G = (V,E)$ where each edge $e \in E$ is reliable with some probability $q_e$ and fails with probability $1-q_e$, what is the probability that all vertices are reachable from all others? This connectivity is desirable in applications, e.g., if $G$ is a computer cluster, the vertices are the machines, and the edges are communication links, then connectivity corresponds to the reachability of all machines. See the aforementioned textbooks for many other applications.

\par Of course, the above network reliability question is not yet well-defined since the edge failures are only prescribed up to their marginal distributions. \emph{Which joint distribution} greatly impacts the answer.

\par The most classical setup posits that edge failures are independent~\citep{moore1956reliable}.
Denote the network reliability probability for this setting by $\ppind$. This quantity $\ppind$ is $\#$P-complete~\citep{valiant1979complexity,provan1983complexity} and thus $\NP$-hard to compute, but there exist fully polynomial randomized approximation schemes (a.k.a.\,FPRAS) for multiplicatively approximating both the connection probability $\ppind$~\citep{karger2001randomized} and the failure probability $1 - \ppind$~\citep{guo2019polynomial}.

\par Here we investigate the setting of \emph{coordinated} edge failures, which dates back to the 1980s~\citep{weiss1986stochastic,zemel1982polynomial}. This coordination may optimize for disconnection (e.g., by an adversary), or for connection (e.g., maximize the time a network is connected while performing maintenance on each edge $e$ during $1 - q_e$ fraction of the time). We define these notions below; see also Figure~\ref{fig:networkrel} for an illustration. Below, $\Ber(q_e)$ denotes the Bernoulli distribution with parameter $q_e$. 

\begin{defin}[Network reliability with correlations]
	For an undirected graph $G = (V,E)$ and edge reliability probabilities $\{q_e\}_{e \in E}$:
	\begin{itemize}
		\item The \emph{worst-case network reliability} is 
		\[
		\ppmin := 
		\min_{P \in \cM(\{\Ber(q_e)\}_{e \in E})} \Prob_{H \sim P}\left[H\text{ is a connected subgraph of } G\right].
		\]
		\item The \emph{best-case network reliability} is
		\[
		\ppmax := \max_{P \in \cM(\{\Ber(q_e)\}_{e \in E})} \Prob_{H \sim P}[H\text{ is a connected subgraph of }G].
		\]
	\end{itemize}
\end{defin}

\begin{figure}
	\centering
	\includegraphics[scale=0.08]{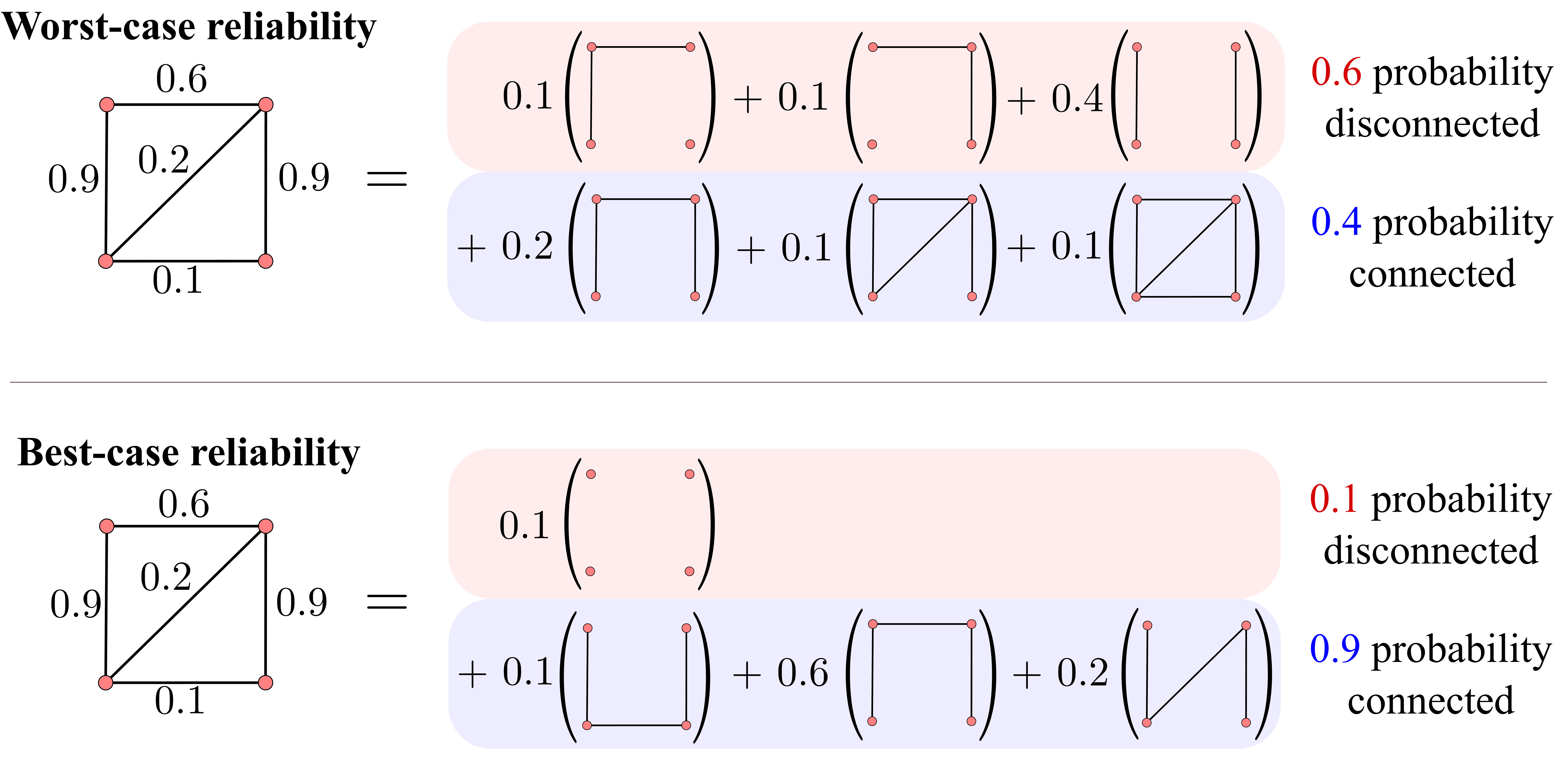}
	\caption{Optimal decompositions for the the worst-case (top) and best-case (bottom) reliability problems on the same graph $G$ and edge reliability probabilities $q_e$ (left). Coordinating edge failures yields significantly different connection probabilities: $\ppmin = 40\%$, $\ppind \approx 60\%$, and $\ppmax = 90\%$.
		}
	\label{fig:networkrel}
\end{figure}

Clearly $\ppmin \leq \ppind \leq \ppmax$. These gaps can be large (e.g., see Figure~\ref{fig:networkrel}), which promises large opportunities for applications in which coordination is possible. However, in order to realize such an opportunity requires being able to compute $\ppmin$ and $\ppmax$, and both of these problems require solving an exponentially large LP with $2^{|E|}$ variables. Below we show how to use set-optimization structure to compute these quantities in $\poly(|E|)$ time, thereby recovering as a special case of our general framework the known polynomial-time algorithms for this particular problem in~\citep{zemel1982polynomial,weiss1986stochastic}, as well as more practical polynomial-time algorithms that scale to input sizes that are an order-of-magnitude larger.

\begin{cor}[Polynomial-time algorithm for network reliability with correlations]\label{cor:rel}
	The worst-case and best-case network reliability can both be computed 
	in $\poly(|E|)$ time.
\end{cor}
\begin{proof}
	By the observation in Example~\ref{ex:binary},
	the optimization problems defining $\ppmin$ and $1 - \ppmax$ are instances of $\MOT$ in which $k=|E|$, $n=2$, $\mu_e = \Ber(q_e)$, and each entry of the cost $C \in \{0,1\}^{|E|}$ is the indicator of whether that subset of edges is a connected or disconnected subgraph of $G$, respectively. It therefore suffices to show that both of these $\MOT$ cost tensors satisfy exact set-optimization structure (Definition~\ref{def:binary}) since that implies a polynomial-time algorithm for exactly solving $\MOT$ (Corollary~\ref{cor:binary-algs}).

	\paragraph{Set-optimization structure for $1-\ppmax$.} In this case, $S$ is the set of connected subgraphs of $G$. Thus the $\MinOCS$ problem is: given weights $p \in \RR^{2 \times |E|}$, compute
	\[
	\min_{\text{connected subgraph } H \text{ of }G} 
	- \sum_{e \in H} p_{2,e} - \sum_{e \notin H} p_{1,e}.
	\]
	Note that this objective is equal to $\sum_{e \in H} x_e - \sum_{e \in E} p_{1,e}$ where $x_e := p_{1,e} - p_{2,e}$. Since the latter sum is independent of $H$, the $\MinOCS$ problem therefore reduces to the problem of finding a minimum-weight connected subgraph in $G$; that is, given edge weights $x \in \R^{|E|}$, compute
	\begin{align}
		\min_{\text{connected subgraph } H \text{ of }G} \sum_{e \in H} x_e.
		\label{eq-pf:ppmax}
	\end{align}
	We first show how to solve this in polynomial time in the case that all edge weights $x_e$ are positive. In this case, the optimal solution $H$ is a minimum-weight spanning tree of $G$. This can be found by Kruskal's algorithm in $O(|E| \log |E|)$ time~\citep{kruskal1956shortest}. 
	\par For the general case of arbitrary edge weights, note that the edges $e$ with non-positive weight $x \leq 0$ can be added to any solution without worsening the cost or feasibility. Thus these edges are without loss of generality in every solution $H$, and so it suffices to solve the same problem~\eqref{eq-pf:ppmax} on the graph $G'$ obtained by contracting these non-positively-weighted edges in $G$. This reduces~\eqref{eq-pf:ppmax} to the same problem of finding a minimum-weight connected subgraph, except now in the special case that all edge weights are positive. Since we have already shown how to solve this case in polynomial time, the proof is complete.
	
	\paragraph{Set-optimization structure for $\ppmin$.} In this case, $S$ is the set of disconnected subgraphs of $G$. We may simplify the $\MinOCS$ problem for $\ppmin$ by re-parameterizing the input $p \in \R^{2 \times |E|}$ to edge weights $x \in \R^{|E|}$ as done above in~\eqref{eq-pf:ppmax} for $1-\ppmax$. Thus the $\MinOCS$ problem for $\ppmin$ is: given weights $x \in \R^{|E|}$, compute
	\begin{align}
		\min_{\text{disconnected subgraph } H \text{ of }G} \sum_{e \in H} x_e.
		\label{eq-pf:ppin}
	\end{align}
	We first show how to solve this in the case that all edge weights $x_e$ are negative. In that case, the optimal solution is of the form $H = E \setminus C$, where $C$ is a maximum-weight cut of the graph $G$ with weights $x_e$. Equivalently, by negating all edge weights, $C$ is a minimum-weight cut of the graph $G$ with weights $-x_e$. Since a minimum-weight cut of a graph with positive weights can be found in polynomial time~\citep{stoer1997simple}, the problem~\eqref{eq-pf:ppin} can be solved in polynomial time when all $x_e$ are negative.
	\par Now in the general case of arbitrary edge weights, note that the edges $e$ with non-negative weight $x \geq 0$ can be removed from any solution without worsening the cost or feasibility. Thus these edges are without loss of generality not in every solution $H$, and so it suffices to solve the same problem~\eqref{eq-pf:ppin} on the graph $G'$ obtained by deleting these non-negatively-weighted edges in $G$. This reduces~\eqref{eq-pf:ppin} to the same problem of finding a minimum-weight disconnected subgraph, except now in the special case that all edge weights are negative. Since we have already shown how to solve this case in polynomial time, the proof is complete.
\end{proof}

\begin{figure}
	\centering
	\begin{tabular}{c@{}c}
			\includegraphics[scale=0.5]{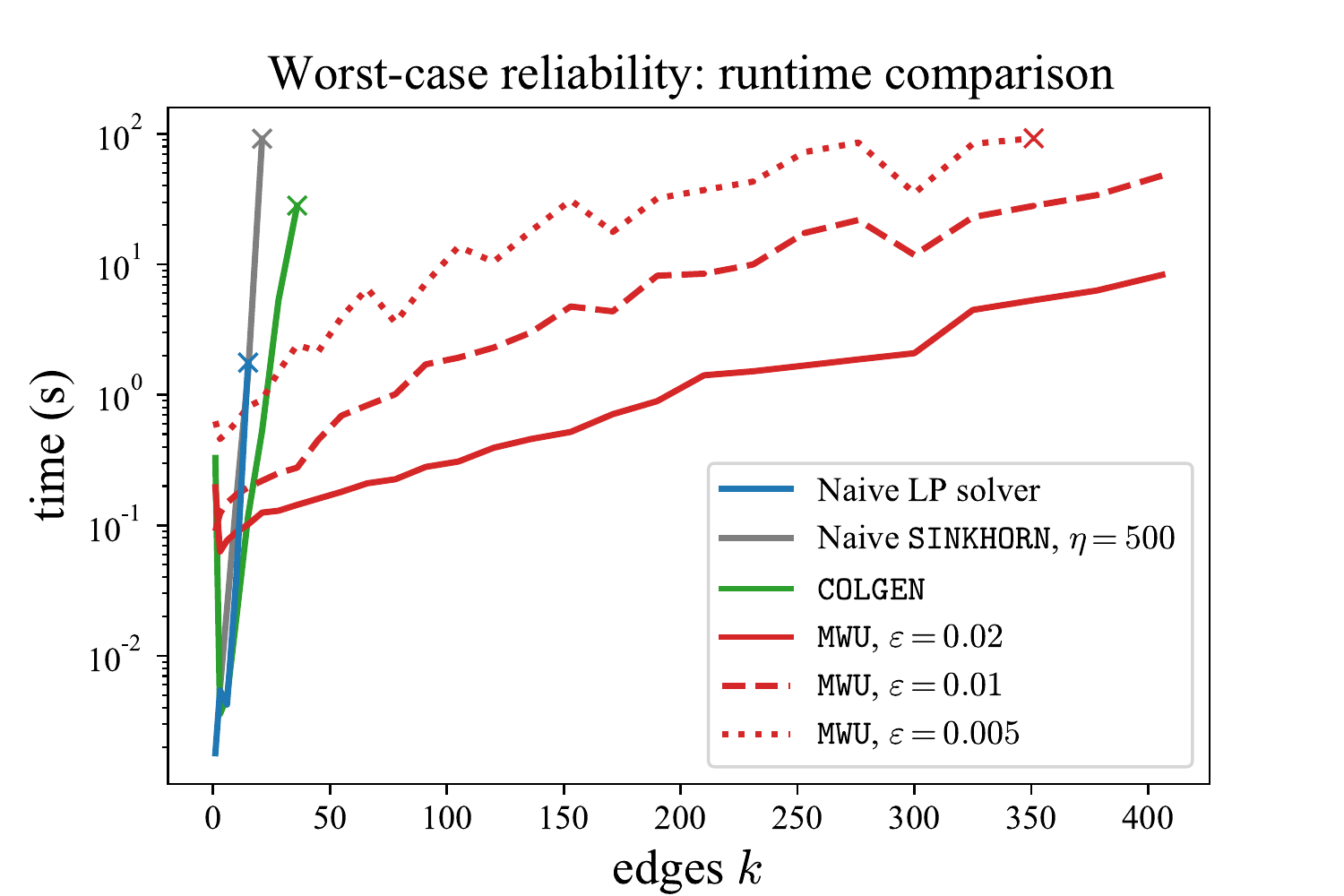} & \includegraphics[scale=0.5]{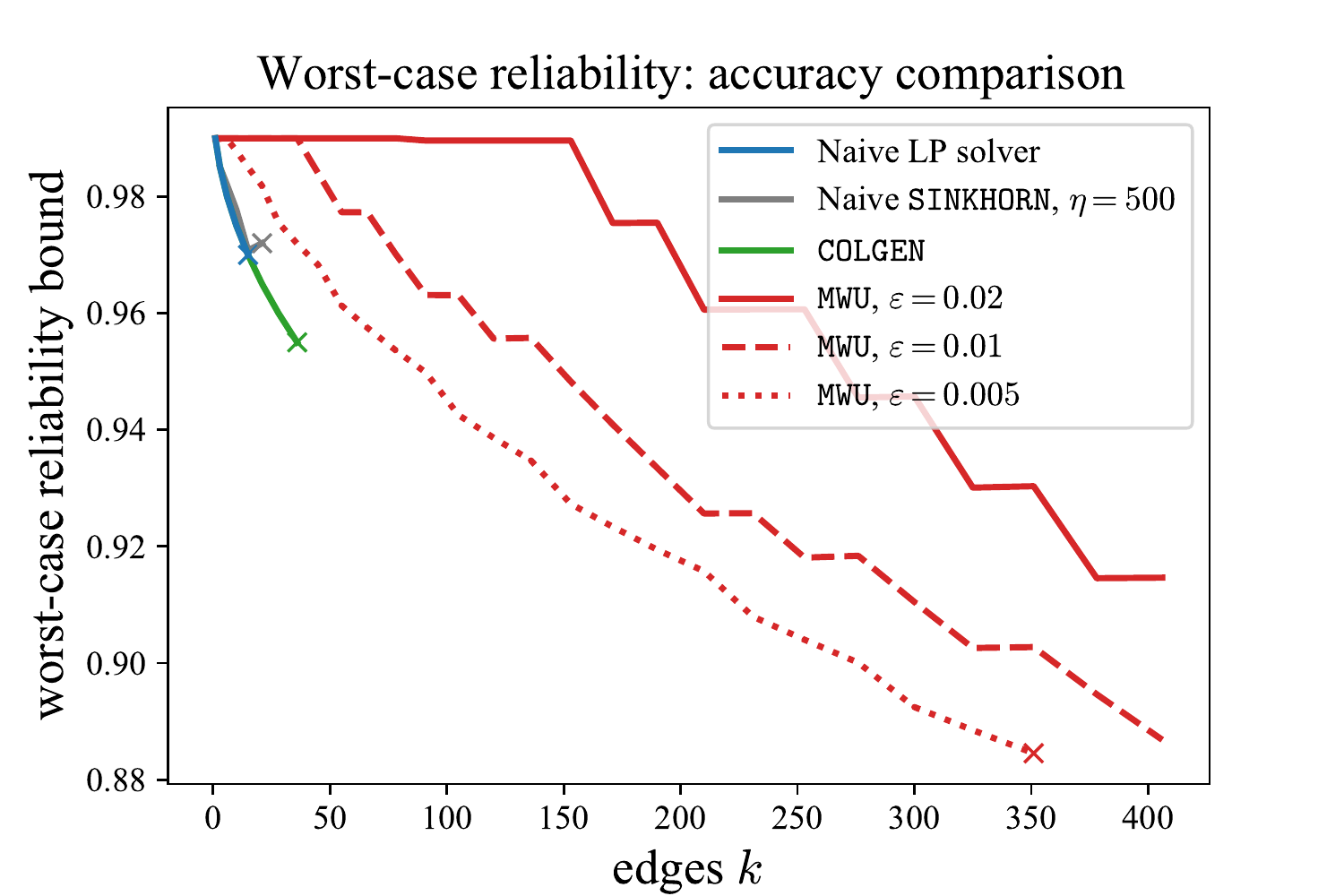} \\
			\includegraphics[scale=0.5]{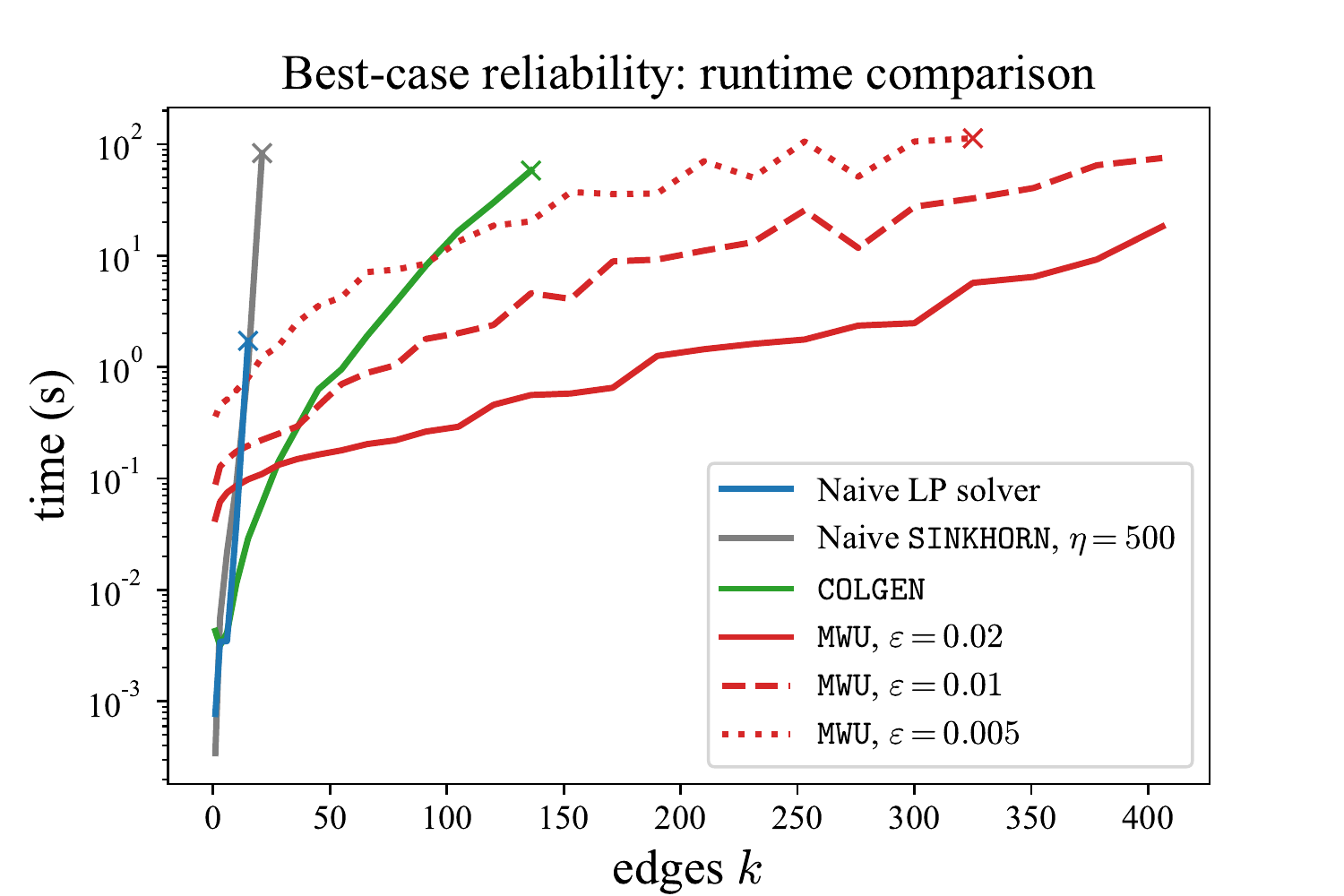} & \includegraphics[scale=0.5]{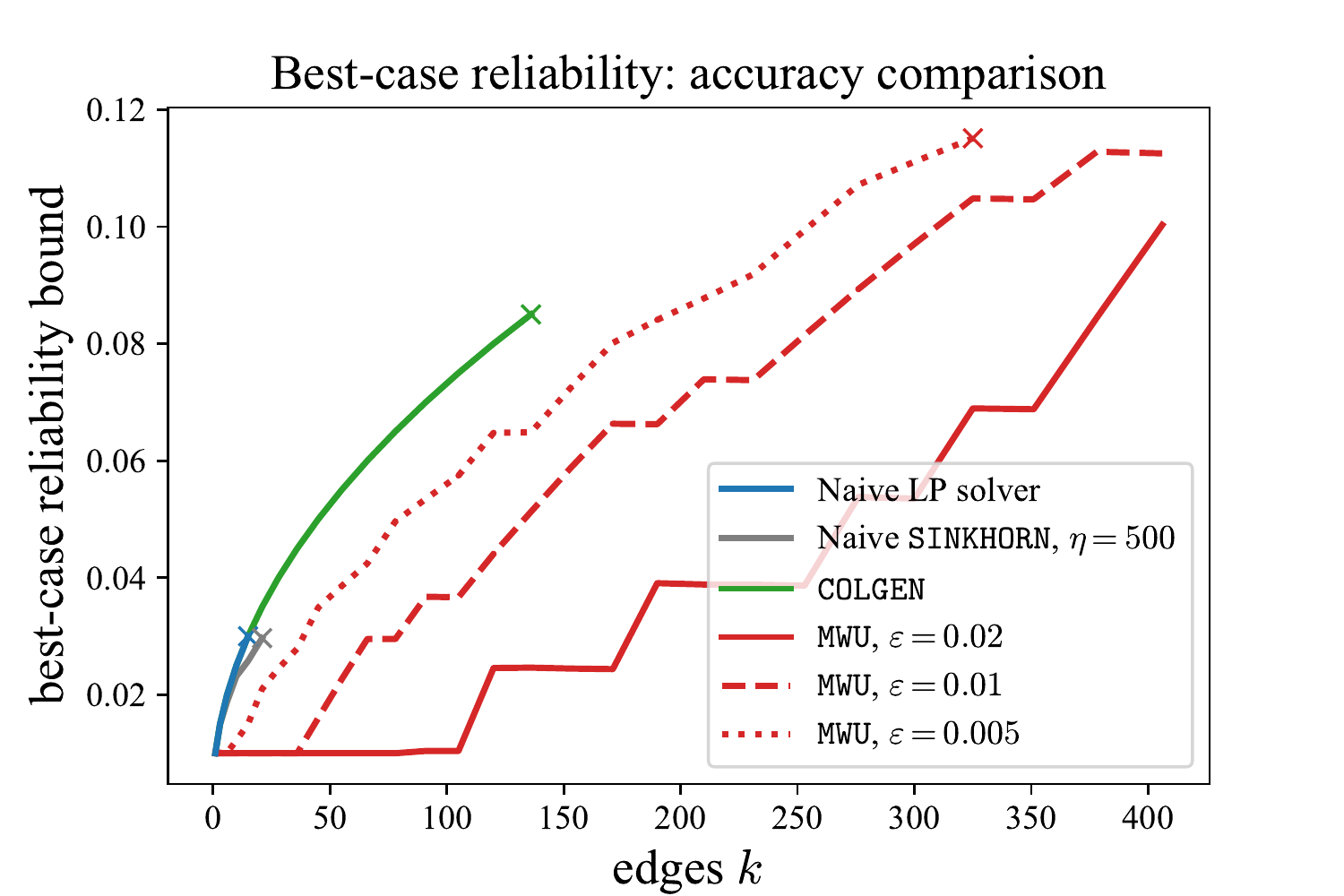}
		\end{tabular}
	\caption{
		Top: comparison of the runtime (left) and accuracy (right) of the algorithms described in the main text, for the worst-case reliability of a clique graph on $t$ vertices and $k = \binom{t}{2}$ edges with reliability probabilities $q_e = 0.99$. Bottom: same, but for best-case reliability and reliability probabilities $q_e = 0.01$. For worst-case reliability, the algorithms compute an upper bound, so a smaller value is better; reverse for best-case reliability. The algorithms are cut off at 2 minutes, denoted by an ``x''. $\Sink$ is run at the highest precision (i.e., highest $\eta$) before numerical precision issues. The $\COLGEN$ algorithm that our framework recovers computes exact solutions an order-of-magnitude faster than the other algorithms, and the new $\MWU$ algorithm computes reasonably approximate solutions for $k = 400$, which amounts to an $\MOT$ LP with $n^k = 2^{400} \approx 2.6 \times 10^{120}$ variables.
	} \label{fig:reliabilityclique}
\end{figure}

In Figure~\ref{fig:reliabilityclique}, we compare the numerical performance of the algorithms in Corollary~\ref{cor:rel}---$\COLGEN$ and $\MWU$ with polynomial-time implementation of their bottlenecks---with the fastest previous algorithms for both best-case and worst-case network reliability. Previously, the fastest algorithms that apply to this problem are (1) out-of-the-box LP solvers run on $\MOT$, (2) the brute-force implementation of $\Sink$ which marginalizes over all $n^k = 2^{|E|}$ entries in each iteration, and (3) this $\COLGEN$ algorithm that we recover~\citep{zemel1982polynomial,weiss1986stochastic}. It is unknown if there is a practically efficient implementation of the $\SMinOCS$ oracle (and thus of $\Sink$) for both best-case or worst-case reliability. Since the previous algorithms (1) and (2) have exponential runtime that scales as $n^{\Omega(k)} = 2^{\Omega(|E|)}$, they do not scale past tiny input sizes. In contrast, the algorithms in Corollary~\ref{cor:rel} scale to much larger inputs. Indeed, the $\COLGEN$ algorithm that our framework recovers can compute exact solutions roughly an order-of-magnitude faster than the other algorithms, and the new $\MWU$ algorithm computes reasonably approximate solutions beyond $k = 400$, which amounts to an $\MOT$ LP with $n^k = 2^{400} \approx 2.6 \times 10^{120}$ variables. 
\section{Application: $\MOT$ problems with low-rank plus sparse structure}\label{sec:lr}

In this section, we consider $\MOT$ problems whose cost tensors $C$ decompose into low-rank and sparse components. We propose the first polynomial-time algorithms for this general class of $\MOT$ problems.
\par The section is organized as follows. In \S\ref{ssec:lr:setup} we formally describe this setup and discuss why it is incomparable to all other structures discussed in this paper. In \S\ref{ssec:lr:alg}, we show that for costs with this structure, the $\AMinO$ and $\SMinO$ oracles can be implemented in polynomial time; from this it immediately follows that $\MWU$ and $\Sink$ can be implemented in polynomial time. Finally, in \S\ref{ssec:lr:risk} and \S\ref{ssec:lr:proj}, we provide two illustrative applications of these algorithms. The former regards portfolio risk management and is a direct application of our result for $\MOT$ with low-rank cost tensors. The latter regards projecting mixture distributions to the transportation polytope and illustrates the versality of our algorithmic results since this problem is quadratic optimization over the transportation polytope rather than linear (a.k.a. $\MOT$).

\subsection{Setup}\label{ssec:lr:setup}

We begin by recalling the definition of tensor rank. It is the direct analog of the standard concept of matrix rank. See the survey~\citep{kolda2009tensor} for further background.

\begin{defin}[Tensor rank]\label{def:tensor-rank}
	A rank-$r$ factorization of a tensor $R \in \Rntk$ is a collection of $rk$ vectors $\{u_{i,\ell}\}_{i \in [k], \ell \in [r]} \subset \R^n$ satisfying
	\[
		R = \sum_{\ell=1}^r \bigotimes_{i=1}^k u_{i,\ell}.
	\]
	The rank of a tensor is the minimal $r$ for which there exists a rank-$r$ factorization.
\end{defin}

In this section we consider $\MOT$ problems with the following ``low-rank plus sparse'' structure.

\begin{defin}[Low-rank plus sparse structure for $\MOT$]\label{def:lr}
	An $\MOT$ cost tensor $C \in \Rntk$ has low-rank plus sparse structure of rank $r$ and sparsity $s$ if it decomposes as
	\begin{align}
		C = 
		R + S,
		\label{eq:lr-def}
	\end{align}  
	where $R$ is a rank-$r$ tensor and $S$ is an $s$-sparse tensor.
\end{defin}

Throughout, we make the natural assumption that $S$ is input through its $s$ non-zero entries, and that $R$ is input through a rank-$r$ factorization. We also make the natural assumption that the entries of both $R$ and $S$ are of size $O(\Cmax)$---this rules out the case of having extremely large entries of $R$ and $S$, one positive and one negative, which cancel to yield a small entry of $C = R + S$.

\begin{remark}[Neither low-rank structure nor sparse structure can be modeled by graphical structure or set-optimization structure]\label{rem:lr-incomparable}
	 In general, both rank-$1$ costs and polynomially sparse costs do not have non-trivial graphical structure. Specifically, modeling these costs with graphical structure requires the complete graph (a.k.a., maximal treewidth of $k-1$)---and because $\MOT$ problems with graphical structure of treewidth $k-1$ are $\NP$-hard to solve in the absence of further structure~\citep{AltBoi20hard}, modeling such problems with graphical structure is useless for the purpose of designing polynomial-time $\MOT$ algorithms. It is also clear that neither low-rank structure nor sparse structure can be modeled by set-optimization structure because in general, neither $R$ nor $S$ nor $R+S$ has binary-valued entries.
\end{remark}

\subsection{Polynomial-time algorithms}\label{ssec:lr:alg}

From a technical perspective, the main result of this section is that there is a polynomial-time algorithm for approximating the minimum entry of a tensor that decomposes into constant-rank and sparse components. Previously, this was not known even for constant-rank tensors. This result may be of independent interest. We remark that this result is optimal in the sense that unless $\P = \NP$, there does not exist an algorithm with runtime that is \emph{jointly} polynomial in the input size and the rank $r$~\citep{AltBoi20hard}.

\begin{theorem}[Polynomial-time algorithm solving $\AMinO$ and $\SMinO$ for low-rank + sparse costs]\label{thm:lr-oracles}
	Consider cost tensors $C \in \Rntk$ that have low-rank plus sparse structure of rank $r$ and sparsity $s$ (see Definition~\ref{def:lr}).
	For any fixed $r$, Algorithm~\ref{alg:lr} runs in $\poly(n, k, s, \Cmax/\eps)$ time and solves the $\eps$-approximate $\AMinO_C$ oracle. Furthermore, it also solves the $\SMinO_{\tilde{C}}$ oracle for $\eta = (2k \log n)/\eps$ on some cost tensor $\tilde{C} \in \Rntk$ satisfying $\|C - \tilde{C}\|_{\max} \leq \eps/2$.
\end{theorem}

We make three remarks about Theorem~\ref{thm:lr-oracles}. First, we are unaware of any polynomial-time implementation of $\SMinO_C$ for the cost $C$. Instead, Theorem~\ref{thm:lr-oracles} solves the $\SMinO_{\tilde{C}}$ oracle for an $O(\eps)$-approximate cost tensor $\tilde{C}$ since this is sufficient for implementing $\Sink$ on the original cost tensor $C$ (see Corollary~\ref{cor:lr-algs} below). Second, it is an interesting open question if the $\poly(n,k,\Cmax/\eps)$ runtime for the $\eps$-approximate $\AMinO_C$ oracle can be improved to $\poly(n,k,\log (\Cmax/\eps))$, as this would imply a $\poly(n,k)$ runtime for the $\MinO_C$ oracle and thus for this class of $\MOT$ problems (see also Footnote~\ref{fn:lr-exact} in the introduction). Third, we remark about practical efficiency: the runtime of Algorithm~\ref{alg:lr} is not just polynomially small in $s$ and $n$, but in fact linear in $s$ and near-linear in $n$. 
However, since this improved runtime is not needed for the theoretical results in the sequel, we do not pursue this further.

Combining the efficient oracle implementations in Theorem~\ref{thm:lr-oracles} with our algorithm-to-oracles reductions in \S\ref{sec:algs} implies the first polynomial-time algorithms for $\MOT$ problems with costs that have constant-rank plus sparse structure. This is optimal in the sense that unless $\P = \NP$, there does not exist an algorithm with runtime that is \emph{jointly} polynomial in the input size and the rank $r$~\citep{AltBoi20hard}.

\begin{cor}[Polynomial-time algorithms solving $\MOT$ for low-rank + sparse costs]\label{cor:lr-algs}
	Consider cost tensors $C \in \Rntk$ that have low-rank plus sparse structure of constant rank $r$ and $\poly(n,k)$ sparsity $s$ (see Definition~\ref{def:lr}). For any $\eps > 0$:
	\begin{itemize}
		\item The $\MWU$ algorithm in \S\ref{ssec:algs:mwu} computes an $\eps$-approximate solution to $\MOT_C$ in $\poly(n,k,\Cmax/\eps)$ time.
		\item The $\Sink$ algorithm in \S\ref{ssec:algs:sink} computes an $\eps$-approximate solution to $\MOT_C$ in $\poly(n,k,\Cmax/\eps)$ time.
	\end{itemize}
	Moreover, $\MWU$ outputs a polynomially sparse tensor, whereas $\Sink$ outputs a fully dense tensor through the implicit representation described in \S\ref{sssec:sink:alg}.
\end{cor}
\begin{proof}
	For $\MWU$, simply combine the polynomial-time reduction to the $\AMinO_C$ oracle (Theorem~\ref{thm:mwumotp}) with the polynomial-time algorithm for the $\AMinO$ oracle (Theorem~\ref{thm:lr-oracles}). For $\Sink$, combining the polynomial-time reduction to the $\SMinO_{\tilde{C}}$ oracle (Theorem~\ref{thm:sink-mot:smin}) with the polynomial-time algorithm for the $\SMinO_{\tilde{C}}$ oracle (Theorem~\ref{thm:lr-oracles}) yields a $\poly(n,k,\Cmax/\eps)$ algorithm for $\eps/2$-approximating the $\MOT$ problem with cost tensor $\tilde{C}$. It therefore suffices to show that the values of the $\MOT$ problems with cost tensors $C$ and $\tilde{C}$ differ by at most $\eps/2$, that is,
	\[
	\abs{\min_{P \in \Coup} \langle P, C \rangle - \min_{P \in \Coup} \langle P, \tilde{C} \rangle} 
	\leq \eps/2.
	\]
	But this holds because both $\MOT$ problems have the same feasible set, and for any feasible $P \in \Coup$ it follows from H\"older's inequality that the objectives of the two $\MOT$ problems differ by at most 
	\[
	\abs{\langle P, C \rangle - \langle P, \tilde{C} \rangle} 
	\leq \|P\|_1 \|C - \tilde{C}\|_{\max}
	\leq \eps/2.
	\] 
\end{proof}

Below, we describe the algorithm in Theorem~\ref{thm:lr-oracles}.
Specifically, in \S\ref{sssec:lr:tech}, we give four helper lemmas which form the technical core of our algorithm; and then in \S\ref{sssec:lr:alg}, we combine these ingredients to design the algorithm and prove its correctness. Throughout, recall that we use the bracket notation $f[A]$ to denote the \emph{entrwise} application of a univariate function $f$ (e.g., $\exp$, $\log$, or a polynomial) to $A$.

\subsubsection{Technical ingredients}\label{sssec:lr:tech}
At a high level, our approach to designing the algorithm in Theorem~\ref{thm:lr-oracles} is to approximately compute the $\SMinO$ oracle in polynomial time by synthesizing four facts:
\begin{enumerate}
	\item By expanding the softmin and performing simple operations,
	 it suffices to compute the total sum of all $n^k$ entries of the entrywise exponentiated tensor $\exp[-\eta R]$ (modulo simple transforms). 
	 \item Although $\exp[-\eta R]$ is in general a full-rank tensor, we can exploit the fact that $R$ is a low-rank tensor in order to approximate $\exp[-\eta R]$ by a low-rank tensor $L$. (Moreover, we can efficiently compute a low-rank factorization of $L$ in closed form.)
	 \item There is a simple algorithm for computing the sum of all $n^k$ entries of $L$ in polynomial time because $L$ is low-rank. (And thus we may approximate the sum of all $n^k$ entries of $\exp[-\eta R]$ as desired in step 1.)
	 \item This approximation is sufficient for computing both the $\AMinO$ and $\SMinO$ oracle in Theorem~\ref{thm:lr-oracles}.
\end{enumerate}

Of these four steps, the main technical step is the low-rank approximation in step two. Below, we formalize these four steps individually in Lemmas~\ref{lem:lr:sparse},~\ref{lem:lr-K},~\ref{lem:lr-marg}, and~\ref{lem:lr:prec}. 
Further detail on how to synthesize these four steps is then provided afterwards, in the proof of Theorem~\ref{thm:lr-oracles}. 

\par It is convenient to write the first lemma in terms of an approximate tensor $\tilde{C} = \tilde{R}+S$ rather than the original cost $C = R + S$. 

\begin{lemma}[Softmin for cost with sparse component]\label{lem:lr:sparse}
	Let $\tilde{C} = \tilde{R} + S$ and $p_1,\dots,p_k \in \R^n$. Then
	\[
	\smineta_{\jvec \in [n]^k} \tilde{C}_{\jvec} - \sum_{i=1}^k [p_i]_{j_i}
	=
	-\eta^{-1} \log(a + b),
	\]
	where $d_i := \exp[\eta p_i] \in \Rp^n$,
	\begin{align}
		a := \sum_{\substack{\jvec \in [n]^k \\ \text{s.t. } S_{\jvec} \neq 0}} \prod_{i=1}^k [d_i]_{j_i} \cdot e^{-\eta \tilde{R}_{\jvec}} \cdot (e^{-\eta S_{\jvec}} - 1)
		\label{eq:lr:alg:a}
	\end{align}
	and 
	\begin{align}
		b := \sum_{\jvec \in [n]^k} \prod_{i=1}^k [d_i]_{j_i} \cdot e^{-\eta \tilde{R}_{\jvec}}.
		\label{eq:lr:alg:b}
	\end{align}
\end{lemma}
\begin{proof}
	By expanding the definition of softmin, and then substituting $p_i$ with $d_i$ and $\tilde{C}$ with $\tilde{R} + S$,
	\begin{align*}
		\smineta_{\jvec \in [n]^k} \tilde{C}_{\jvec} - \sum_{i=1}^k [p_i]_{j_i}
		&=
		-\frac{1}{\eta}  \log \left( \sum_{\jvec \in [n]^k} 
		e^{\eta \sum_{i=1}^k [p_i]_{j_i}}
		e^{-\eta \tilde{C}_{\jvec}} 
		\right)
		= 
		-\frac{1}{\eta} \log \left( \sum_{\jvec \in [n]^k} \prod_{i=1}^k [d_i]_{j_i} \cdot e^{-\eta \tilde{R}_{\jvec}} e^{-\eta S_{\jvec}} \right).
	\end{align*}
	By simple manipulations, we conclude that the above quantity is equal to the desired quantity:
	\begin{align*}
		\cdots =&-\frac{1}{\eta} \log \left( \sum_{\substack{\jvec \in [n]^k \\ \text{s.t. } S_{\jvec} \neq 0}} \prod_{i=1}^k [d_i]_{j_i} \cdot e^{-\eta \tilde{R}_{\jvec}} e^{-\eta S_{\jvec}}
		+
		\sum_{\substack{\jvec \in [n]^k \\ \text{s.t. } S_{\jvec} = 0}} \prod_{i=1}^k [d_i]_{j_i} \cdot e^{-\eta \tilde{R}_{\jvec}}  \right)
		\\=&
		-\frac{1}{\eta}  \log \left( \sum_{\substack{\jvec \in [n]^k \\ \text{s.t. } S_{\jvec} \neq 0}} \prod_{i=1}^k [d_i]_{j_i} \cdot e^{-\eta \tilde{R}_{\jvec}} \left( e^{-\eta S_{\jvec}} - 1\right)
		+
		\sum_{\jvec \in [n]^k} \prod_{i=1}^k [d_i]_{j_i} \cdot e^{-\eta \tilde{R}_{\jvec}} \right)
		\\=&
		- \frac{1}{\eta} \log (a + b).
	\end{align*}
	Above, the first step is by partitioning the sum over $\jvec \in [n]^k$ based on if $S_{\jvec} = 0$, the second step is by adding and subtracting $\sum_{\jvec \in [n]^k \text{ s.t. } S_{\jvec} \neq 0} \prod_{i=1}^k [d_i]_{j_i} \cdot e^{-\eta \tilde{R}_{\jvec}}$, and the last step is by definition of $a$ and $b$. 
\end{proof}

\begin{lemma}[Low-rank approximation of the exponential of a low-rank tensor]\label{lem:lr-K}
	There is an algorithm that given $R \in \Rntk$ in rank-$r$ factored form, $\eta > 0$, and a precision $\epst < e^{-\eta R_{\max}}$, takes $n \cdot \poly(k,\tilde{r})$ time to compute a rank-$\tilde{r}$ tensor $L  \in \Rntk$ in factored form satisfying $\|L - \exp[-\eta R]\|_{\max} \leq \epst$, where
	\begin{align}
		\tilde{r}
		\leq \binom{r + O(\log\tfrac{1}{\epst})}{r}.
		\label{eq:lr:rtilde}
	\end{align}
\end{lemma}
\begin{proof}
	By classical results from approximation theory (see, e.g.,~\citep{Tre19}), there exists a polynomial $q$ of degree $m = O(\log1/\epst)$ satisfying 
	\[
	\abs{\exp(-\eta x) - q(x)} \leq \epst, \qquad \forall x \in [-R_{\max},R_{\max}].
	\]
	For instance, the Taylor or Chebyshev expansion of $x \mapsto \exp(-\eta x)$ suffices. Thus the tensor $L$ with entries 
	\[
	L_{\jvec} = q(R_{\jvec})
	\]
	approximates $\exp[-\eta R]$ to error 
	\[
	\|L - \exp[-\eta R]\|_{\max} \leq \epst.
	\]
	\par We now show that $L$ has rank $\tilde{r} \leq \binom{r + m }{r}$, and moreover that a rank-$\tilde{r}$ factorization can be computed in $n \cdot \poly(k,\tilde{r})$ time. Denote $q(x) = \sum_{t=0}^m a_t x^t$ and $R = \sum_{\ell=1}^r \otimes_{i=1}^k u_{i,\ell}$. By definition of $L$, definition of $q$ and $R$, and then the Multinomial Theorem, 
	\[
	L_{\jvec}
	= q(R_{\jvec})
	= \sum_{t=0}^m a_t \left( \sum_{\ell=1}^r \prod_{i=1}^k [u_{i,\ell}]_{j_i} \right)^t
	= \sum_{\alpha \in \N_0^r \, : \, |\alpha| \leq m} \binom{|\alpha|}{\alpha} a_{|\alpha|} \prod_{\ell=1}^r \prod_{i=1}^k [u_{i,\ell}]_{j_i}^{\alpha_i},
	\]
	where the sum is over $r$-tuples $\alpha$ with non-negative entries summing to at most $m$.
	Thus 
	\[
	L = \sum_{\alpha \in \N_0^r \, : \, |\alpha| \leq m} \bigotimes_{i=1}^kv_{i,\alpha},
	\]
	where $v_{i,\alpha} \in \Rn$ denotes the vector with $j$-th entry $ \binom{|\alpha|}{\alpha} a_{|\alpha|}  \prod_{\ell=1}^r [u_{i,\ell}]_j^{\alpha_i}$ for $i=1$, and $\prod_{\ell=1}^r [u_{i,\ell}]_j^{\alpha_i}$ for $i > 1$. This yields the desired low-rank factorization of $L$ because
	\[
	\tilde{r} \leq \# \{\alpha \in \N_0^r \; : \; |\alpha| \leq m \} = \binom{r+m}{r}.
	\]
	Finally, since each of the $k\tilde{r}$ vectors $v_{i,\alpha}$ in the factorization of $L$ can be computed efficiently from the closed-form expression above, the desired runtime follows.
\end{proof}

\begin{lemma}[Marginalizing a scaled low-rank tensor]\label{lem:lr-marg}
	Given vectors $d_1, \dots, d_k \in \Rn$ and a tensor $L \in \Rntk$ through a rank $\tilde{r}$ factorization, we can compute $m((\otimes_{i=1}^k d_i) \odot L)$ in $O(nk\tilde{r})$ time. 
\end{lemma}
\begin{proof}
	Denote the factorization of $L$ by $L = \sum_{\ell=1}^{\tilde{r}} \otimes_{i=1}^k v_{i,\ell}$. Then 
	\begin{align*}
		m((\otimes_{i=1}^k d_i) \odot L )
		&=
		\sum_{\jvec \in [n]^k} \left[ (\otimes_{i=1}^k d_i) \odot L \right]_{\jvec}
		= 
		\sum_{\jvec \in [n]^k} \sum_{\ell=1}^{\tilde{r}} \prod_{i=1}^k [d_i]_{j_i} [v_{i,\ell}]_{j_i}
		= 
		\sum_{\ell=1}^{\tilde{r}} \prod_{i=1}^k \sum_{j=1}^n [d_i]_{j} [v_{i,\ell}]_{j}
		= 
		\sum_{\ell=1}^{\tilde{r}} \prod_{i=1}^k \langle d_i, v_{i,\ell} \rangle,
	\end{align*}
	where the first step is by definition of the $m(\cdot)$ operation that sums over all entries, the second step is by definition of $L$, and the third step is by swapping products and sums. Thus computing the desired quantity amounts to computing $\tilde{r}k$ inner products of $n$-dimensional vectors.  This can be done in $O(nr\tilde{k})$ time.
\end{proof}

\begin{lemma}[Precision of the low-rank approximation]\label{lem:lr:prec}
	Let $\eps \leq 1$. Suppose $L \in \Rntk$ satisfies $\|L - \exp[-\eta R]\|_{\max} \leq \tfrac{\eps}{3} e^{-\eta R_{\max}}$. Then the matrix $\tilde{C} := - \frac{1}{\eta} \log[L] + S$ satisfies
		\begin{align}
			\|\tilde{C} - C\|_{\max} \leq \frac{\eps}{2}.
			\label{eq:lem:lr:prec}
		\end{align}
\end{lemma}

\begin{proof}
	Observe that the minimum entry of $L$ is at least 
	\begin{align}
		e^{-\eta R_{\max}} - \tfrac{\eps}{3} e^{-\eta R_{\max}} \geq \tfrac{2}{3}e^{-\eta R_{\max}}.
		\label{eq:lem:lr:prec-min}
	\end{align}
	Since this is strictly positive, the tensor $\tilde{R} := -\eta^{-1} \log[L]$ is well defined. Furthermore, 
	\[
	\|\eta \tilde{R} - \eta R\|_{\max}
	=
	\max_{\jvec \in [n]^k} \abs{\eta \tilde{R}_{\jvec} - \eta R_{\jvec}}
	\leq
	\max_{\jvec \in [n]^k} 
	\frac{\abs{L_{\jvec} - e^{-\eta R_{\jvec}}}}{\min(L_{\jvec}, e^{-\eta R_{\jvec}})
	}
	\leq
	\frac{\tfrac{\eps}{3} e^{-\eta R_{\max}}}{\tfrac{2}{3} e^{-\eta R_{\max}}}
	=
	\frac{\eps}{2},
	\]
	where above the first step is by definition of the max norm; the second step is by the elementary inequality $|\log x - \log y| \leq |x-y|/\min(x,y)$ which holds for positive scalars $x$ and $y$~\citep[Lemma K]{AltBacRud19}; and the third step is by~\eqref{eq:lem:lr:prec-min} and the approximation bound of $L$. Since $\eta \geq 1$, we therefore conclude that $\|\tilde{R} - R\|_{\max} \leq \eps/2$. By adding and subtracting $S$, this implies $\|\tilde{C} - C\|_{\max} = \|\tilde{R} - R\|_{\max} \leq \eps/2$.
\end{proof}

\subsubsection{Proof of Theorem~\ref{thm:lr-oracles}}\label{sssec:lr:alg}

We are now ready to state the algorithm in Theorem~\ref{thm:lr-oracles}. Pseudocode is in Algorithm~\ref{alg:lr}. Note that $\tilde{R} = -\eta^{-1} \log [L]$ and $\tilde{C} = \tilde{R} + S$ are never explicitly computed because in both Lines~\ref{line:lr:a} and~\ref{line:lr:b}, the algorithm performs the relevant operations only through the low-rank tensor $L$ and the sparse tensor $S$.

\begin{algorithm}
	\caption{Polynomial-time algorithm for $\AMinO$ and $\SMinO$ for low-rank + sparse costs (Theorem~\ref{thm:lr-oracles}).}
	\hspace*{\algorithmicindent} \textbf{Input:} 
	Low-rank tensor $R$, sparse tensor $S$, matrix $p \in \Rbar^{n \times k}$,
	accuracy $\eps > 0$  \\
	\hspace*{\algorithmicindent} \textbf{Output:} Solution to both $\AMinO_C(p,\eps)$ on cost tensor $C = R+S$, and also $\SMinO_{\tilde{C}}(p,(2k\log n)/\eps)$ on some approximate cost tensor $\tilde{C}$ satisfying $\|C - \tilde{C}\|_{\max} \leq \eps/2$
	\begin{algorithmic}[1]
		\State $\eta \gets (2 k \log n)/\eps$
		\State Compute low-rank approximation $L$ of $\exp[-\eta R]$ via Lemma~\ref{lem:lr-K}, for precision $\epst = \frac{\eps}{3}e^{-\eta \Rmax}$\label{line:lr:lr}
		\State Compute $a$ in~\eqref{eq:lr:alg:a} directly by enumerating over the polynomially many non-zero entries of $S$\label{line:lr:a}, where $\tilde{R} = -\eta^{-1} \log [L]$
		\State Compute $b$ in~\eqref{eq:lr:alg:b} via Lemma~\ref{lem:lr-marg}, where $\tilde{R} = -\eta^{-1} \log [L]$ \label{line:lr:b}
		\State Return $-\eta^{-1} \log(a + b)$ 
	\end{algorithmic}
	\label{alg:lr}
\end{algorithm}

\begin{proof}[Proof of Theorem~\ref{thm:lr-oracles}]
	\underline{Proof of correctness for $\SMinO$.} 
	Consider any oracle inputs $p = (p_1, \dots, p_k) \in \Rbar^{n \times k}$. 
	By Lemma~\ref{lem:lr:prec}, the tensor $\tilde{C} = \tilde{R}+ S = -\eta^{-1} \log L + S$ satisfies $\|\tilde{C} - C\|_{\max} \leq \eps/2$. Therefore it suffices to show that Algorithm~\ref{alg:lr} correctly computes $\SMinO_{\tilde{C}}(p,\eta)$. This is true because that quantity is equal to $-\eta^{-1} \log(a + b)$ by Lemma~\ref{lem:lr:sparse}.
	
	\par \underline{Proof of correctness for $\AMinO$.} We have just established that Algorithm~\ref{alg:lr} computes $\SMinO_{\tilde{C}}(p,\eta)$. Because $\eta = (2k \log n)/\eps$ and the fact that $\SMinO$ is a special case of $\AMinO$ (Remark~\ref{rem:oracles:smin-amin}), it follows that $\SMinO_{\tilde{C}}(p,\eta)$ is within additive accuracy $\eps/2$ of $\MinO_{\tilde{C}}(p,\eta)$. Therefore, by the triangle inequality, it suffices to show that $\MinO_{\tilde{C}}(p)$ is within $\eps/2$ additive accuracy of $\MinO_{C}(p)$. That is, it suffices to show that
	\[
		\abs{\min_{\jvec \in [n]^k} C_{\jvec} - \sum_{i=1}^k [p_i]_{j_i}
		-
		\min_{\jvec \in [n]^k} \tilde{C}_{\jvec} - \sum_{i=1}^k [p_i]_{j_i}} \leq \eps/2.
	\]
	But this is true because $\|C - \tilde{C}\|_{\max} \leq \eps/2$ by Lemma~\ref{lem:lr:prec}, and thus the quantities $C_{\jvec} - \sum_{i=1}^k [p_i]_{j_i}$ and $\tilde{C}_{\jvec} - \sum_{i=1}^k [p_i]_{j_i}$ are within additive accuracy $\eps/2$ for each $\jvec \in [n]^k$.

	\par \underline{Proof of runtime.} We prove the claimed runtime bound simultaneously for the $\AMinO$ and $\SMinO$ computation because we use the same algorithm for both. To this end, we first bound the rank $\tilde{r}$ of the low-rank approximation $L$ computed in Lemma~\ref{lem:lr-K}. Note that since $\tilde{\eps} = \tfrac{\eps}{3} e^{-\eta R_{\max}}$ and since it is assumed that $R_{\max} = O(C_{\max})$, we have $\log 1/\tilde{\eps} = O(\tfrac{\Cmax}{\eps}k \log n )$. Therefore 
	\[
		\tilde{r} 
		\leq \binom{r + O(\log 1/\epst)}{r}
		= O(\log1/\tilde{\eps})^r = O(\tfrac{\Cmax}{\eps}k \log n )^r = \poly(\log n,k,\Cmax/\eps).
	\]
	Above, the first step is by Lemma~\ref{lem:lr-K}, and the final step is because $r$ is assumed constant.
	
	\par Therefore Line~\ref{line:lr:lr} in Algorithm~\ref{alg:lr} takes polynomial time by Lemma~\ref{lem:lr-K}, Line~\ref{line:lr:a} takes polynomial time by simply enumerating over the $s$ non-zero entries of $S$, and Line~\ref{line:lr:b} takes polynomial time by Lemma~\ref{lem:lr-marg}.
\end{proof}

\subsection{Application vignette: risk estimation}\label{ssec:lr:risk}

Here we consider an application to portfolio risk management. 
For simplicity of exposition, let us first describe the setting of $1$ financial instrument (``stock''). Consider investing in one unit of a stock for $k$ years. For $i \in \{0, \dots, k\}$, let $X_i$ denote the price of the stock at year $i$. Suppose that the return $\rho_i = X_{i}/X_{i-1}$ of the stock between years $i-1$ and $i$ is believed to follow some distribution $\rho_i \sim \mu_i$. A fundamental question about the riskiness of this stock is to compute the investor's expected profit in the worst-case over all joint probability distributions on future returns $(\rho_1,\dots,\rho_k)$ that are consistent with the modeled marginal distributions $(\mu_1,\dots,\mu_k)$. This is an $\MOT$ problem with cost $C$ given by
\[
	C(\rho_1,\dots,\rho_k) = \prod_{i \in [k]} \rho_i,
\]
where here we view $C$ as a function rather than a tensor for notational simplicity. If each return $\rho_i$ has $n$ possible values (e.g., after quantization), then the cost $C$ is equivalently represented as a rank-1 tensor in $\Rntk$ (by assigning an index to each of the $n$ possible values of each $\rho_i$). Therefore our result Corollary~\ref{cor:lr-algs} provides a polynomial-time algorithm for solving this $\MOT$ problem defining the investor's worst-case profit.

\par Rather than formalize this proof for $1$ stock, we directly generalize to the general case of investing in $r$ stocks, $r \geq 1$. 
This is essentially identical to the simple case of $r=1$ stock, modulo additional notation. 

\begin{cor}[Polynomial-time algorithm for expected profit given marginals on the returns]\label{cor:lr-return}
	Suppose an investor holds $1$ unit of $r$ stocks for $k$ years. For each stock $\ell \in [r]$ and each year $i \in [k]$, let $\rho_{i,\ell}$ denote the relative price of stock $\ell$ between years $i$ and $i-1$. Suppose $\rho_{i,\ell}$ has distribution $\mu_{i,\ell}$, and that each $\mu_{i,\ell}$ has at most $n$ atoms. Let $\Rmax = \max_{\{\rho_{i,\ell}\}} \sum_{\ell=1}^r \prod_{i=1}^k \rho_{i,\ell}$ denote the maximal possible return. For any constant number of stocks $r$, there is a $\poly(n,k,\Rmax/\eps)$ time algorithm for $\eps$-approximating the expected profit in the worst-case over all futures that are consistent with the returns' marginal distributions.
\end{cor}
\begin{proof}
	This is the optimization problem
	\[
		\min_{P \in	\cM(\{\mu_{i,\ell}\}_{i \in [k], \ell \in [r]} )}
		 \E_{ \{\rho_{i,\ell}\}_{i \in [k], \ell \in [r]} \sim P} \left[ \sum_{\ell=1}^r \prod_{i=1}^k \rho_{i,\ell} \right]
	\]
	over all joint distributions $P$ on the returns $\{\rho_{i,\ell}\}_{i \in [k], \ell \in [k]}$ that are consistent with the marginal distibutions $\{\mu_{i,\ell}\}_{i \in [k], \ell \in [k]}$. This is an $\MOT$ problem with $k' = rk$ marginals, each over $n$ atoms, with cost function
	\begin{align}
		C\Big(\{ \rho_{i,\ell} \}_{i \in [k], \ell \in [r]}\Big)
		= 
		\sum_{\ell' \in [r]} \prod_{(i,\ell) \in [k] \times [r] \cong [k']} \big(\rho_{i,\ell} \cdot \mathds{1}[\ell = \ell'] + \mathds{1}[\ell \neq \ell']\big).
		\label{eq:lr-return}
	\end{align}
	By viewing this cost function $C$ as a cost tensor in the natural way (i.e., assigning an index to each of the $n$ possible values of $\rho_{i,\ell}$), this representation~\eqref{eq:lr-return} shows that the corresponding cost tensor $C \in (\R^n)^{\otimes k'}$ has rank $r$. Moreover, observe that the maximum entry of the cost is $\Rmax$. Therefore we may appeal to our polynomial-time $\MOT$ algorithms in Corollary~\ref{cor:lr-algs} for costs with constant rank. 
\end{proof}

The algorithm is readily generalized, e.g., if the investor has different units of a stock, or if a stock is held for a different number of years. The former is modeled simply by adding an extra year in which the return of stock $\ell$ is equal to the number of units, with probability $1$. The latter is modeled simply by setting the return of stock $\ell$ to be $1$ for all years after it is held, with probability $1$. 

\par In Figure~\ref{fig:riskestimation}, we provide a numerical illustration comparing our new polynomial-time algorithms for this risk estimation task with the previous fastest algorithms. Previously, the fastest algorithms that apply to this problem are out-of-the-box LP solvers run on $\MOT$, and the brute-force implementation of $\Sink$ which marginalizes over all $n^k$ entries in each iteration. Since both of these previous algorithms have exponential runtime that scales as $n^{\Omega(k)}$, they do not scale beyond tiny input sizes of $n=10$ and $k=8$ even with two minutes of computation time. In contrast, our new polynomial-time algorithms compute high-quality solutions for problems that are orders-of-magnitude larger. For example, our polynomial-time implementation of $\Sink$ takes less than a second to solve an $\MOT$ LP with $n^k = 10^{30}$ variables. 

\par Details for this numerical experiment: we consider $r=1$ stock over $k$ timesteps, where each marginal distribution $\mu_i$ is uniform on $[1,1+1/k]$, discretized with $n = 10$. We implement the $\AMinO$ and $\SMinO$ oracle efficiently by using our above algorithm to exploit the rank-one structure of the cost tensor. In particular, the polynomial approximation we use here to approximate $\exp[-\eta C]$ is the degree-$5$ Taylor approximation (cf., Lemma~\ref{lem:lr-K}).
This lets us run $\Sink$ and $\MWU$ in polynomial time, as described above. In the numerical experiment, we also implement an approximate version of $\COLGEN$ using our polynomial-time implementation of the approximate violation oracle $\AMinO$.
Since the algorithms compute an upper bound, lower value is better in the right plot  of Figure~\ref{fig:riskestimation}. We observe that $\MWU$ yields the loosest approximation for this application, whereas our implementations of $\Sink$ and $\COLGEN$ produce high-quality approximations, as is evident by comparing to the exact LP solver in the regime that the latter is tractable to run.

\begin{figure}
\centering
\begin{tabular}{cc} \includegraphics[scale=0.5]{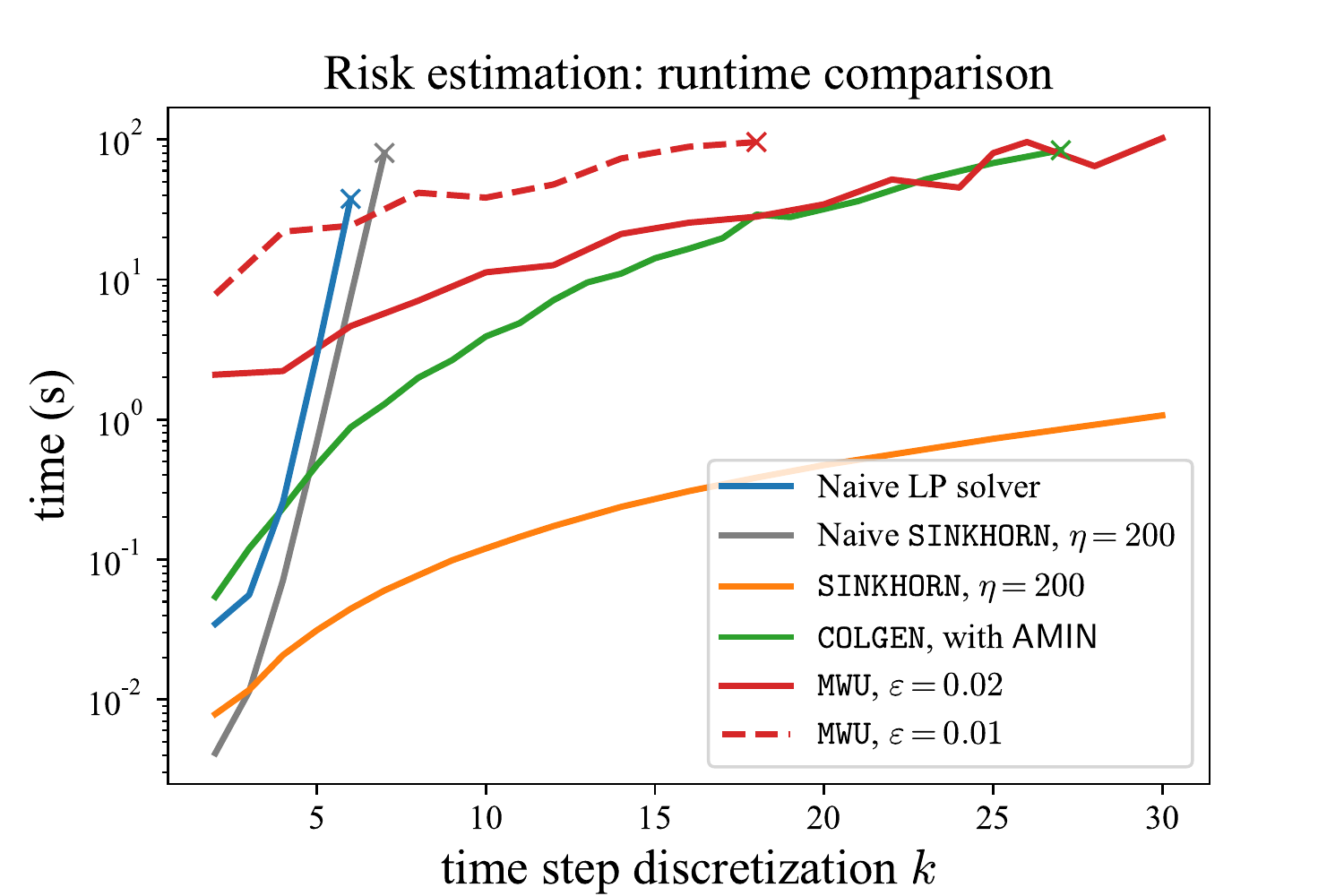} &
\includegraphics[scale=0.5]{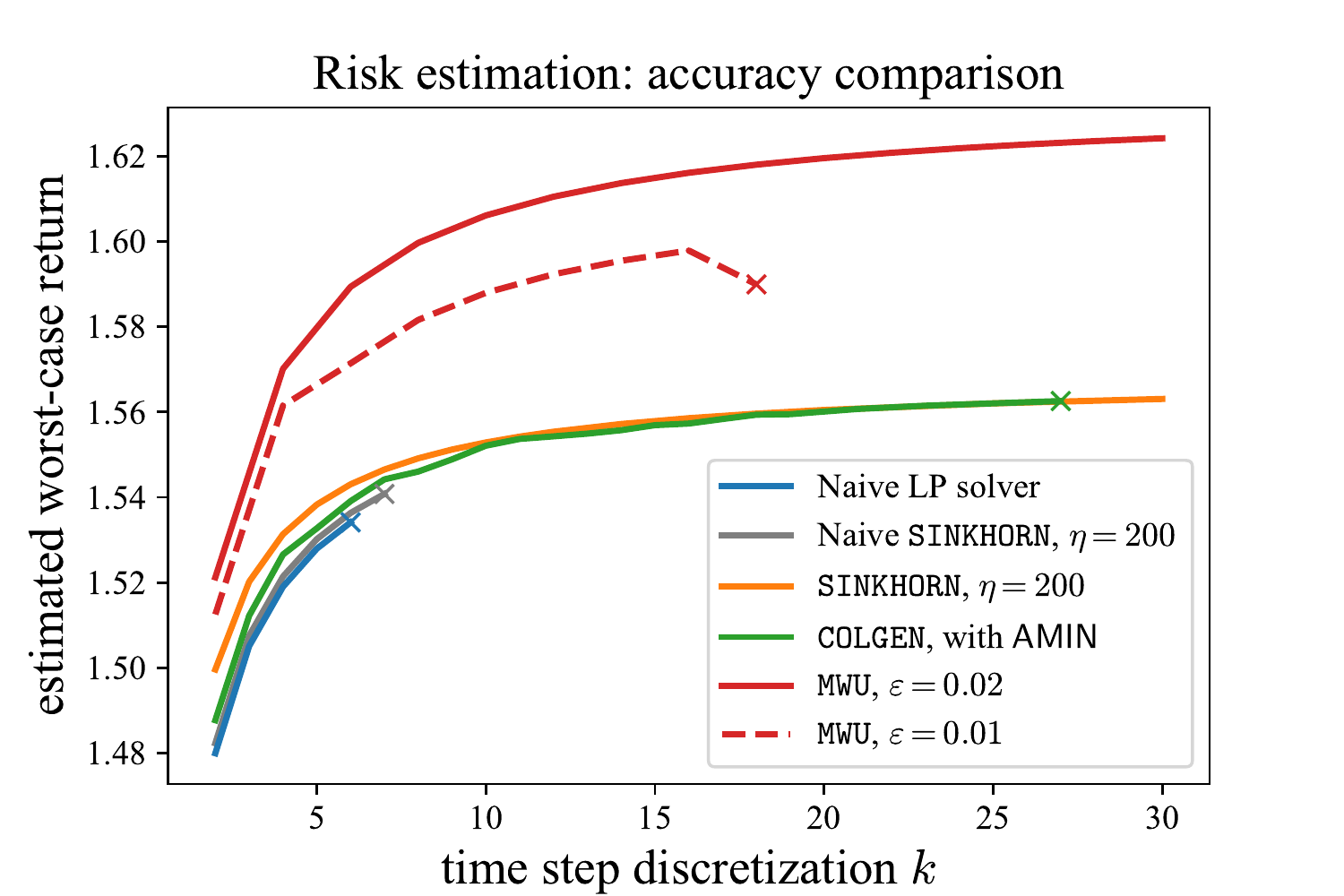}
\end{tabular}
\caption{
	Comparison of the runtime (left) and accuracy (right) of the fastest existing algorithms (naive LP solver and naive $\Sink$ which both have exponential runtimes that scale as $n^{\Omega(k)}$) with our algorithms ($\Sink$, $\MWU$, and $\COLGEN$ and $\MWU$ with polynomial-time implementations of their bottlenecks) for the risk estimation problem described in the main text. The algorithms are cut off at 2 minutes, denoted by an ``x''. Our new polynomial-time implementation of $\Sink$ returns high-quality solutions for problems that are orders-of-magnitude larger than previously possible: e.g., it takes less than a second to solve the problem for $k = 30$, which amounts to an $\MOT$ LP with $10^{30}$ variables.
}
\label{fig:riskestimation}
\end{figure}

\subsection{Application vignette: projection to the transportation polytope}\label{ssec:lr:proj}

Here we consider the fundamental problem of projecting a joint probability distribution $Q$ onto the transportation polytope $\Coup$, i.e.,
\begin{align}
	\argmin_{P \in \Coup} \sum_{\jvec} (P_{\jvec} - Q_{\jvec})^2.
	\label{eq:lr:proj}
\end{align}
We provide the first polynomial-time algorithm for solving this problem in the case where $Q$ is a distribution that decomposes into a low-rank component plus a sparse component. The low-rank component enables modeling mixtures of product distributions (e.g., mixtures of isotropic Gaussians), which arise frequently in statistics and machine learning; see, e.g.,~\citep{feldman2008learning}. In such applications, the number of product distributions in the mixture corresponds to the tensor rank. The sparse component further enables modeling arbitrary corruptions to the distribution in polynomially many entries. 

\par We emphasize that this projection problem~\eqref{eq:lr:proj} is \emph{not} an $\MOT$ problem since the objective is quadratic rather than linear. This illustrates the versatility of our algorithmic results. Our algorithm is based on a reduction from quadratic optimization to linear optimization over $\Coup$ that is tailored to this problem. Crucial to this reduction is the fact that the $\MOT$ algorithms in \S\ref{sec:algs} can compute \emph{sparse} solutions. In particular, this reduction does not work with $\Sink$ because $\Sink$ cannot compute sparse solutions.

\begin{cor}[Efficient projection to the transportation polytope]\label{cor:lr-proj}
	Let $Q = R + S\in \Rpntk$, where $R$ has constant rank and $S$ is polynomially sparse. Suppose that $R_{\max}$ and $S_{\max}$ are $O(1)$. Given $R$ in factored form, $S$ through its non-zero entries, measures $\mu_1, \dots, \mu_k \in \Delta_n$, and accuracy $\eps > 0$, we can compute in $\poly(n,k,1/\eps)$ time a feasible $P \in \Coup$ that has $\eps$-suboptimal cost for the projection problem~\eqref{eq:lr:proj}. This solution $P$ is a sparse tensor output through its $\poly(n,k,1/\eps)$ non-zero entries. 
\end{cor}

\begin{proof}
	We apply the Frank-Wolfe algorithm (a.k.a., Conditional Gradient Descent) to solve~\eqref{eq:lr:proj}, specifically using approximate LP solutions for the descent direction as in~\citep[Algorithm 2]{Jag13}. By the known convergence guarantee of this algorithm~\citep[Theorem 1.1]{Jag13},
	  if each LP is solved to $\eps' = O(\eps)$ accuracy, then $T = O(1/\eps)$ Frank-Wolfe iterations suffice to obtain an $\eps$-suboptimal solution to~\eqref{eq:lr:proj}. 
	\par The crux, therefore, is to show that each Frank-Wolfe iteration can be computed efficiently, and that the final solution is sparse. Initialize $P^{(0)}$ to be an arbitrary vertex of $\Coup$. Then $P^{(0)}$ is feasible and is polynomially sparse (see \S\ref{ssec:prelim:mot}).
	Let $P^{(t)} \in \Rpntk$ denote the $t$-th Frank-Wolfe iterate. Performing the next iteration requires two computations: 
	\begin{enumerate}
		\item Approximately solve the following LP to $\eps'$ accuracy:
		\begin{align}
			D^{(t)} \gets \min_{P \in \Coup} \langle P, P^{(t)} - Q \rangle.
			\label{eq:lr:proj-LP}
		\end{align}
		\item Update $P^{(t+1)} \gets (1 - \gamma_t)P^{(t)} + \gamma_t D^{(t)}$, where $\gamma_t = 2/(t+2)$ is the current stepsize. 
	\end{enumerate}
	For the first iteration $t = 0$, note that the LP~\eqref{eq:lr:proj-LP} is an $\MOT$ problem with cost $C^{(0)} = P^{(0)} - Q = P^{(0)} - R - S$ which decomposes into a polynomially sparse tensor $P^{(0)}-S$ plus a constant-rank tensor $-R$. Therefore the algorithm in Corollary~\ref{cor:lr-algs} can solve the LP~\eqref{eq:lr:proj-LP} to $\eps' = O(\eps)$ additive accuracy in $\poly(n,k,1/\eps)$ time, and it outputs a solution $D^{(0)}$ that is $\poly(n,k,1/\eps)$ sparse. It follows that $P^{(1)}$ can be computed in $\poly(n,k,1/\eps)$ time and moreover is $\poly(n,k,1/\eps)$ sparse since it is a convex combination of the similarly sparse tensors $P^{(0)}$ and $D^{(0)}$. 
	By repeating this argument identically for $T = O(1/\eps)$ iterations, it follows that each iteration takes $\poly(n,k,1/\eps)$ time, and that each iterate $P^{(t)}$ is $\poly(n,k,1/\eps)$ sparse. 
\end{proof}

\section{Discussion}\label{sec:discussion}

In this paper, we investigated what structure enables $\MOT$---an LP with $n^k$ variables---to be solved in $\poly(n,k)$ time. We developed a unified algorithmic framework for $\MOT$ by characterizing what ``structure'' is required to solve $\MOT$ in polynomial time by different algorithms in terms of simple variants of the dual feasibility oracle. On one hand, this enabled us to show that $\ELLIP$ and $\MWU$ solve $\MOT$ in polynomial time whenever any algorithm can, whereas $\Sink$ requires strictly more structure. And on the other hand, this made the design of polynomial-time algorithms for $\MOT$ much simpler, as we illustrated on three general classes of $\MOT$ cost structures.

\par Our results suggest several natural directions for future research. One exciting direction is to identify further tractable classes of $\MOT$ cost structures beyond the three studied in this paper, since this may enable new applications of $\MOT$. Our results help guide this search because they make it significantly easier to identify if an $\MOT$ problem is polynomial-time solvable (see \S\ref{sssec:intro-q4}).

\par Another important direction is practicality. While the focus of this paper is to characterize when $\MOT$ problems can be solved in polynomial time, in practice there is of course a difference between small and large polynomial runtimes. It is therefore a question of practical significance to improve our ``proof of concept'' polynomial-time algorithms by designing algorithms with smaller polynomial runtimes. Our theoretical results help guide this search for practical algorithms because they make it significantly easier to identify if an $\MOT$ problem is polynomial-time solvable in the first place.

\par In order to develop more practical algorithms, recall that, roughly speaking, our approach for designing $\MOT$ algorithms consisted of three parts:
\begin{itemize}
	\item An ``outer loop'' algorithm such as $\ELLIP$, $\MWU$, or $\Sink$ that solves $\MOT$ in polynomial time conditionally on a polynomial-time implementation of a certain bottleneck oracle.
	\item An ``intermediate'' algorithm that reduces this bottleneck oracle to polynomial calls of a variant of the dual feasibility oracle. 
	\item An ``inner loop'' algorithm that solves the relevant variant of the dual feasibility oracle for the structured $\MOT$ problem at hand.
\end{itemize} 
Obtaining a smaller polynomial runtime for any of these three parts immediately implies smaller polynomial runtimes for the overall $\MOT$ algorithm.
Another approach is to design altogether different algorithms that avoid the polynomial blow-up of the runtime that arises from composing these three parts. Understanding how to solve an $\MOT$ problem more ``directly'' in this way is an interesting question.

\begin{acknowledgements}
We are grateful to Jonathan Niles-Weed, Pablo Parrilo, and Philippe Rigollet for insightful conversations; to Frederic Koehler for suggesting a simpler proof of Lemma~\ref{lem:smin-separation}; to Ben Edelman and Siddhartha Jayanti who were involved in the brainstorming stages and provided helpful references; and to Karthik Natarajan for references to the random combinatorial optimization literature.
\end{acknowledgements}

\addtocontents{toc}{\protect\setcounter{tocdepth}{1}}
\appendix

\section{Deferred proof details}\label{app:pfs}

\subsection{Proof of Lemma~\ref{lem:mwu:step1}}\label{app:pf:mwu}

Our proof is based on two helper claims.
\begin{claim}[Lemma~3 of \citep{young2001sequential}]
If $\MWUoracle(P,\mu,\lambda,\eps)$ returns ``null'', then $\MOT_C(\mu) > \lambda$.
\end{claim}
\begin{proof}
We prove the contrapositive. Let $P^* \in \Coup$ such that $\langle C, P^* \rangle \leq \lambda$. Then 
\begin{align*}
\sum_{\jvec \in [n]^k} P^*_{\jvec} \pd{}{h} \Phi(P + h \delta_{\jvec}) \mid_{h = 0} &= \frac{\sum_{\jvec \in [n]^k} P^*_{\jvec}  (\frac{C_{\jvec}}{\lambda}\exp(\langle C, P \rangle / \lambda) + \sum_{i=1}^k \frac{1}{[\mu_i]_{j_i}}\exp([m_i(P)]_{j_i} / [\mu_i]_{j_i}))} {\exp(\langle C, P \rangle / \lambda) + \sum_{s=1}^k \sum_{t=1}^n \exp([m_s(P)]_{t} / [\mu_s]_t)}\\
&= \frac{\frac{\langle C, P^*\rangle}{\lambda}\exp(\langle C, P \rangle / \lambda) + \sum_{s=1}^k \sum_{t=1}^n \frac{[m_s(P^*)]_t}{[\mu_s]_{t}}\exp([m_s(P)]_{t} / [\mu_s]_{t})} {\exp(\langle C, P \rangle / \lambda) + \sum_{s=1}^k \sum_{t=1}^n \exp([m_s(P)]_{t} / [\mu_s]_t)} \\
&\leq \frac{\exp(\langle C, P \rangle / \lambda) + \sum_{s=1}^k \sum_{t=1}^n \exp([m_s(P)]_{t} / [\mu_s]_{t})} {\exp(\langle C, P \rangle / \lambda) + \sum_{s=1}^k \sum_{t=1}^n \exp([m_s(P)]_{t} / [\mu_s]_t)} \\
&= 1.
\end{align*}
Since $P^*$ is non-negative and $\sum_{\jvec \in [n]^k} P^*_{\jvec} = m(P^*) = 1$, there must exist a $\jvec \in [n]^k$ satisfying $\pd{}{h} \Phi(P + h \delta_{\jvec}) \mid_{h = 0} \leq 1$.
\end{proof}
This first claim ensures that if the algorithm returns ``infeasible'', then indeed $\MOT_C(\mu) > \lambda$. This proves the first part of the lemma. We now prove a second claim, useful for bounding the running time and the quality of the returned solution when the algorithm does not return ``infeasible''.
\begin{claim}[Lemmas~1 and 4 of \citep{young2001sequential}]
In Lines~\ref{line:step1first} to \ref{line:step1penultimate}, we maintain the invariant that
$\Phi(P) - (1 + \eps)^2 m(P) \leq \log(nk + 1)$.
\end{claim}
\begin{proof}
Note that initially $P = 0$ and $\Phi(0) - (1+\eps)^2 m(0) = \Phi(0) = \log(nk+1)$. So it suffices to prove that $\Phi(P) - (1+\eps)^2 m(P)$ does not increase on each iteration. Indeed, if $\jvec$ is returned by $\MWUoracle$, then $\pd{}{h} \Phi(P + h \delta_{\jvec}) \mid_{h=0} \leq (1+\eps)$. Furthermore, let $\eps' = \eps \cdot \min(\lambda / C_{\jvec}, \min_i [\mu_i]_{j_i})$. By the smoothness of the softmax, it holds that
$$\Phi(P + \eps' \delta_{\jvec}) \leq \Phi(P) + \eps' (1+\eps) \pd{}{h} \Phi(P + h\delta_{\jvec}) \mid_{h=0} \leq \eps' (1+\eps)^2.$$
So, in Line~\ref{line:step1penultimate}, $(1+\eps)^2m(P)$ increases by $(1+\eps)^2\eps'$, and $\Phi(P)$ increases by at most $(1+\eps)^2 \eps'$. This proves the claim.
\end{proof}

We use this second claim to bound the running time of Step 1. Each iteration of the loop from Line~\ref{line:step1first} to Line~\ref{line:step1penultimate} of Algorithm~\ref{alg:mwu}, increases the value of $$\Psi(P) := \langle C, P \rangle / \lambda + \sum_i \|m_i(P) / \mu_i\|_1$$ by at least $\eps$. So after $T$ iterations we must have $$\Phi(P) \geq \Psi(P) / (nk+1) \geq T\eps / (nk+1),$$ where we have used Jensen's inequality to relate $\Phi$ and $\Psi$. By the claim, this means that on any iteration $T \geq (\eta + \log(nk+1))(nk+1)(1+\eps)^2/\eps = \tilde{O}(nk / \eps^2)$, we must have $$m(P) \geq (\Phi(P) - \log(nk+1)) / (1 + \eps)^2 \geq \eta,$$ so the loop must terminate after at most $\tilde{O}(nk / \eps^2)$ iterations. Since each iteration can increase the number of non-zero entries by at most 1, this also proves the sparsity bound on $P$.

We now prove the bounds on the marginals and cost, again using the claim. When Line~\ref{line:step1last} is reached, we must have $m(P) \in [\eta, \eta + \eps]$, because each iteration increases $m(P)$ by at most $\eps$. Therefore, 
\begin{align*}
\max(\langle C, P \rangle / \lambda, m_1(P)/\mu_1,\ldots,m_k(P)/\mu_k) &\leq \Phi(P) &\mbox{by Lemma~\ref{lem:smin} for softmax} \\
&\leq \log(nk+1) + (1+\eps)^2 m(P) &\mbox{by the claim} \\
&\leq \log(nk + 1) + (1 + \eps)^2 (\eta + \eps) & \\
&\leq (1+\eps)^4 \eta &\mbox{by $\eta \geq 2\log(nk+1) \geq 1$}
\end{align*}
Therefore, the rescaling in Line~\ref{line:step1last} yields $P$ satisfying the guarantees of Lemma~\ref{lem:mwu:step1}.

\subsection{Proof of Lemma~\ref{lem:argaminequalsamino}}

It is obvious how the $\AMinO_C$ oracle can be implemented via a single call of the $\ArgAMinO_C$ oracle; we now show the converse. Specifically, given $p_1,\ldots,p_k \in \R^n$, we show how to compute a solution $\jvec = (j_1,\dots,j_k) \in [n]^k$ for $\ArgAMinO_C([p_1, \dots, p_k],\eps)$ using $nk$ calls to the $\AMinO_C$ oracle with accuracy $\eps/(2k)$. As in the proof of Lemma~\ref{lem:minargmin}, we use the first $n$ calls to compute the first index $j_1$ of the solution, the next $n$ calls to compute the next index $j_2$, and so on.
\par Formally, for $s \in [k]$, let us say that $(j_1^*, \dots,j_{s}^*) \in [n]^s$ is a $\delta$-approximate ``partial solution'' of size $s$ if there exists a solution $j \in [n]^k$ for $\ArgAMinO_C([p_1, \dots, p_k],\delta)$ that satisfies $j_i = j_i^*$ for all $i \in [s]$. Then it suffices to show that for every $s \in [k]$, it is possible to compute an $(s\eps/k)$-approximate partial solution $(j_1^*,\dots,j_{s}^*)$ of size $s$ from an $((s-1)\eps/k)$-approximate partial solution $(j_1^*,\dots,j_{s-1}^*)$ of size $s-1$ using $n$ calls to $\AMinO_C$ and polynomial additional time. 
\par Do this by setting $j_s^*$ to be a minimizer of 
\begin{align}
	\min_{j_s' \in [n]} \AMinO_C\Big(\Big[q_{1,j_1^*}, \ldots, q_{s-1,j_{s-1}^*}, q_{s,j_s'}, p_{s+1},\ldots, p_k\Big], \frac{\eps}{2k}\Big),
	\label{eq:lem-pf:aminargamin}
\end{align}
where the $q$ vectors are defined as in the proof of Lemma~\ref{lem:minargmin}. The runtime claim is obvious; it suffices to prove correctness. To this end, 
observe that
\begin{align*}
	\min_{\substack{\jvec \in [n]^k \\ \text{s.t. }j_1 = j_1^*, \dots, j_s = j_s^* }} C_{\jvec} - \sum_{i=1}^k [p_i]_{j_i}
	&=
	\MinO_C\Big(\Big[q_{1,j_1^*}, \dots,q_{s,j_s^*}, p_{s+1},\dots, p_k\Big]\Big)
	\\ &\leq
	\frac{\eps}{2k} + 
	\AMinO_C\Big(\Big[q_{1,j_1^*}, \dots,q_{s,j_s^*}, p_{s+1},\dots, p_k\Big], \frac{\eps}{2k}\Big)
	\\ &=
	\frac{\eps}{2k} + 
	\min_{j_s' \in [n]}
	\AMinO_C\Big(\Big[q_{1,j_1^*}, \dots, q_{s-1,j_{s-1}^*}, q_{s,j_s'}, p_{s+1},\dots, p_k\Big], \frac{\eps}{2k}\Big) 
	\\ &\leq
	\frac{\eps}{k} + 
	\min_{j_s' \in [n]}
	\MinO_C\Big(\Big[q_{1,j_1^*}, \dots, q_{s-1,j_{s-1}^*}, q_{s,j_s'}, p_{s+1},\dots, p_k\Big] \Big) 
	\\ &= 
	\frac{\eps}{k} + 
	\min_{j_s' \in [n]}
	\min_{\substack{\jvec \in [n]^k \\ \text{s.t. }j_1 = j_1^*, \dots j_{s-1} = j_{s-1}^*, j_s = j_s' }} C_{\jvec} - \sum_{i=1}^k [p_i]_{j_i}
	\\ &=
	\frac{\eps}{k} + 
	\min_{\substack{\jvec \in [n]^k \\ \text{s.t. }j_1 = j_1^*, \dots j_{s-1} = j_{s-1}^* }} C_{\jvec} - \sum_{i=1}^k [p_i]_{j_i} 
	\\ &=
	\frac{s\eps}{k} +	\MinO_{C}\Big(\Big[p_1,\dots,p_k\Big]\Big).
\end{align*}
Above, the first and fifth steps are by Observation~\ref{obs:minargmin}, the second and fourth steps are by definition of the $\AMinO$ oracle, the third step is by construction of $j_s^*$, the penultimate step is by simplifying, and the final step is by definition of $(j_1^*,\dots,j_{s-1}^*)$ being an $((s-1)\eps/k)$-approximate partial solution of size $s-1$. We conclude that $(j_1^*,\dots,j_s^*)$ is an $(s\eps/k)$-approximate partial solution of size $s$, as desired.

\section{Additional numerical experiments}\label{app:exp}

Here, we provide additional numerics for the generalized Euler flow application in \S\ref{ssec:graphical:fluid} in order to demonstrate that similar behavior is observed on other standard benchmark inputs in the literature~\citep{Bre08}. 
These instances are identical to Figure~\ref{fig:fluids}, except with different input permutations $\sigma$ between the initial and final positions of the particles.
Note that our algorithm $\COLGEN$ computes an exact, sparse solution with at most $nk-k+1$ non-zero entries (Theorem~\ref{thm:colgen}). In contrast, the $\Sink$ algorithm of~\citep{BenCarCut15} computes approximate, fully dense solutions with $n^k$ non-zero entries, which leads to blurry visualizations.

\begin{figure}[h]
	\centering
	\begin{tabular}{c@{}c@{\hskip 0.5in}c@{}c}
	{\renewcommand{\arraystretch}{1.7}\begin{tabular}{@{}c@{}} t = 1 \\ t = 2 \\ t = 3 \\ t = 4 \\ t = 5 \\ t = 6 \\ t = 7 \end{tabular}}& \begin{tabular}{@{}c@{}} \includegraphics[scale=0.25]{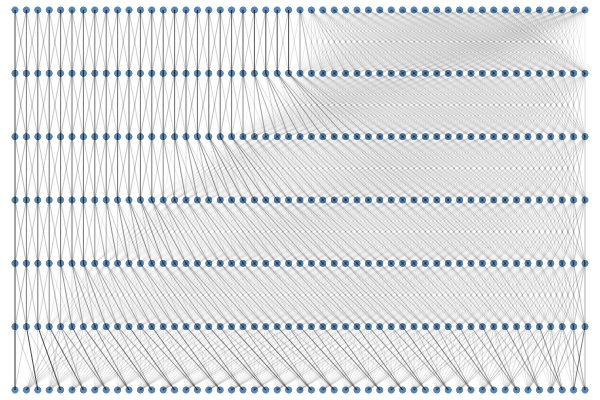} \end{tabular} & {\renewcommand{\arraystretch}{1.7}\begin{tabular}{@{}c@{}} t = 1 \\ t = 2 \\ t = 3 \\ t = 4 \\ t = 5 \\ t = 6 \\ t = 7 \end{tabular}} & \begin{tabular}{@{}c@{}}\includegraphics[scale=0.25]{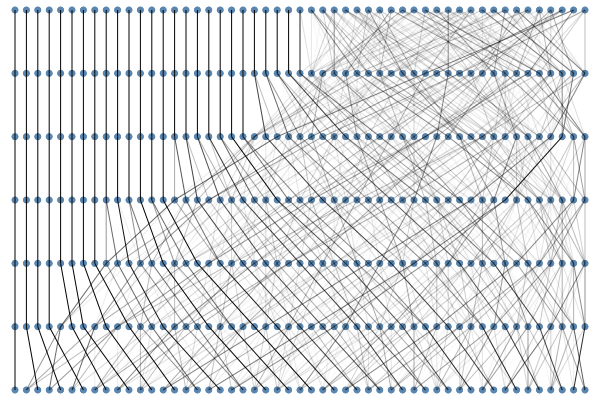}\end{tabular}
	\\
	& $\Sink$ & & $\COLGEN$ 
	\end{tabular}
	\caption{Same as Figure~\ref{fig:fluids}, but now with the permutation $\sigma$ that sends the particle at initial location $x \in [0,1]$ to final location $\sigma(x) = \min(2x, 2 - 2x)$. $\COLGEN$ runs in 7.88 seconds, while $\Sink$ with regularization $\eta = 2000$ runs in 6.97 seconds. In the $\COLGEN$ solution, roughly half the particles have trajectories that never split.}
\label{fig:fluids:diagram2}
\end{figure}

\begin{figure}[h]
	\centering
	\begin{tabular}{c@{}c@{\hskip 0.5in}c@{}c}
	{\renewcommand{\arraystretch}{1.7}\begin{tabular}{@{}c@{}} t = 1 \\ t = 2 \\ t = 3 \\ t = 4 \\ t = 5 \\ t = 6 \\ t = 7 \end{tabular}}& \begin{tabular}{@{}c@{}} \includegraphics[scale=0.25]{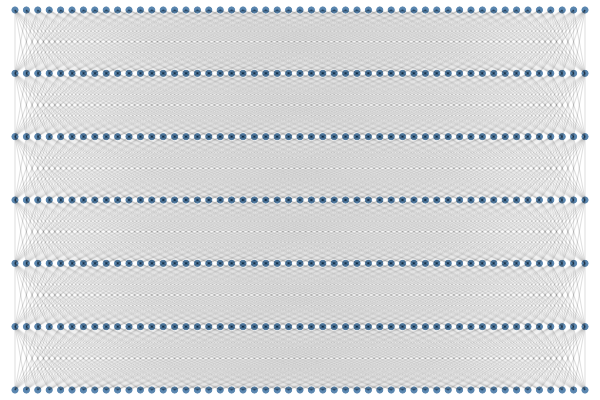} \end{tabular} & {\renewcommand{\arraystretch}{1.7}\begin{tabular}{@{}c@{}} t = 1 \\ t = 2 \\ t = 3 \\ t = 4 \\ t = 5 \\ t = 6 \\ t = 7 \end{tabular}} & \begin{tabular}{@{}c@{}}\includegraphics[scale=0.25]{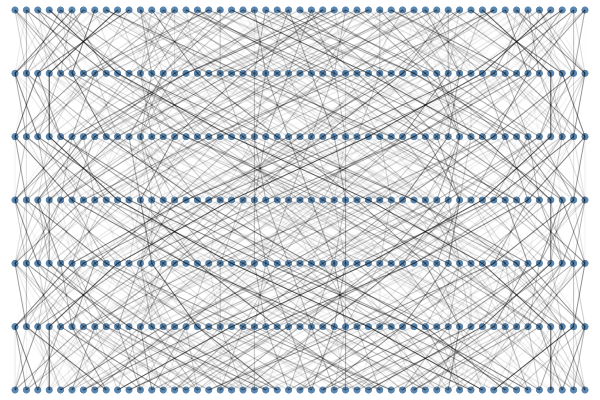}\end{tabular}
	\\
	& $\Sink$ & & $\COLGEN$ 
	\end{tabular}
	\caption{Same as Figure~\ref{fig:fluids}, but now with the permutation $\sigma$ that sends the particle at initial location $x \in [0,1]$ to final location $\sigma(x) = 1 - x$. $\COLGEN$ runs in 10.53 seconds, while $\Sink$ with regularization $\eta = 1500$ runs in 2.10 seconds.
	}\label{fig:fluids:diagram3}
\end{figure}

\footnotesize
\addcontentsline{toc}{section}{References}
\bibliographystyle{abbrv}
\bibliography{mot}{}

\end{document}